\documentclass[fontsize=11pt,twoside=false,draft=false]{scrartcl}
\usepackage{hyperref}
\hypersetup{plainpages=false,
            linkcolor=black,
            citecolor=black,
            urlcolor=black,
            colorlinks=true}
            
\usepackage[utf8]{inputenc}

\usepackage[automark]{scrlayer-scrpage}

\usepackage[english]{babel}

\usepackage[intlimits]{amsmath}
\usepackage{amsfonts}
\usepackage{amstext}
\usepackage{amsopn}
\usepackage{mathabx}
\usepackage{url}

\usepackage{tikz}
\usetikzlibrary{calc}
\usetikzlibrary{patterns}
\tikzset{
  nw/.style={pattern=north west lines, line width=0.1pt, pattern color=#1},
  nw/.default=black
}
\tikzset{
  ne/.style={pattern=north east lines,line width=0.1pt, pattern color=#1},
  ne/.default=black
}

\usepackage{longtable}
\usepackage{graphicx}
\usepackage{caption}
\usepackage{subfig}
\usepackage{wrapfig}
\usepackage{enumitem}

\tikzstyle{every picture}+=[font=\footnotesize]

\usepackage{amsmath,amssymb,mathtools,commath}
\usepackage{cases}
\usepackage[amsmath,hyperref,thmmarks]{ntheorem}
\numberwithin{equation}{section}
\numberwithin{figure}{section}
\numberwithin{table}{section}

\makeatletter
\theoremheaderfont{\bfseries\upshape}
\theoremseparator{.}
\newtheorem{theorem}{Theorem}[section]
\newtheorem{lemma}[theorem]{Lemma}

\newtheorem{corollary}[theorem]{Corollary}

\theoremsymbol{\ensuremath{\Box}}

\theoremsymbol{}
\theorembodyfont{\upshape}
\newtheorem{definition}[theorem]{Definition}

\theoremseparator{}

\theoremstyle{break}

\theoremstyle{plain}
\theoremheaderfont{\itshape}
\newtheorem{remark}[theorem]{Remark}

\theoremstyle{plain}
\theoremsymbol{\ensuremath{\Box}}
\theoremseparator{.}
\newtheorem*{proof}{Proof}
\theoremseparator{}
\newtheorem*{proofof}{Proof of}

\usepackage{algorithm}
\usepackage[noend]{algpseudocode}

\algrenewcommand\algorithmicrequire{\textbf{\ Input:}}
\algrenewcommand\algorithmicensure{\textbf{\ Output:}}
\floatname{algorithm}{Algorithm}

\usepackage{dirtree}
\usepackage{zref}

\usepackage{booktabs}

\let\oldbullet\bullet
  \newlength{\raisebulletlen}
  \setbox1=\hbox{$\bullet$}\setbox2=\hbox{\tiny$\bullet$}
  \renewcommand\bullet{\raisebox{\raisebulletlen}{\,\tiny$\oldbullet$}\,}
\pgfdeclarepatterninherentlycolored[\hatchColor]{myLines}{\pgfqpoint{-1pt}{-1pt}}{\pgfqpoint{4pt}{4pt}}{\pgfqpoint{3pt}{3pt}}%
{
    \pgfsetfillcolor{\hatchColor!10}
    \pgfpathrectangle{\pgfqpoint{0pt}{0pt}}{\pgfqpoint{3.1pt}{3.1pt}}
    \pgfusepath{fill}
    \pgfsetlinewidth{0.8pt}
    \pgfsetcolor{\hatchColor}
    \pgfpathmoveto{\pgfqpoint{-0.4pt}{-0.4pt}}
    \pgfpathlineto{\pgfqpoint{3.4pt}{3.4pt}}
    \pgfusepath{stroke}
}

\tikzset{
    hatchColor/.store in=\hatchColor, hatchColor=gray
}
\tikzset{slopetriangle/.style={
  bottom color=black!20,
  middle color=black!5,
  top color=white,
  draw=black
}}

\def\Xint#1{\mathchoice
{\XXint\displaystyle\textstyle{#1}}%
{\XXint\textstyle\scriptstyle{#1}}%
{\XXint\scriptstyle\scriptscriptstyle{#1}}%
{\XXint\scriptscriptstyle\scriptscriptstyle{#1}}%
\!\int}
\def\XXint#1#2#3{{\setbox0=\hbox{$#1{#2#3}{\int}$}
\vcenter{\hbox{$#2#3$}}\kern-.5\wd0}}
\newcommand{\intmean}{\Xint-}

\definecolor{dkgreen}{rgb}{0,0.6,0}
\definecolor{gray}{rgb}{0.5,0.5,0.5}
\definecolor{mauve}{rgb}{0.58,0,0.82}
\usepackage{listings}
\lstdefinestyle{fullsrcsmall}{
  basicstyle=\small\ttfamily,
  xleftmargin=1em,
}
\lstdefinestyle{fullsrcfnsize}{
  basicstyle=\footnotesize\ttfamily,
  xleftmargin=2em,
}
\lstdefinestyle{fullsrcscsize}{
  basicstyle=\scriptsize\ttfamily,
  numbers=none,
  xleftmargin=2em,
}
\lstdefinestyle{inline}{
  basicstyle=\ttfamily,
  numbers=none,
  xleftmargin=0em,
}
\lstdefinestyle{inlinesmall}{
  basicstyle=\small\ttfamily,
  numbers=none,
  xleftmargin=0em,
}
\lstdefinestyle{inlinefnsize}{
  basicstyle=\footnotesize\ttfamily,
  numbers=none,
  xleftmargin=1em,
}

\usepackage{a4wide}
\usepackage[capitalise, english, nameinlink]{cleveref}
\crefname{subsection}{Subsection}{Subsections}
\crefname{subsubsection}{Subsection}{Subsections}
\crefname{AFEM4EVP}{AFEM4EVP}{AFEM4EVP}
\crefname{AFEM}{AFEM}{AFEM}
\usepackage{pgfplots}
\usepackage{placeins}
\usepackage{enumitem}
\makeatletter
\newcommand{\mylabel}[2]{#2\def\@currentlabel{#2}\label{#1}}
 \DeclareCaptionLabelFormat{noname}{#2}
\DeclareCaptionLabelFormat{algnonumber}{AFEM4EVP}
\makeatother
\RedeclareSectionCommand[
  beforeskip=.5\baselineskip,
  afterskip=-2em]{paragraph}  
  \RedeclareSectionCommand[
  beforeskip=.4\baselineskip,
  afterskip=.4\baselineskip]{section}    
\RedeclareSectionCommand[
  beforeskip=.3\baselineskip,
  afterskip=.2\baselineskip]{subsection}    
  \RedeclareSectionCommand[
  beforeskip=.3\baselineskip,
  afterskip=.2\baselineskip]{subsubsection}    
\usepackage{caption}
\allowdisplaybreaks
\usepackage{textcmds}
\usepackage{xr}
\title{Adaptive guaranteed lower eigenvalue bounds with optimal convergence rates
}
\author{
Carsten Carstensen\footnote{Department of Mathematics, Humboldt-Universit\"at zu Berlin, Unter den Linden 6, 10099 Berlin, Germany. \href{mailto:cc@math.hu-berlin.de}{cc@math.hu-berlin.de} 
and \href{mailto:puttkams@math.hu-berlin.de}{puttkams@math.hu-berlin.de}}
\and
Sophie Puttkammer$^*$
}
\date{}

\newcommand{\SPnew}[1]{\textcolor{black}{#1}}
\begin{document}

\maketitle
\begin{abstract}\vspace{-1cm}
Guaranteed lower Dirichlet eigenvalue bounds (GLB) can be computed for the $m$-th Laplace operator with a recently introduced 
extra-stabilized nonconforming Crouzeix-Raviart ($m=1$) or Morley ($m=2$)
finite element eigensolver. 
Striking numerical evidence for the superiority of a new adaptive eigensolver
motivates the convergence analysis in this paper with a proof of optimal convergence rates 
of the GLB  towards a simple eigenvalue. The proof is based on (a generalization of) known abstract 
arguments entitled as the axioms of adaptivity. Beyond the known a priori 
convergence rates, a medius analysis is enfolded in this paper for the proof of
best-approximation results. This and subordinated $L^2$ error estimates for locally refined triangulations appear of independent interest. 
The analysis of optimal convergence rates of an adaptive mesh-refining algorithm is performed in $3$D and highlights a new version of discrete reliability. 
\end{abstract}
\section{Introduction}\label{sec:introduction}
	\textbf{Motivation.} 
		Guaranteed lower Dirichlet eigenvalue bounds (GLB) can be computed for the $m$-th Laplace operator from a global postprocessing of 
		respective nonconforming finite element eigensolvers 
		like the Crouzeix-Raviart resp.\ Morley finite element method (FEM) for $m=1$ resp.\ $m=2$ \cite{CGal14,CGed14}.  
		The maximal mesh-size $h_{\max}$ enters as an explicit parameter and this can be non-effective for an imperative adaptive mesh-refinement. 
		This has recently motivated the design of extra-stabilized nonconforming finite element eigensolvers for $m=1,2$ that directly compute GLB 
		under moderate mesh-size restrictions and allow an efficacious adaptive mesh-refinement \cite{CZZ18,CEP20,CP_Part1}. 
		The striking superiority of those adaptive schemes has been displayed in numerical experiments in \cite{CEP20,CP_Part1} 
		and motivates the mathematical analysis of optimal convergence rates in this paper. This appears to be the first method that combines the localization of eigenvalues as GLB with their efficient approximation. 
	\\[1ex]\textbf{Model problem.} 
		The continuous eigenvalue problem (EVP) seeks eigenpairs $(\lambda, u) \in \mathbb{R}^+\times (V\setminus\{0\})$ with
		 	\begin{align}
			 	a(u,v)= \lambda\, b(u,v)\quad \text{for all } v\in V \label{eq:contEVP}
			 \end{align}
		 in the Hilbert space $V:=H^{m}_0(\Omega)$ and its energy scalar product $a(\bullet,\bullet):=(D^m\bullet, D^m\bullet)_{L^2(\Omega)}$ 
		 with the gradient $D^1:=\nabla$ or the Hessian $D^2$ and the $L^2$ scalar product 
		 $b(\bullet,\bullet):=(\bullet, \bullet)_{L^2(\Omega)}$ on a bounded polyhedral Lipschitz domain $\Omega\subset\mathbb{R}^3$.
		 The infinite but countably many eigenvalues $0<\lambda_1\le\lambda_2\le\dots$ with $\lim_{j\to \infty}\lambda_j=\infty$ in \eqref{eq:contEVP} 
		 are enumerated in ascending order counting multiplicities \cite{BO91,Boffi2010}.
	\\[1ex]\textbf{Discretization.} 
		 The discrete space $\boldsymbol{V_h}\hspace{-.2em}= \hspace{-.2em}P_{m}(\mathcal{T})\hspace{-.1em}\times \hspace{-.1em}V({\mathcal{T}})
		 \hspace{-.2em}\subset\hspace{-.2em}  P_{m}(\mathcal{T})\hspace{-.1em}\times\hspace{-.1em} P_{m}(\mathcal{T})$ 
		 consists of piecewise polynomials of degree at most ${m}$ on the shape-regular triangulation $\mathcal{T}$ of 
		 $\Omega\subset\mathbb{R}^3$ into closed 
		 tetrahedra.  
		 Throughout this paper, $V(\mathcal{T})$ abbreviates the Crouzeix-Raviart finite element space $\textit{CR}^1_0(\mathcal{T})$ 
		 \cite{CR73} for $m=1$ and the Morley finite element space $M(\mathcal{T})$ \cite{Mor68,MX06} for $m=2$. 
		 The algebraic eigenvalue problem seeks eigenpairs 
		 $(\lambda_h, \boldsymbol{u_h}) \in \mathbb{R}^+\times (\boldsymbol{V_h}\setminus\{0\})$ with 
			\begin{align}
					\boldsymbol{a_h}(\boldsymbol{u_h},\boldsymbol{v_h})
						=\lambda_h \boldsymbol{b_h}(\boldsymbol{u_h},\boldsymbol{v_h})\quad\text{for all }\boldsymbol{v_h}\in\boldsymbol{V_h}. 			 	
						\label{eq:dis_EVP_alt}
			\end{align} 
		The discrete scalar product $\boldsymbol{a_h}$ contains the  scalar product 
		$a_{\mathrm{pw}}(\bullet, \bullet):=(D^m_\mathrm{pw}\bullet,D^m_\mathrm{pw}\bullet)_{L^2(\Omega)}$ of the piecewise derivatives of 
		order $m$ and some stabilization with explicit (known) constant $\kappa_m>0$ from \cite{CP_Part1}, while the bilinear form 
		$\boldsymbol{b_h}$ is the $L^2$ scalar product $b(\bullet,\bullet)$ of the piecewise polynomial components, 
			\begin{align*}
					\boldsymbol{a_h}(\boldsymbol{v_h},\boldsymbol{w_h})&=a_{\mathrm{pw}}(v_{\mathrm{nc}},w_{\mathrm{nc}})
						+\kappa_{m}^{-2}(h_{\mathcal{T}}^{-2{m}}(v_{\mathrm{pw}}-v_{\mathrm{nc}}), w_{\mathrm{pw}}-w_{\mathrm{nc}})_{L^2(\Omega)},\\
					\boldsymbol{b_h}(\boldsymbol{v_h},\boldsymbol{w_h})&=b(v_{\mathrm{pw}},w_{\mathrm{pw}})
					\qquad\text{for all }\boldsymbol{v_h}=(v_{\mathrm{pw}},v_{\mathrm{nc}}),\, 
					\boldsymbol{w_h}=(w_{\mathrm{pw}},w_{\mathrm{nc}})\in\boldsymbol{V_h}.
				\end{align*} 
		The piecewise constant mesh-size function $h_{\mathcal{T}}\in P_0(\mathcal{T})$ has the value 
		$h_{\mathcal{T}}|_T=h_T:=\textup{diam}(T)$ in each tetrahedron $T\in\mathcal{T}$ and $h_{\max}:=\max_{T\in\mathcal{T}}h_T$ denotes the maximal mesh-size. 
		The $M:=\textup{dim}(P_m(\mathcal{T}))$ finite discrete eigenvalues of \eqref{eq:dis_EVP_alt} are enumerated in ascending order    
		$0< \lambda_h(1)\le \lambda_h(2)\le\dots\le \lambda_h(M)<\infty$ counting multiplicity. 
	\\[1ex]\textbf{GLB.}
		For the biharmonic operator ($m=2$) the discrete eigenvalue problem \eqref{eq:dis_EVP_alt} is analysed in \cite{CP_Part1}.
		For the Laplace operator ($m=1$) in $2$D, \eqref{eq:dis_EVP_alt} describes the lowest-order skeleton method in \cite{CZZ18}; 
		for $3$D it is different and suggested in \cite{CP_Part1}. 
		The discrete eigenvalue problem \eqref{eq:dis_EVP_alt} directly computes  guaranteed lower bounds \cite[Thm.~1.1]{CP_Part1}
		in that  
			\begin{align}
				\min\{\lambda_h({k}),\lambda_k\} \kappa^2_m h_{\max}^{2{m}}\le 1 \quad\text{ implies } \quad\lambda_h({k})\le \lambda_{k}
				\quad\text{ for all }k=1,\dots, M. 			
					\label{eq:GLB_cond}
			\end{align} 
	\textbf{AFEM.} 	
		The adaptive algorithm 
		\cite{Doerfler1996,MorinNochettoSiebert2002,CFPP14,CR16} is based on the refinement indicator $\eta(T)$ defined in 
		\eqref{eq:def_eta} below for any triangulation $\mathcal{T}$ and any tetrahedron $T\in\mathcal{T}$.
		Let  $\big(\lambda_h, \boldsymbol{u_h}\big)\in \mathbb{R}^+\times \boldsymbol{V_h}$ denote the $k$-th eigenpair of 
		\eqref{eq:dis_EVP_alt} with  $\lambda_h:=\lambda_h(k)$ and 	
		$\boldsymbol{u_h}=(u_{\mathrm{pw}},u_{\mathrm{nc}})\in  \boldsymbol{V_h}$.
		For any tetrahedron  $T\in\mathcal{T}$ 
		with volume $|T|$ and set of faces $\mathcal{F}(T)$, the local estimator contribution $\eta^2(T)=(\eta(T))^2$ reads
			\begin{align}
				\eta^2(T) = |T|^{2{m}/3}\Vert \lambda_h u_{\mathrm{nc}} \Vert^2_{L^2(T)}
												+|T|^{1/3}\sum_{F\in\mathcal{F}(T)}\Vert [{D}^{m}_{\mathrm{pw}}  u_{\mathrm{nc}}]_F\times \nu_F\Vert^2_{L^2(F)} 		
												\label{eq:def_eta}
			\end{align}
		with the tangential components $[D^m_{\mathrm{pw}} u_{\mathrm{nc}}]_F\times \nu_F$ of the jump $[D^m_{\mathrm{pw}}  u_{\mathrm{nc}}]_F$ along any face 
		$F\in\mathcal{F}(T)$ and  the (piecewise) gradient $D^1_{\mathrm{pw}}\hspace{-.2em}=\hspace{-.2em}\nabla_{\mathrm{pw}}$ (${m}\hspace{-.2em}=\hspace{-.2em}1$) or Hessian $D^2_{\mathrm{pw}}$ (${m\hspace{-.2em}}=\hspace{-.2em}2$). 
		Let $\mathbb{T}:=\mathbb{T}(\mathcal{T}_0)$ denote the set of all admissible regular triangulations 
		computed by successive newest-vertex bisection (NVB) \cite{Stev08,GSS14} 
		of a regular initial triangulation $\mathcal{T}_0$  of $\Omega\subset\mathbb{R}^3$. 
		The AFEM algorithm  with  D\"orfler marking and newest-vertex bisection 
			abbreviates  $\eta_\ell(T)$ for any $T\in\mathcal{T}:=\mathcal{T}_\ell\in\mathbb{T}$  and 
			$\eta_\ell^2:=\eta^2(\mathcal{T}_\ell):=\sum_{T\in\mathcal{T}_\ell}\eta_\ell^2(T)$. 
			The selection of the set $\mathcal{M}_\ell$ in the step Mark of \namecref{alg:AFEM4EVP}   
			with \emph{minimal cardinality} is possible at linear cost \cite{PfeilerPraetorius2020}. 
			
				\begin{algorithm}[H]							
					\caption{}\label[AFEM4EVP]{alg:AFEM4EVP}
					\begin{algorithmic}
					\Require regular triangulation $\mathcal{T}_0$ and parameters $0<\theta\leq 1$ and ${k}\in\mathbb{N}$ 
					\For{$\ell = 0,1,2,\dots$}
					\State\hspace{-1em} \textbf{Solve} the discrete problem \eqref{eq:dis_EVP_alt} exactly and compute the $k$-th algebraic eigenpair   
					\\\hspace{1.5cm}$(\lambda_\ell({k}), \boldsymbol{u_\ell}({k}))$ with 
					$\boldsymbol{u_\ell}({k})=(u_{\mathrm{pw}},u_{\mathrm{nc}})\in P_{m}(\mathcal{T}_\ell)\times V({\mathcal{T}}_\ell)$ 
					and $\mathcal{T}$ replaced by $\mathcal{T}_\ell$
					\State\hspace{-1em} \textbf{Compute} $\eta_\ell(T)$ for any $T\in\mathcal{T}_\ell$  from \eqref{eq:def_eta} 
					with $(\lambda_h,u_{\mathrm{nc}},\mathcal{T})$ 
					replaced by $(\lambda_\ell(k),u_{\mathrm{nc}},\mathcal{T}_\ell)$ 
					\State\hspace{-1em} \textbf{Mark} minimal subset $\mathcal{M}_\ell\subseteq\mathcal{T}_\ell$ with 
					$\theta\eta_\ell^2 \leq \sum_{T\in\mathcal{M}_\ell}\eta_\ell^2(T)$
					\State\hspace{-1em} \textbf{Refine} $\mathcal{T}_\ell$ with newest-vertex bisection to 
					compute $\mathcal{T}_{\ell+1}$ with $\mathcal{M}_\ell\subseteq \mathcal{T}_\ell \setminus \mathcal{T}_{\ell+1}$ \textbf{od}
					\EndFor
					\Ensure sequence of triangulations $(\mathcal{T}_\ell)_{\ell\in\mathbb{N}_0}$ with 
					$(\lambda_\ell({k}), \boldsymbol{u_\ell}({k}))_{\ell\in\mathbb{N}_0}$ and $(\eta_\ell)_{\ell\in\mathbb{N}_0}$
					\end{algorithmic}
					\end{algorithm}
		\textbf{Optimal convergence rates.}
		The optimal convergence rates of \namecref{alg:AFEM4EVP} in 
		the error estimator 
		means that the outputs \((\mathcal{T}_\ell)_{\ell\in\mathbb{N}_0}\)  and \((\eta_\ell)_{\ell\in\mathbb{N}_0}\) 
		of	\namecref{alg:AFEM4EVP} satisfy  
			\begin{equation}
		 		\sup_{\ell\in\mathbb{N}_0} (1+\abs{\mathcal{T}_\ell} - \abs{\mathcal{T}_0})^s \eta_\ell 
		 		\approx  \sup_{N\in\mathbb{N}_0} (1+N)^s \min\{\eta(\mathcal{T}):\, \mathcal{T}\in\mathbb{T} 
		 		\text{ with } |\mathcal{T}|\le |\mathcal{T}_0|+ N\}
		 		\label{eq:Optimality}
			\end{equation}
		 for any \(s>0\) and the counting measure \(\vert\bullet\vert=\textup{card}(\bullet)\). 
		In other words, if the estimator $\eta(\mathcal{T})$ 
		converges with rate $s>0$ for some optimal selection of triangulations \(\mathcal{T}\in\mathbb{T}\), 
		then the output $\eta_\ell$ of \namecref{alg:AFEM4EVP} converges with the same rate. 
		
		\begin{theorem}[rate optimality of \namecref{alg:AFEM4EVP}]\label{thm:Optimality4GLB}
			Suppose that $\lambda_k=\lambda$ is a simple eigenvalue of \eqref{eq:contEVP}, then
			there exist $\varepsilon>0$ and $0<\theta_0<1$ such that 
			$\mathcal{T}_0\in \mathbb{T}(\varepsilon):=\{\mathcal{T}\in\mathbb{T}:\,h_{\max}:=\max_{T\in\mathcal{T}}h_T\le \varepsilon\}$ and 
			 $\theta$ with 
			$0<\theta \le \theta_0$ imply  \eqref{eq:Optimality} 
			for any $s>0$. 
		\end{theorem}
		
At first glance the discrete problem \eqref{eq:dis_EVP_alt} involves a stabilization that is expected to generate the additional 
term $\kappa_m^{-2}|T|^{-2m/3}\Vert u_{\mathrm{pw}}-u_{\mathrm{nc}}\Vert^2_{L^2(T)}$
in  the error estimator \eqref{eq:def_eta}.  The  negative power of the mesh-size in the latter term prevents
a reduction property \cite{MorinNochettoSiebert2002,CFPP14,CR16} and has to be circumvented. 
The only other known affirmative result for optimal convergence rates of an adaptive algorithm with stabilization (and negative powers of the mesh-size in the discrete problem) is \cite{BonitoNochetto2010} on discontinuous Galerkin (dG) schemes. An over-penalization therein diminishes the influence of the stabilization and eventually shows the dominance of the remaining a posteriori error terms. In the present case, the stabilization parameter $\kappa_m$ is fixed to maintain the GLB property and this requires a different argument: Since \eqref{eq:dis_EVP_alt} is equivalent to a rational eigenvalue problem for a nonconforming scheme, a careful perturbation analysis eventually shows efficiency and reliability of the nonconforming error estimator  \eqref{eq:def_eta} for sufficiently small mesh-sizes. The verification requires a medius analysis  \cite{Gudi10}, which applies  arguments from a~posteriori 
 error analysis  (e.g., efficiency in  \eqref{eq:ControlOfDeltaUh} below) in an a~priori error analysis. 

\medskip

		\textbf{Outline.}
		The remaining parts of this paper are devoted to the proof of \cref{thm:Optimality4GLB} and are organized as follows. 
		A general interpolation operator $I$ and a right-inverse $J$ in \cref{sec:preliminaries} allow for a simultaneous analysis for 
		$m=1$ and $m=2$ in the Crouzeix-Raviart and Morley FEM. 
		The  medius analysis in \cref{sec:apriori}  provides new best-approximation results and thereby prepares the proof of \cref{thm:Optimality4GLB}  in 
		\cref{sec:aposteriori}--\ref{sec:optimality4evp}. 
		The proof of the optimal convergence rates requires a  framework extended from  
		\cite{CFPP14,CR16} in \cref{sec:Appendix}. 
		{\medskip\\
		The results hold in $2$D and $3$D and are presented in $3$D for brevity. }
\section{Preliminaries}\label{sec:preliminaries}
			This section summarizes abstract conditions  \ref{item:I_kappa}--\ref{item:I_kappa_disc} 
			on an interpolation operator $I:V\to V(\mathcal{T})$ and \ref{item:J_rightInverse}--\ref{item:J_T_That} 
			on a right inverse $J:V(\mathcal{T})\to V$. The conditions hold for the Crouzeix-Raviart and the Morley finite element space in 
			the two model  examples for the Laplacian ${m}=1$  and the bi-Laplacian ${m}=2$.
		\subsection{Notation}\label{sec:notations}
			Standard notation on Lebesgue and Sobolev spaces applies throughout this paper; $(\bullet,\bullet)_{L^2(\Omega)}$ abbreviates
			the $L^2$ scalar product and $H^{m}(T)$ 
			abbreviates $H^{m}(\textup{int}(T ))$ for a tetrahedron $T\in\mathcal{T}$.
			The vector space $H^{m}(\mathcal{T}):=\{v\in L^2(\Omega):\, v|_T\in H^{m}(T)\}$ consists of  
			piecewise $H^{m}$ functions and is equipped with the 
			semi-norm 
			$\vvvert\bullet\vvvert_{\mathrm{pw}}^2
			:=(D^m_{\mathrm{pw}}\bullet,D^m_{\mathrm{pw}}\bullet)_{L^2(\Omega)}$. The piecewise gradient $D^1_{\mathrm{pw}}$ 
			or piecewise Hessian $D^2_{\mathrm{pw}}$ 
			is understood with respect to 
			the (non-displayed) regular triangulation $\mathcal{T}\in\mathbb{T}$ of the bounded polyhedral Lipschitz domain  
			$\Omega\subset\mathbb{R}^3$ into tetrahedra. 	
			The triangulation $\mathcal{T}$ is computed by successive newest-vertex bisection (NVB) \cite{Stev08,GSS14} 
			of a regular initial triangulation $\mathcal{T}_0$ (plus some initialization of tagged tetrahedra) 
			of $\Omega\subset\mathbb{R}^3$.
			The set  $\mathbb{T}:=\mathbb{T}(\mathcal{T}_0)$ of all admissible triangulations is (uniformly) shape-regular. 
			For any $\mathcal{T}\in\mathbb{T}$, let $\mathbb{T}(\mathcal{T})$ abbreviate the set of all admissible refinements of 
			$\mathcal{T}$. 
			 For any $0<\varepsilon<1$ let 
			 $\mathbb{T}(\varepsilon):=\{\mathcal{T}\in\mathbb{T}:\,h_{\max}:=\max_{T\in\mathcal{T}}h_T\le \varepsilon\}$ 
			denote the set of all admissible triangulations with maximal mesh-size $h_{\max}\le\varepsilon$.	
			The context-depending notation $|\bullet|$ denotes the Euclidean length of a vector, the cardinality of 
			a finite set, as well as the non-trivial \mbox{three-,} two-, or one-dimensional Lebesgue measure of a subset of  
			$\mathbb{R}^3$. 
			For any  positive, piecewise polynomial $\varrho\in P_k(\mathcal{T})$ with $\varrho\ge 0$, 
			$k\in\mathbb{N}_0$, 
			$(\bullet, \bullet)_\varrho:=(\varrho\bullet,\bullet)_{L^2(\Omega)}$  abbreviates  the weighted $L^2$ 
			scalar product with induced $\varrho$-weighted $L^2$ norm 
			$\Vert \bullet \Vert_\varrho:=\Vert \varrho^{1/2} \bullet\Vert_{L^2(\Omega)}$. 
			The discrete space 
			$P_{m}(\mathcal{T}):=\{p_m \in L^2(\Omega):\, p_m|_T\in P_m(T) \text{ is a polynomial of degree at most } m\text{ for any }T\in\mathcal{T}\}$ 
		 	consists of piecewise polynomials, the spaces $\textit{CR}^1_0(\mathcal{T})$ resp.\ $M(\mathcal{T})$ 
		 	will be defined in \cref{sec:CR} resp.~\ref{sec:Morley} below. 
			Given a function $v\in L^2(\omega)$, define the integral mean 
			$\intmean_\omega v\,\textup{d}x:= 1/|\omega|\,\int_\omega v\,\textup{d}x$. 		 	
		 	The $L^2$ projection $\Pi_0$ onto the piecewise constant functions $ P_0(\mathcal{T})$ reads 
			$(\Pi_0f)|_T:=\intmean_T f\,\textup{d}x$ for all 
			$f\in L^2(\Omega)$ and  $T\in\mathcal{T}$.  
			Let $\sigma:=\min\{1,\sigma_{\mathrm{reg}}\}$ denote the minimum of one and 
			the  index of elliptic regularity $\sigma_{\mathrm{reg}}>0$ for the source problem of 
			the $m$-Laplacian $(-1)^m\Delta^m$ in $H^m_0(\Omega)$:  
			Given any right-hand side $f\in L^2(\Omega)$, the weak
			solution $u\in V$ to $(-1)^m\Delta^m u=f $ satisfies 
			\begin{align}
					u\in H^{{m}+\sigma}(\Omega)
					\text{ and } \Vert u\Vert_{H^{m+\sigma}(\Omega)}\le C(\sigma)\Vert f\Vert_{L^2(\Omega)}. 
					\label{eq:def_sigma}
			\end{align}
			(This is well-established for $m=1$ 
			\cite{Necas1967,GilbargTrudinger1983,Dauge1988,Grisvard1992,Agmon2010} and $m=2$ in $2$D 
			\cite{BlumRannacher1980} with $\sigma_{\mathrm{reg}}>1/2$ and otherwise a hypothesis throughout this paper.) 
			The Sobolev space $H^{m+s}(\Omega)$ is defined for $0<s<1$ by complex interpolation of 
			$H^m(\Omega)$ and $H^{m+1}(\Omega)$, $m\in\mathbb{N}_0$. 
			Throughout this paper, $a \lesssim b$ abbreviates $a\le Cb$ with a generic constant $C$ 
			depending on $\sigma$ in \eqref{eq:def_sigma} and the shape-regularity 
			of $\mathcal{T}\in\mathbb{T}$ only; 
			$a \approx b$ stands for $a\lesssim  b\lesssim   a$. 

		\subsection{Interpolation}\label{sec:Interpolation}
			The operators $I$ and $J$ concern the (nonconforming) discrete  space ${V(\mathcal{T})}\subset P_{m}(\mathcal{T})$ and 
			$V:=H^m_0(\Omega)$ for an admissible triangulation $\mathcal{T}\in\mathbb{T}$. 
			An advantage of separate interest is that the analysis with $I$ and $J$  is performed 
			simultaneously for $m\ge 1$, while the examples in Subsection~\ref{sec:example} below concern $m=1,2$. 
		
		   Suppose that, for each admissible triangulation $\mathcal{T}\in\mathbb{T}$, there exists a linear interpolation operator $I$
		   onto  $V(\mathcal{T})$  that is defined on $V+V(\widehat{\mathcal{T}})$ for any refinement 
		   $\widehat{\mathcal{T}}\in\mathbb{T}(\mathcal{T})$ and that satisfies 
		   the following properties with universal positive constants $\kappa_{m}$ and $\kappa_{d}$; in all examples below $\kappa_{m}$ is known and the existence of $\kappa_{d}$ is clarified.
			
			\begin{enumerate}[label=(I\arabic*)]
			\item 		\label{item:I_kappa}
						Any $T\in \mathcal{T}$ and $v\in H^{m}(T)$ satisfy 
						$\Vert v- I v\Vert_{L^2(T)}\le \kappa_{m} h_T^{m} \vert v-Iv\vert_{H^m(T)}$.
						
			\item 		\label{item:I_Pi0} 
						The piecewise derivative $D^m_{\mathrm{pw}}$ 
						of any $v\in V+V(\widehat{\mathcal{T}})$ satisfies $D^m_{\mathrm{pw}}Iv=\Pi_0D^m_{\mathrm{pw}}v$.  
			\item 		\label{item:I_identity}
						The operator $I$ acts as identity in non-refined tetrahedra in that 
						$
							(1-I)\widehat{v}_{\mathrm{nc}}|_T=0 
							\text{ in }T\in \mathcal{T}\cap \widehat{\mathcal{T}}
									 \text{ for all }\widehat{v}_{\mathrm{nc}}\in V(\widehat{\mathcal{T}}). 
						$
			 			The interpolation operator $\widehat{I}$ associated with $V(\widehat{\mathcal{T}})$ satisfies 
						$I\circ \widehat{I}=I$ in $V+V(\widehat{\mathcal{T}})$. 
			\item \label{item:I_kappa_disc}
						Any $T\in \mathcal{T}$ and $\widehat{v}_{\mathrm{nc}}\in V(\widehat{\mathcal{T}})$ satisfy 
						$
								\Vert \widehat{v}_{\mathrm{nc}}- I \widehat{v}_{\mathrm{nc}}\Vert_{L^2(T)}
														\le {\kappa_d} h_T^{m} \vert \widehat{v}_{\mathrm{nc}}-I\widehat{v}_{\mathrm{nc}}\vert_{H^m(T)}.			
						$ 
			\end{enumerate}
			\begin{corollary}[properties of $I$]\label{cor:InterpolationOperator}
				\begin{enumerate}[label=(\alph*), ref=\alph*, leftmargin=1.6em ]	
					\item\label{item:cor_I_orthogonality}
						Given $\widehat{\mathcal{T}}\in\mathbb{T}(\mathcal{T})$,  
						any $v\in V+V(\widehat{\mathcal{T}})$ and $w_{\mathrm{nc}}\in V(\mathcal{T})$ satisfy 
						$\displaystyle a_{\mathrm{pw}}(v-Iv,w_{\mathrm{nc}})=0$ 
						and $\displaystyle \vvvert v-Iv\vvvert_{\mathrm{pw}}=\min_{v_{\mathrm{nc}}\in V(\mathcal{T})}\vvvert v-v_{\mathrm{nc}}\vvvert_{\mathrm{pw}}$.
					\item\label{item:cor_I_alpha} 
						Any $v\in H^{m+s}(\Omega)$ with $1/2<s\le 1$  satisfies 
						$
							\vvvert (1-I)v\vvvert_{\mathrm{pw}}\le (h_{\max}/\pi)^{s} \Vert v\Vert_{H^{m+s}(\Omega)}. 
						$
			 			
					\item\label{item:cor_I_a} Any $v,\, w\in V$ and $v_{\mathrm{nc}}\in V(\mathcal{T})$ satisfy
								$
						 			a_{\mathrm{pw}}(v, v_{\mathrm{nc}})=a_{\mathrm{pw}}(Iv,v_{\mathrm{nc}})		
								$
							and\\ 
							$\hfill\displaystyle
						 		a_{\mathrm{pw}}(v, (1-I)w)=a_{\mathrm{pw}}((1-I)v, (1-I)w)
						 			\le \min_{v_{\mathrm{nc}}\in V(\mathcal{T})} \vvvert v-v_{\mathrm{nc}}\vvvert_{\mathrm{pw}}
						 				\min_{w_{\mathrm{nc}}\in V(\mathcal{T})} \vvvert w-w_{\mathrm{nc}}\vvvert_{\mathrm{pw}}	. 		 
							$
					\item\label{item:cor_I_b} Any $w\in V$ and $v\in V+V(\mathcal{T})$ satisfy\\
							$\hfill\displaystyle
						 		b(v, (1-I)w)\le \Vert h_{\mathcal{T}}^m v\Vert_{L^2(\Omega)}\Vert h_{\mathcal{T}}^{-m} (1-I)w\Vert_{L^2(\Omega)}
						 			\le \kappa_m\Vert h_{\mathcal{T}}^m v\Vert_{L^2(\Omega)}\min_{w_{\mathrm{nc}}\in V(\mathcal{T})} 
						 			\vvvert w-w_{\mathrm{nc}}\vvvert_{\mathrm{pw}}. 		 
							$
				\end{enumerate}
			\end{corollary}
			\begin{proof}
					Since $D^m_{\mathrm{pw}}w_{\mathrm{nc}}\in P_0(\mathcal{T};\mathbb{R}^{3^m})$, 
					\ref{item:I_Pi0} implies (\ref{item:cor_I_orthogonality}). 
					In combination with a piecewise  Poincar\'e inequality, \ref{item:I_Pi0} implies (\ref{item:cor_I_alpha}) 
					(see \cite[Cor.~2.2.a]{CP_Part1} for details).
					The first claim in (\ref{item:cor_I_a}) follows from (\ref{item:cor_I_orthogonality}). 
					The combination of (\ref{item:cor_I_orthogonality}) with the Cauchy-Schwarz inequality proves (\ref{item:cor_I_a}). 
					The Cauchy-Schwarz inequality, the approximation property \ref{item:I_kappa}, and (\ref{item:cor_I_a}) conclude the proof of (\ref{item:cor_I_b}). 
			\end{proof}
		\subsection{Conforming companion}\label{sec:Companion}
			Given any tetrahedron $T\in\mathcal{T}$ in a triangulation $\mathcal{T}\in\mathbb{T}$, 
			let $\mathcal{V}(T)$ denote the set of its vertices ($0$-subsimplices) 
			and let $\mathcal{F}(T)$ denote the set of its faces ($2$-subsimplices). 
		 	A linear operator $J: {V(\mathcal{T})}\to V$ is called \emph{conforming 
		 	companion} if \ref{item:J_rightInverse}--\ref{item:J_T_That} hold 
		 	with universal constants $M_1,M_2,\,M_4$ (that exclusively depend on $\mathbb{T}$).
		 	 \begin{enumerate}[label=(J\arabic*)]
		 		\item\label{item:J_rightInverse} 
		 				$J$ is a right inverse to the interpolation $I$ in the sense that 
		 					$
		 						I\circ J =1 \text{ in  } {V(\mathcal{T})}. 
		 					$
		 		\item\label{item:J_approx} 
					$
						 \displaystyle\Vert h_{\mathcal{T}}^{-{m}} (1-J)v_{\mathrm{nc}}\Vert_{L^2(\Omega)}
						 +\vvvert (1-J)v_{\mathrm{nc}}\vvvert_{\mathrm{pw}}
									\le \Big(M_1 \sum_{T\in\mathcal{T}}|T|^{1/3}\sum_{F\in\mathcal{F}(T)}\Vert [D^m_{\mathrm{pw}} v_{\mathrm{nc}}]_F\times\nu_F\Vert^2_{L^2(F)}\Big)^{1/2}$
					$\\		\le {M_2} \min_{v\in V}\vvvert v_{\mathrm{nc}}-v\vvvert_{\mathrm{pw}}$ 
					 for any $v_{\mathrm{nc}}\in V(\mathcal{T})$.		
		    	\item\label{item:J_orthogonality}
						 $\displaystyle(1-J)(V(\mathcal{T}))\perp P_{m}(\mathcal{T})$  holds in $L^2(\Omega)$.		
		    	\item\label{item:J_T_That} 	
		 				$\displaystyle\vert v_{\mathrm{nc}} -Jv_{\mathrm{nc}}\vert_{H^m(K)}^2
		 				\le {M_4} \sum_{T\in \mathcal{T}(\Omega(K))}|T|^{1/3}
		 				\sum_{F\in\mathcal{F}(T)}\Vert  [D^{m}_{\mathrm{pw}}  v_{\mathrm{nc}}]_F\times \nu_F\Vert_{L^2(F)}^2$
		 				holds  for any $v_{\mathrm{nc}}\in V(\mathcal{T})$ and $K\in\mathcal{T}$ with the set 
		 				$\mathcal{T}(\Omega(K)):=\{T\in\mathcal{T}: \mathrm{dist}(T,K)=0\}$ of adjacent tetrahedra.  
		 	\end{enumerate}
		 	The properties \ref{item:J_rightInverse}--\ref{item:J_T_That} 
		 	\cite{CGS15,Gal15,CP_Part1}  are stated for convenient quotation throughout this paper. 
	 	The localized version \ref{item:J_T_That}  applies at the very end (in \cref{thm:A34EVP}) and implies parts of \ref{item:J_approx}. The second inequality in \ref{item:J_approx} is the efficiency of a posteriori error estimators. 	 	
		 	\begin{remark}[on  \ref{item:J_T_That}] \label{rem:J2_wanted}
		 			For any refinement $\widehat{\mathcal{T}}\in\mathbb{T}(\mathcal{T})$ of a triangulation 
		 			$\mathcal{T}\in\mathbb{T}$, let 
		 			$\mathcal{R}_1:=\{K\in\mathcal{T}:\,\exists\, T\in \mathcal{T}\setminus\widehat{\mathcal{T}} 
		 			\text{ with }\allowbreak \textup{dist}(K,T)=~0\}\subset~\mathcal{T}$ denote the set of coarse but not 
		 			fine tetrahedra plus one layer of coarse tetrahedra around. Then \ref{item:J_T_That} and a finite overlap argument imply 
		 			the existence of $M_5>0$ such that any $v_{\mathrm{nc}}\in V(\mathcal{T})$ satisfies 
		 			\begin{align*}
		 				\Vert D^{m}_{\mathrm{pw}}(v_{\mathrm{nc}} -Jv_{\mathrm{nc}})\Vert_{L^2(\mathcal{T}\setminus\widehat{\mathcal{T}})}^2
		 				\le M_5
		 				\sum_{T\in \mathcal{R}_1}|T|^{1/3}\sum_{F\in\mathcal{F}(T)}\Vert  [D^{m}_{\mathrm{pw}}  v_{\mathrm{nc}}]_F\times \nu_F\Vert_{L^2(F)}^2.
					\end{align*} 	
		 			The superset $\mathcal{R}_1$ of $\mathcal{T}\setminus\widehat{\mathcal{T}}$ serves as a simple example and 
		 			could indeed be replaced by $\mathcal{T}\setminus\widehat{\mathcal{T}}$ provided $J$ may depend on $\widehat{\mathcal{T}}$;  
		 			cf.  \cite[\S6]{CP18} for details in the two model problems below.\phantom{x}\hfill$\Box$
		 	\end{remark}
		 	
		 	\begin{corollary}[properties of $J$]\label{cor:ConformingCompanion} Any $w\in V$ and  $v_{\mathrm{nc}}\in V(\mathcal{T})$ satisfy
		 		\begin{enumerate}[label=(\alph*), ref=\alph* ]
		 			\item\label{item:cor_J_L2_vh}
		 			$\displaystyle\Vert v_{\mathrm{nc}}-Jv_{\mathrm{nc}}\Vert_{L^2(\Omega)}=\Vert (1-I)Jv_{\mathrm{nc}}\Vert_{L^2(\Omega)}
		 			\le \kappa_{m} \vvvert h_{\mathcal{T}}^{m}(v_{\mathrm{nc}}-Jv_{\mathrm{nc}})\vvvert_{\mathrm{pw}}$\\
		 			$\displaystyle\phantom{xxxxxxxxxxxxxi}
		 			\le h_{\max}^{m}\kappa_{m}{M_2}\min_{v\in V}\vvvert v_{\mathrm{nc}}-v\vvvert_{\mathrm{pw}};$
		 			\item\label{item:cor_J_b_vh}
		 			$\displaystyle b(w, v_{\mathrm{nc}}-Jv_{\mathrm{nc}})=b(w-Iw,v_{\mathrm{nc}}-Jv_{\mathrm{nc}})
		 			\le \Vert w-Iw\Vert_{L^2(\Omega)}\Vert v_{\mathrm{nc}}-Jv_{\mathrm{nc}}\Vert_{L^2(\Omega)}$\\
		 			$\displaystyle\phantom{xxxxxxxxxxxxxi}\le h_{\max}^{2{m}} \kappa_{m}^2{M_2} 
		 			\min_{w_{\mathrm{nc}}\in V(\mathcal{T})}\vvvert w-w_{\mathrm{nc}}\vvvert_{\mathrm{pw}}
														 \min_{v\in V}\vvvert v_{\mathrm{nc}}-v\vvvert_{\mathrm{pw}};$
		 			\item\label{item:cor_J_a_vh}
		 			$\displaystyle a_{\mathrm{pw}}(w, v_{\mathrm{nc}}-Jv_{\mathrm{nc}})
		 			=a_{\mathrm{pw}}(w-Iw,v_{\mathrm{nc}}-Jv_{\mathrm{nc}})\le\vvvert w-Iw\vvvert_{\mathrm{pw}}
		 			\vvvert v_{\mathrm{nc}}-Jv_{\mathrm{nc}}\vvvert_{\mathrm{pw}}$\\
		 			$\displaystyle\phantom{xxxxxxxxxxxxxxxi}	\le {M_2} \min_{w_{\mathrm{nc}}\in V(\mathcal{T})}
		 			\vvvert w-w_{\mathrm{nc}}\vvvert_{\mathrm{pw}}\min_{v\in V}\vvvert v-v_{\mathrm{nc}}\vvvert_{\mathrm{pw}}.$
		 		\end{enumerate}
		 	\end{corollary}
		 	\begin{proof}
		 		The combination of \ref{item:J_rightInverse}, \ref{item:I_kappa}, and \ref{item:J_approx}  proves (\ref{item:cor_J_L2_vh}).
		 		The claim (\ref{item:cor_J_b_vh}) follows from \ref{item:J_orthogonality}, the Cauchy-Schwarz inequality,  \ref{item:I_kappa}, and 
		 		(\ref{item:cor_J_L2_vh}). 
		 		\cref{cor:InterpolationOperator}.\ref{item:cor_I_a}  and 
		 		\ref{item:J_rightInverse}--\ref{item:J_approx}  lead to (\ref{item:cor_J_a_vh}). 
		 	\end{proof}	 		
			
		 \subsection{Examples}\label{sec:example}
				Two examples for $V(\mathcal{T})\subset P_m(\mathcal{T})$ are analysed simultaneously in this paper for $m=1,2$.  
				 It is appealing to follow 	our methodology for $m\ge 3$ \cite{WangXu2013} in future research. 
			
			\subsubsection{Crouzeix-Raviart finite elements for the Laplacian ($\boldsymbol{m=1}$)} \label{sec:CR}
			 	Given the shape-regular triangulation $\mathcal{T}\in\mathbb{T}$,	
			 	let $\mathcal{F}$ (resp. $\mathcal{F}(\Omega)$ or $\mathcal{F}(\partial\Omega)$) denote the set of all 
				(resp. interior or boundary) faces.
			 	Throughout this paper,  the model problem with ${m}=1$ approximates the Dirichlet eigenvectors $u\in H^1_0(\Omega)$ 
			 	of the  Laplacian $ -\Delta u=\lambda u$ 	
			 	 in the Crouzeix-Raviart finite element space \cite{CR73}
				\begin{align*}
						V(\mathcal{T}):=\textit{CR}^1_0(\mathcal{T})
							:=\{v\in P_1(\mathcal{T}):\ &v \text{ is continuous at }\textup{mid}(F)\text{ for all }F\in\mathcal{F}(\Omega)
							\text{ and }\\&v(\textup{mid}(F))=0\text{ for all }F\in\mathcal{F}(\partial\Omega) \}.	
				\end{align*} 
				Given the face-oriented basis functions $\psi_{F}\in \textit{CR}^1(\mathcal{T})$ with $\psi_F(\textup{mid}(E))=\delta_{EF}$ for all faces 
				$E,F\in\mathcal{F}$ ($\delta_{EF}$ is Kronecker's delta), the standard interpolation operator reads
					$$
						I_{\text{CR}}(v):=\sum_{F\in\mathcal{F}(\Omega)}\bigg(\intmean_F v\,\textup{d}\sigma\bigg)\psi_F\quad\text{ for any }
										v\in H^1_0(\Omega)+\textit{CR}^1_0(\widehat{\mathcal{T}}). $$
			 	The interpolation operator $I_{\text{CR}}$ satisfies \ref{item:I_kappa}--\ref{item:I_kappa_disc} 
			 	with $\kappa_1:=\sqrt{{1}/{\pi^2}+{1}/{120}}$, 
			 	see \cite[Sec. 4.2--4.4]{CP18} and the references therein. 
			 	The constant $\kappa_1$ is provided in \cite{CGal14,CGed14,CZZ18}. 

			 	The design of the conforming companion $J:\textit{CR}^1_0(\mathcal{T})\to S^{5}_0(\mathcal{T}):=P_{5}(\mathcal{T})\cap C_0(\Omega)$ 
			 	with \ref{item:J_rightInverse}--\ref{item:J_T_That}  is a straightforward generalization of \cite[Prop.~2.3]{CGS15} to $3$D. 
			 	The arguments in  \cite[Prop.~2.3]{CGS15} can be localized \cite[Thm.~5.1]{CEHL12} and 
			 	lead with \cite[Thm.~3.2]{CBJ01},\cite[Thm.~4.9]{CGS13} to \ref{item:J_approx} and 
			 	\ref{item:J_T_That}.  
			 		 		
			\subsubsection{Morley finite elements for the bi-Laplacian ($\boldsymbol{m=2}$)}\label{sec:Morley}
				Given the shape-regular triangulation $\mathcal{T}\in\mathbb{T}$,  
				let $\mathcal{E}$ (resp. $\mathcal{E}(\Omega)$ or $\mathcal{E}(\partial\Omega)$) denote the set of all 
				(resp. interior or boundary) edges.  Let $\mathcal{F}(E):=\{F\in\mathcal{F}:\, E\subset \overline{F}\}$ 
				denote the set of all faces containing the edge $E\in\mathcal{E}$. 
				For any face $F\in \mathcal{F}$, let $\nu_F$ denote the unit normal with fixed orientation and $[\bullet]_F$ the jump across $F$. 
			 	The model problem with ${m}=2$  approximates the Dirichlet eigenvectors $u\in H^2_0(\Omega)$ of the bi-Laplacian 
			 	$ \Delta^2 u=\lambda u$ 	
				in the  discrete Morley finite element space \cite{Mor68,MX06} 
			 	\begin{align*}
					V(\mathcal{T}):={M}(\mathcal{T}):=\Big\{v\in P_2(\mathcal{T}):\ &
														\intmean_E [v]_F\,\textup{d}s=0\text{ for all }E\in\mathcal{E}\text{ and }F\in\mathcal{F}(E) 
														\text{, }\\&\text{and }
														\intmean_F [\nabla v]_F\cdot\nu_F\,\textup{d}\sigma=0 \text{ for all }F\in\mathcal{F}
														\Big\}.		
				\end{align*} 
				Given the nodal basis functions $\Phi_E,\Phi_F$ for any $E\in\mathcal{E}$ and  $F\in \mathcal{F}$ 
				(see \cite[Eq.~(2.1)--(2.2)]{CP_Part1} for details),  the standard interpolation operator 
				\cite{CGal14, Gal15, CP18, CP_Part1} reads 
				$$I_M(v):=\sum_{E\in\mathcal{E}(\Omega)}\bigg(\intmean_E v\,\textup{d}s\bigg)\phi_E
					+\sum_{F\in\mathcal{F}(\Omega)}\bigg(\intmean_F\nabla v\cdot\nu_F\,\textup{d}\sigma\bigg) \phi_F
								\quad\text{ for any } 
								v\in H^2_0(\Omega)+M(\widehat{\mathcal{T}}).$$
			 	The operator $I_M$ satisfies \ref{item:I_kappa}--\ref{item:I_kappa_disc} with 
			 	$\kappa_2:={\kappa_{1}}/{\pi}+\sqrt{({3\kappa_{1}^2+2\kappa_{1} })/80}$ as discussed in \cite{CP18,CP_Part1};  
			 	$\kappa_2$ is provided in \cite{CGal14,CP_Part1}. 
			 	
			 	There exists a conforming companion $J:M(\mathcal{T})\to V$ 
			 	based on the Hsieh-Clough-Tocher FEM \cite[Chap.~6]{Ciarlet78}  with \ref{item:J_rightInverse}--\ref{item:J_T_That}  
			 	in \cite{Gal15,VZ19,CP18} in $2$D and 
				on the Worsey-Farin FEM  \cite{WF87} with \ref{item:J_rightInverse}--\ref{item:J_orthogonality}  
				in \cite{CP_Part1} in $3$D.
				Since the arguments in the proof of \ref{item:J_approx} in \cite[Thm.~3.1.b]{CP_Part1} are  local, 
				\ref{item:J_T_That} follows in $3$D as well. 
\section{Medius analysis}\label{sec:apriori}
	This section shows that \ref{item:I_kappa}--\ref{item:I_Pi0} and \ref{item:J_rightInverse}--\ref{item:J_orthogonality} 
	lead to best-approximation and error estimates in weaker Sobolev norms. 	

\subsection{Main result and layout of the proof}\label{sec:apriori_thm} 
					Throughout this paper, $k\in\mathbb{N}$ is the number of a \emph{simple} exact eigenvalue $\lambda\equiv\lambda_k$.  
					The aim of this section is the proof of \cref{thm:BoundInterpolationError} with 
					$\Vert\bullet\Vert_{\delta}$ defined in \eqref{eq:def_delta} below. 
			\begin{theorem}[best-approximation]\label{thm:BoundInterpolationError}
					Let $(\lambda,u)\in \mathbb{R}^+\times V$  denote  the ${k}$-th continuous eigenpair 
					of \eqref{eq:contEVP} with a \emph{simple} eigenvalue $\lambda\equiv\lambda_k$ and  $\Vert u\Vert_{L^2(\Omega)}=1$. 
					There exist $\varepsilon_5>0$ and $C_0>0$ such that, for all  
					$\mathcal{T}\in \mathbb{T}(\varepsilon_5):=\{\mathcal{T}\in\mathbb{T}:\,h_{\max}\le \varepsilon_5\}$,   
					there exists a discrete eigenpair   $(\lambda_h,\boldsymbol{u_h})\in \mathbb{R}^+\times \boldsymbol{V_h}$ of number ${k}$ of 
					\eqref{eq:dis_EVP_alt}  
					with $\lambda_h\equiv\lambda_h({k})$, $\boldsymbol{u_h}=(u_{\mathrm{pw}},u_{\mathrm{nc}})$,  
					$\Vert u_{\mathrm{nc}}\Vert_{L^2(\Omega)}=1$,  
					and $b(u,u_{\mathrm{nc}})> 0$ such that 
					\begin{enumerate}[label=(\alph*),ref=\alph*,wide]
						\item \label{item:lambdahSimple}
							$\lambda_h(k)$ is a simple algebraic eigenvalue of \eqref{eq:disEVP_intermediate}
							with ${\lambda_k}/{2}\le \lambda_h(k)$,
						\item \label{item:lambdahGLB}
							$\lambda_h(j)\le\lambda_j$ for all $j=1,\dots,k+1$, 
						\item \label{item:BoundInterpolationError}
					 		$\displaystyle |\lambda-\lambda_h|+\vvvert u-u_{\mathrm{nc}}\vvvert_{\mathrm{pw}}^2
					  		+h_{\max}^{-2\sigma}\Vert u-u_{\mathrm{nc}}\Vert_{L^2(\Omega)}^2+\Vert u_{\mathrm{nc}}\Vert_{\delta}^2
					  		\le C_0 \vvvert u- I u\vvvert_{\mathrm{pw}}^2$.
					\end{enumerate}
			\end{theorem}
					Some comments on related results and an outline of the proof of \cref{thm:BoundInterpolationError} are in order before 
					Subsections~\ref{sec:IntermediateEVP}--\ref{sec:BestApproxFinish} provide details. 
					\begin{remark}[known convergence results]\label{rem:convergence_result}
						The analysis in \cite{CP_Part1} (\S~2.3.3 for $m=1$ and Thm.~1.2 for $m=2$) 
						guarantees the convergence of the eigenvalues $\lambda_h$ to $\lambda$
						and the component $u_{\mathrm{pw}}\in P_m(\mathcal{T})$ to $u\in V$. 
						The assumption that $\lambda=\lambda_{k}$ is a simple eigenvalue of \eqref{eq:contEVP} and the convergence 
					  	$\lambda_h({k})\equiv\lambda_h\to \lambda$ as $h_{\max}\to 0$ lead to the existence of $\varepsilon_0>0$ such that 
					  	the number $M:=\textup{dim}(P_m(\mathcal{T}))$ of discrete eigenvalues of \eqref{eq:dis_EVP_alt} is larger than $k+1$ and  
					  	$\lambda_h({k}-1)<\lambda_h({k})\equiv\lambda_h<\lambda_h({k}+1)$ as well as  $\lambda_k/{2}\le \lambda_h(k)$  
					    for all $\mathcal{T}\in\mathbb{T}(\varepsilon_0)$. 
					  	Then the eigenfunction  $\boldsymbol{u_h}=(u_{\mathrm{pw}},u_{\mathrm{nc}})\in \boldsymbol{V_h}\setminus\{0\}$  
					    is unique.
					\end{remark} 
						The convergence analysis in \cite{CP_Part1} displays convergence of the eigenvector $u_{\mathrm{pw}}\in P_m(\mathcal{T})$ 
						but not for the nonconforming component $u_{\mathrm{nc}}\in V(\mathcal{T})$. 
					    This section focusses on the convergence analysis for $u_{\mathrm{nc}}\in V(\mathcal{T})$.
						Recall that $k\in\mathbb{N}$ is fixed and $(\lambda, u)$  denotes the $k$-th eigenpair of 
						\eqref{eq:contEVP} with a simple eigenvalue 
						$\lambda\equiv\lambda_k>0$ and $\Vert u\Vert_{L^2(\Omega)}=1$.  
						Set $\varepsilon_1:=\min\{\varepsilon_0, (2\lambda_{k+1}\kappa_m^2)^{-1/(2m)}\}$ and suppose 
						$\mathcal{T}\in\mathbb{T}(\varepsilon_1)$.
						Let $(\lambda_h,\boldsymbol{u_h})$ denote the $k$-th discrete eigenpair  in \eqref{eq:dis_EVP_alt}
						with $\lambda_h\equiv\lambda_h(k)>0$, $\boldsymbol{u_h}=(u_{\mathrm{pw}},u_{\mathrm{nc}})\in\boldsymbol{V_h}$, 
						$\Vert u_{\mathrm{nc}}\Vert_{L^2(\Omega)}=1$, and $b(u,u_{\mathrm{nc}})\ge 0$. 
					\begin{proofof}\textit{\cref{thm:BoundInterpolationError}.\ref{item:lambdahSimple}.}
						This follows from \cref{rem:convergence_result} for $\varepsilon_1:=\min\{\varepsilon_0, (2\lambda_{k+1}\kappa_m^2)^{-1/(2m)}\}$.$\ $
					\end{proofof}
					\begin{proofof}\textit{\cref{thm:BoundInterpolationError}.\ref{item:lambdahGLB}.}
						The choice $\varepsilon_1:=\min\{\varepsilon_0, (2\lambda_{k+1}\kappa_m^2)^{-1/(2m)}\}$ implies for all $j=1,\dots,k$  that 
						$\lambda_j\kappa_{m}^2h^{2{m}}_{\max}\le\lambda_{k+1}\kappa_{m}^2\varepsilon_1^{2{m}}=1/2$.  
						 Hence \eqref{eq:GLB_cond} proves \cref{thm:BoundInterpolationError}.\ref{item:lambdahGLB}. 
					\end{proofof}
					\begin{remark}[weight $\delta$]\label{rem:delta}				
						The piecewise constant weight $\delta\in P_0(\mathcal{T})$ in the weighted $L^2$ norm 
						$\Vert\bullet\Vert_\delta:=\Vert \sqrt{\delta}\bullet\Vert_{L^2(\Omega)}$ 
						on the left-hand side of \cref{thm:BoundInterpolationError}.\ref{item:BoundInterpolationError} reads 
						\begin{align}
							\delta:=\frac{1}{1-\lambda_h\kappa_{m}^2h_{\mathcal{T}}^{2{m}}}-1
								=\frac{\lambda_h\kappa_{m}^2h_{\mathcal{T}}^{2{m}}}{1-\lambda_h\kappa_{m}^2h_{\mathcal{T}}^{2{m}}}
								=\lambda_h\kappa_{m}^2h_{\mathcal{T}}^{2{m}}(1+\delta)\in P_0(\mathcal{T}). \label{eq:def_delta}
						\end{align}
						Notice that $h_{\max}\le\varepsilon_1$ implies $\delta\le\delta_{\max}:=(1-\lambda_h \kappa_{m}^2h_{\max}^{2{m}})^{-1}-1\le1$. 
						The constant $C_\delta:=2\lambda\kappa_m^2$ satisfies $\delta\le C_{\delta}h_{\mathcal{T}}^{2{m}}\le C_{\delta}h_{\max}^{2{m}}$  
						(because $\lambda_h\le\lambda$  from \cref{thm:BoundInterpolationError}.\ref{item:lambdahGLB})
						and $\delta$ converges to zero as the maximal mesh-size $h_{\max}\to 0$ approaches zero. 
					\end{remark}
					\begin{remark}[related work]
						This section extends the analysis in \cite[Section 2--3]{CGS15} to a  simultaneous analysis of the Crouzeix-Raviart and Morley 
						FEM and to the extra-stabilized discrete eigenvalue problem (EVP) \eqref{eq:dis_EVP_alt} and to $3$D. 
					\end{remark}
					\begin{remark}[equivalent problem]
			  			Since $\lambda_h\kappa_{m}^2h^{2{m}}_{\max}\le\lambda_{k+1}\kappa_{m}^2\varepsilon_1^{2{m}}=1/2$,  \eqref{eq:dis_EVP_alt} is 
						equivalent to a reduced rational eigenvalue problem that seeks 
						$(\lambda_h,u_{\mathrm{nc}})\in \mathbb{R}^+\times (V(\mathcal{T})\setminus\{0\})$ with 
						\begin{align}
							a_{\mathrm{pw}}(u_{\mathrm{nc}},v_{\mathrm{nc}})
							=\lambda_h \Big(\frac{u_{\mathrm{nc}}}{1-\lambda_h\kappa_{m}^2 h_{\mathcal{T}}^{2{m}}},v_{\mathrm{nc}}\Big)_{L^2(\Omega)}
							\qquad\text{for all }v_{\mathrm{nc}}\in V(\mathcal{T}) \label{eq:disEVP}
						\end{align} and $u_{\mathrm{pw}}=(1-\lambda_h\kappa_{m}^2h_{\mathcal{T}}^{{2m}})^{-1}u_{\mathrm{nc}}$ 
						 \cite[Prop.~2.5,~\S~2.3.3]{CP_Part1}.
					 \end{remark}
				\textbf{Outline of the proof of \cref{thm:BoundInterpolationError}.\ref{item:BoundInterpolationError}.} 
					The outline of the proof of \cref{thm:BoundInterpolationError}.\ref{item:BoundInterpolationError} provides an overview and clarifies the 
					various steps for a reduction of $\varepsilon_1$ to $\varepsilon_5$, before the technical details follow in the subsequent subsections. 
					The coefficient $(1-\lambda_h\kappa_m^2h_{\mathcal{T}}^{2m})^{-1}=1+\delta\in P_0(\mathcal{T})$ with $\lambda_h\equiv\lambda_h({k})$ 
					on the right-hand side of \eqref{eq:disEVP} is frozen in the intermediate EVP.  
				\begin{definition}[intermediate EVP]\label{def:intermediateEVP}					
					Recall $(\bullet,\bullet)_{1+\delta}:=((1+\delta)\bullet,\bullet)_{L^2(\Omega)}$. 
					Let $(\mu,\phi)\in \mathbb{R}^+\times V(\mathcal{T})\setminus\{0\}$ solve the (algebraic) eigenvalue problem 
					\begin{align}
						a_{\mathrm{pw}}(\phi,v_{\mathrm{nc}})
						=\mu (\phi,v_{\mathrm{nc}})_{1+\delta}
						\quad\text{for all }v_{\mathrm{nc}}\in V({\mathcal{T}}).\label{eq:disEVP_intermediate}
					\end{align}
					The two coefficient matrices in \eqref{eq:disEVP_intermediate} are SPD and 	there exist  
					$N:=\textup{dim}\,V(\mathcal{T})$ (algebraic) eigenpairs 
					$(\mu_1,\phi_1),\dots,(\mu_N,\phi_N)$ of \eqref{eq:disEVP_intermediate}. 
					The eigenvectors $\phi_1,\dots,\phi_N$ are  $(\bullet,\bullet)_{1+\delta}$-orthonormal and 
					the eigenvalues $\mu_1\le \dots\le \mu_N$ are enumerated in ascending order counting multiplicities. 
				\end{definition}
					Since  $\lambda_h$ is an eigenvalue of the rational problem \eqref{eq:disEVP},  
					$\lambda_h\in \{\mu_1,\dots,\mu_N\}$ belongs to the  eigenvalues of \eqref{eq:disEVP_intermediate}.	
					\cref{lem:mu_is_fine} below guarantees the convergence $|\mu_j- \lambda_h(j)|\to 0$ as $h_{\max}\to 0$ for $j =1,\dots,k+1$. 
					Hence there exist positive $\varepsilon_2\le\min\{1/2,\varepsilon_1\}$ 
					and  $M_6$ such that $\mathcal{T}\in\mathbb{T}(\varepsilon_2)$ implies 
					\begin{enumerate}[label=(H\arabic*)]
						\item\label{item:H1} $\ \ \mu_k=\lambda_h(k)$ is a simple algebraic eigenvalue of \eqref{eq:disEVP_intermediate}, 
						\item\label{item:H2}  $\displaystyle \max_{\substack{j=1,\dots, N\\ j\not =k}}\frac{\lambda_k}{|\lambda_k-\mu_j|}\le M_6$.
					\end{enumerate}
					The intermediate EVP and the following associated source problem  allow for the control of the extra-stabilization. 
				\begin{definition}[auxiliary source problem]\label{def:intermediatesource} 
					Let $z_{\mathrm{nc}}\in V({\mathcal{T}}) $ denote the solution to  
					\begin{align}
						a_{\mathrm{pw}}(z_{\mathrm{nc}},v_{\mathrm{nc}})=(\lambda u,v_{\mathrm{nc}})_{1+\delta}
						\quad\text{for all }v_{\mathrm{nc}}\in V({\mathcal{T}}) . \label{eq:def_zCR}
					\end{align} 
				\end{definition}
					For any $\mathcal{T}\in\mathbb{T}(\varepsilon_2)$, \cref{sec:IntermediateSource} below provides  $C_1,C_2>0$ that satisfy 
					\begin{align}
						\Vert u-u_{\mathrm{nc}}\Vert_{L^2(\Omega)}&\le C_1 \Vert u-z_{\mathrm{nc}}\Vert_{L^2(\Omega)}, \label{eq:u_zCR_L2}\\
						C_2^{-1}\Vert u-z_{\mathrm{nc}}\Vert_{L^2(\Omega)}
									&\le  h_{\max}^{\sigma}\vvvert u-z_{\mathrm{nc}}\vvvert_{\mathrm{pw}}+\Vert \delta \lambda u\Vert_{L^2(\Omega)}.
									\label{eq:u_zCR_energy}		
					\end{align}		
					The proof of \eqref{eq:u_zCR_L2} in \cref{sec:IntermediateSource} extends  \cite[Lem.~2.4]{CGS15}. 
					The proof of \eqref{eq:u_zCR_energy} utilizes another continuous source 
					problem with the right-hand side $u-Jz_{\mathrm{nc}}$. 
					For all $\mathcal{T}\in\mathbb{T}(\varepsilon_2)$, \cref{sec:IntermediateContSource} below provides a constant $C_3>0$ such that 	
					\begin{align}
							C_3^{-1}\vvvert u-z_{\mathrm{nc}}\vvvert_{\mathrm{pw}}\le \vvvert u-Iu\vvvert_{\mathrm{pw}}
					 		+\Vert \delta \lambda u\Vert_{L^2(\Omega)}.  \label{eq:u_Iu_Energy}	
					\end{align}
					The proof of \eqref{eq:u_Iu_Energy} below rests upon a decomposition of 
					$\vvvert u-z_{\mathrm{nc}}\vvvert_{\mathrm{pw}}^2$ into terms controlled by the conditions  
					\ref{item:I_kappa}--\ref{item:I_Pi0} and \ref{item:J_rightInverse}--\ref{item:J_orthogonality}. 
					Since $h_{\max}\le 1$, the combination of  \eqref{eq:u_zCR_L2}--\eqref{eq:u_Iu_Energy}  reads 
					\begin{align}
									\Vert u-u_{\mathrm{nc}}\Vert_{L^2(\Omega)}
									\le C_1C_2\big(C_3h_{\max}^\sigma \vvvert u-Iu\vvvert_{\mathrm{pw}}+(1+C_3) \Vert \delta\lambda u\Vert_{L^2(\Omega)}\big).
									\label{eq:proof_L2error_collection}
					\end{align}
					The control of  $\Vert \delta \lambda u\Vert_{L^2(\Omega)}$ on the right-hand side of \eqref{eq:proof_L2error_collection} 
					consists of two steps and leads to $c_1:=2\lambda^2\kappa_m^2 C_1C_2(1+C_3)$ and $\varepsilon_3:=\min\{\varepsilon_2,(2c_1)^{-1/2m}\}$. 
					A triangle inequality 
					$\Vert \delta \lambda u\Vert_{L^2(\Omega)}
					\le \Vert \delta \lambda (u-u_{\mathrm{nc}})\Vert_{L^2(\Omega)}+\Vert \delta \lambda u_{\mathrm{nc}}\Vert_{L^2(\Omega)}$, 
					the estimate $\delta\le 2\lambda\kappa_m^2 h_{\max}^{2m}$ in \cref{rem:delta}, 
					and \eqref{eq:proof_L2error_collection} imply 
					\begin{align*}
									\Vert \delta \lambda u\Vert_{L^2(\Omega)}
									\le \frac{c_1C_3h_{\max}^{2m}}{1+C_3}h_{\max}^{\sigma} \vvvert u-Iu\vvvert_{\mathrm{pw}}
									+c_1h_{\max}^{2m}\Vert \delta \lambda u\Vert_{L^2(\Omega)}
									+\Vert \delta \lambda u_{\mathrm{nc}}\Vert_{L^2(\Omega)}. 
					\end{align*}	
					The choice of $\varepsilon_3$ shows 
					$c_1h_{\max}^{2m}\Vert \delta \lambda u\Vert_{L^2(\Omega)}\le \Vert \delta \lambda u\Vert_{L^2(\Omega)}/2$ 
					for any $\mathcal{T}\in\mathbb{T}(\varepsilon_3)$. Therefore 
					\begin{align}
						\Vert \delta \lambda u\Vert_{L^2(\Omega)}\le C_3/(1+C_3) h_{\max}^\sigma \vvvert u-Iu\vvvert_{\mathrm{pw}}
							+2\Vert \delta \lambda u_{\mathrm{nc}}\Vert_{L^2(\Omega)}.\label{eq:delta_uh_2}
					\end{align}
					Notice that $\Vert \delta u_{\mathrm{nc}}\Vert_{L^2(\Omega)}
					\le 2\lambda\kappa_m^2 h_{\max}^m \Vert h_{\mathcal{T}}^{m} u_{\mathrm{nc}}\Vert_{L^2(\Omega)}$ (from \cref{rem:delta})  
					allows for the application of an efficiency estimate 
					\begin{align}
						C_4^{-1} \Vert h_{\mathcal{T}}^{m} u_{\mathrm{nc}}\Vert_{L^2(\Omega)}
							\le h_{\max}^{m}\Vert u-u_{\mathrm{nc}}\Vert_{L^2(\Omega)}+\lambda^{-1} \vvvert u-Iu\vvvert_{\mathrm{pw}}			
							\label{eq:ControlOfDeltaUh}
					\end{align} 
					based on Verf\"uhrt's bubble-function methodology \cite{Verf2013}; see \cref{sec:IntermediateContSource} for the proof of \eqref{eq:ControlOfDeltaUh}.   
					Abbreviate $c_2:=4\lambda^2\kappa_m^2 C_1C_2(1+C_3)C_4$ and $C_5:=2{C_1C_2}\big(2C_3+4\lambda \kappa_m^2 (1+C_3)C_4\big)$. 
					The combination of \eqref{eq:delta_uh_2}--\eqref{eq:ControlOfDeltaUh} controls $\Vert \delta \lambda u\Vert_{L^2(\Omega)}$ in 
					 \eqref{eq:proof_L2error_collection} and shows 
							\begin{align}
								\Vert u-u_{\mathrm{nc}}\Vert_{L^2(\Omega)}
								\le \frac{C_5}{2} h_{\max}^\sigma \vvvert u-Iu\vvvert_{\mathrm{pw}}
									+c_2h_{\max}^{2m}  \Vert u-u_{\mathrm{nc}}\Vert_{L^2(\Omega)}. \label{eq:towards_L2errorEnergyError}
							\end{align}
					The choice $\varepsilon_4:=\min\{\varepsilon_3, (2c_2)^{-1/2m}\}<1$ shows 
					$c_2h_{\max}^{2m} \Vert u-u_{\mathrm{nc}}\Vert_{L^2(\Omega)}\le \Vert u-u_{\mathrm{nc}}\Vert_{L^2(\Omega)}/2$ 
					for $\mathcal{T}\in\mathbb{T}(\varepsilon_4)$. This and 
					\eqref{eq:towards_L2errorEnergyError} show 
					the central estimate in \cref{thm:BoundInterpolationError}.\ref{item:BoundInterpolationError} 
					\begin{align}
					\Vert u-u_{\mathrm{nc}}\Vert_{L^2(\Omega)}\le C_5 h_{\max}^{\sigma } \vvvert u-Iu\vvvert_{\mathrm{pw}}. 				
						\label{eq:L2errorEnergyError}
					\end{align}
					Note that \eqref{eq:L2errorEnergyError} and 
					\ref{item:I_Pi0} imply the convergence $\Vert u-u_{\mathrm{nc}}\Vert_{L^2(\Omega)}\to 0$ as $h_{\max}\to 0$. 
					This and some $\varepsilon_5\le\varepsilon_4$  ensures $b(u,u_{\mathrm{nc}})>0$ for all 
					$\mathcal{T}\in\mathbb{T}(\varepsilon_5)$. 
					Based on this outline, it remains to prove \eqref{eq:u_zCR_L2}--\eqref{eq:u_Iu_Energy},  	
					 \eqref{eq:ControlOfDeltaUh}, and \eqref{eq:L2errorEnergyError}  and to identify $C_0,\dots, C_4$ below. 
					 The remaining estimates in  \cref{thm:BoundInterpolationError}.\ref{item:BoundInterpolationError} 
					follow  in \cref{sec:BestApproxFinish}. 
					 
				\subsection{Intermediate EVP}\label{sec:IntermediateEVP}
					Recall $\varepsilon_1:=\min\{\varepsilon_0, (2\lambda_{k+1}\kappa_m^2)^{-1/(2m)}\}$ and 
					that $(\lambda_h,\boldsymbol{u_h})$ denotes the ${k}$-th eigenpair of \eqref{eq:dis_EVP_alt} 
					with $\lambda_h\equiv\lambda_h(k)>0$, $\boldsymbol{u_h}=(u_{\mathrm{pw}},u_{\mathrm{nc}})\in\boldsymbol{V_h}$, 
						$\Vert u_{\mathrm{nc}}\Vert_{L^2(\Omega)}=1$, and $b(u,u_{\mathrm{nc}})\ge 0$. 
					Recall the intermediate EVP \eqref{eq:disEVP_intermediate}  and that  
					$(\lambda_h,u_{\mathrm{nc}})\in\mathbb{R}^+\times V(\mathcal{T})$ solves the rational EVP \eqref{eq:disEVP}. 
					
					\begin{remark}[$\Vert \bullet\Vert_{1+\delta}\approx \Vert \bullet\Vert_{L^2(\Omega)}$]\label{rem:Normequivalence}
							The weighted norm $\Vert \bullet\Vert_{1+\delta}$  is equivalent to the $L^2$-norm.  
							Since $\lambda_h \kappa_{m}^2 \varepsilon_1^{2{m}}< \lambda_{k+1} \kappa_{m}^2 \varepsilon_1^{2{m}}\le 1/2$ 
							and $1\le (1+\delta)|_T\le 2$ 
							for all $T\in\mathcal{T}\in\mathbb{T}(\varepsilon_1)$, 
							$\Vert v_{\mathrm{nc}}\Vert_{L^2(\Omega)}\le \Vert v_{\mathrm{nc}}\Vert_{1+\delta}
									\le \sqrt{2}\Vert v_{\mathrm{nc}}\Vert_{L^2(\Omega)}
							$ holds for any $v_{\mathrm{nc}}\in V({\mathcal{T}}) $.\hfill$\Box$
					\end{remark}
				
				\begin{lemma}[comparison of \eqref{eq:dis_EVP_alt} with \eqref{eq:disEVP_intermediate}]\label{lem:mu_is_fine}
					 	Given $\mathcal{T}\in\mathbb{T}(\varepsilon_1)$, let  
					 	$\lambda_h({j})$ denote the ${j}$-th eigenvalue of \eqref{eq:dis_EVP_alt}, and $\mu_{j}$ the  
					 	${j}$-th eigenvalue of \eqref{eq:disEVP_intermediate} for any ${j} =1,\dots,k+1$. Then 
						\begin{align}
					 		(1-\lambda_{k+1}\kappa_{m}^2 h_{\max}^{2{m}})\mu_{j} 	
					 			\le (1-\lambda_h(j)\kappa_{m}^2 h_{\max}^{2{m}})\mu_{j} 
					 			\le\lambda_h({j})&\le \mu_{j}+2\lambda_h^2\kappa_m^2 h_{\max}^{2m}.
					 			\label{eq:lambdah_mu} 
					 	\end{align}  
					 	The upper bound $\lambda_h({j})\le \mu_{j}+2\lambda_h^2\kappa_m^2 h_{\max}^{2m}$ holds for all ${j}=1,\dots, N$;  
					 	$N:=\textup{dim}\,V(\mathcal{T})$. 
				 \end{lemma}
				\begin{proofof}\textit{the upper bound}.
					Since the eigenfunctions $\phi_1,\dots,\phi_N$ of \eqref{eq:disEVP_intermediate} are 
					$(\bullet, \bullet)_{1+\delta}$-ortho-normal,  
					$a_{\mathrm{pw}}(\phi_{j},\phi_\ell)=\mu_{j} \delta_{{j}\ell}$ and $(\phi_{j},\phi_\ell)_{1+\delta}=\delta_{{j}\ell}$ for all 
					$j,\ell=1,\dots, N$. 
					 Set $\psi_{j}:=(1+\delta)\phi_{j}$ and 
					$\boldsymbol{U_{j}}:=\textup{span}\{(\psi_1,\phi_1),$ $\dots, (\psi_{j},\phi_{j})\}\subset \boldsymbol{V_h}$. 
					Since $b(\psi_j,\phi_\ell)=(\phi_j,\phi_\ell)_{1+\delta}=\delta_{j\ell}$,  
					the functions $\phi_1,\dots,\phi_N$ are linear independent and so $\textup{dim}(\boldsymbol{U_{j}})={j}$ for any ${j} =1,\dots,N$.  
					The discrete min-max principle \cite{StrangFix2008, Boffi2010} for the algebraic eigenvalue problem \eqref{eq:dis_EVP_alt} shows
					\begin{align}
						\lambda_h({j})\le \max_{\boldsymbol{v_h}\in \boldsymbol{U_{j}}\setminus\{0\}}
								{\boldsymbol{a_h}(\boldsymbol{v_h},\boldsymbol{v_h})}/{\boldsymbol{b_h}(\boldsymbol{v_h},\boldsymbol{v_h})}.
								\label{eq:lem_mu_is_fine_1}
					\end{align}
					The maximum in \eqref{eq:lem_mu_is_fine_1} is attained for some $\boldsymbol{v_h}=(\psi,\phi)\in \boldsymbol{U_{j}}\setminus\{0\}$ with 
					$\phi=\sum_{\ell=1}^{{j}}\alpha_\ell\phi_\ell\in V({\mathcal{T}})$, 
					$\psi=\sum_{\ell=1}^{{j}}\alpha_\ell\psi_\ell=(1+\delta)\phi\in P_{m}(\mathcal{T})$, and 
					$1=\Vert \phi\Vert_{1+\delta}^2=\sum_{\ell=1}^{{j}}\alpha_\ell^2$. 
					Then $\boldsymbol{b_h}(\boldsymbol{v_h},\boldsymbol{v_h})=\Vert (1+\delta)\phi\Vert_{L^2(\Omega)}^2\ge 1$ and 
					$
						\boldsymbol{a_h}(\boldsymbol{v_h},\boldsymbol{v_h})
								=\vvvert \phi\vvvert_{\mathrm{pw}}^2+ \Vert \kappa_{m}^{-1}h_{\mathcal{T}}^{-{m}}(\psi-\phi)\Vert_{L^2(\Omega)}^2.
					$ 
					Since $a_{\mathrm{pw}}(\phi_{j},\phi_\ell)=\mu_{j} \delta_{{j}\ell}$ for $\ell,{j}=1,\dots, N$,  
					$\sum_{\ell=1}^{{j}}\alpha_\ell^2=1$ implies $\vvvert \phi\vvvert_{\mathrm{pw}}^2=\sum_{\ell=1}^{{j}}\alpha_\ell^2\mu_\ell\le \mu_{j}$. 
					Since $\delta=\lambda_h\kappa_{m}^2h_{\mathcal{T}}^{2{m}}(1+\delta)$ a.e. in $\Omega$, the stabilization term in $\boldsymbol{a_h}$ reads 
					 \begin{align*}
						 &\Vert \kappa_{m}^{-1}h_{\mathcal{T}}^{-{m}}(\psi-\phi)\Vert_{L^2(\Omega)}^2
					 			=\Vert \kappa_{m}^{-1} h_{\mathcal{T}}^{-{m}} \delta\phi\Vert^2_{L^2(\Omega)}
					 			= \lambda_h^2\kappa_{m}^2\Vert h_{\mathcal{T}}^{m}(1+ \delta)\phi\Vert^2_{L^2(\Omega)}.
					 \end{align*}
					 The bound  $1+\delta\le 2$ from \cref{rem:delta} and $\Vert \phi\Vert_{1+\delta}=1$ imply
					 $\Vert h_{\mathcal{T}}^{m}(1+ \delta)\phi\Vert^2_{L^2(\Omega)}\le 2h_{\max}^{2{m}}$.  
					Consequently, 
					$\Vert \kappa_{m}^{-1}h_{\mathcal{T}}^{-{m}}(\psi-\phi)\Vert_{L^2(\Omega)}^2\le 2\lambda_h^2\kappa_m^2h_{\max}^{2{m}}$. 
					The substitution of the resulting estimates $\boldsymbol{b_h}(\boldsymbol{v_h},\boldsymbol{v_h})\ge 1$ and 
					$\boldsymbol{a_h}(\boldsymbol{v_h},\boldsymbol{v_h})\le\mu_j+ 2\lambda_h^2\kappa_m^2h_{\max}^{2{m}}$ in \eqref{eq:lem_mu_is_fine_1} 
				 	concludes the proof of 
					$\lambda_h(j)\le  \mu_{j}+2\lambda_h^2\kappa_m^2h_{\max}^{2{m}}$ 
					in  \eqref{eq:lambdah_mu} for $j=1, \dots, N$.\phantom{x} 
					\end{proofof}					
					\begin{proofof}\textit{the lower bound}.
					This situation is similar to \cite[Thm.~6.4]{CZZ18} and adapted below for completeness. 
					For $j=1,\dots, k+1$, let $(\lambda_h({j}),\boldsymbol{\phi_h}({j}))\in \mathbb{R}^+\times \boldsymbol{V_h}$ denote the first  
					$\boldsymbol{b_h}$-orthonormal eigenpairs of \eqref{eq:dis_EVP_alt} with 
					$\boldsymbol{\phi_h}({j})=(\phi_{\mathrm{pw}}({j}),\phi_{\mathrm{nc}}({j}))$. 
					The test functions $(v_{\mathrm{nc}},v_{\mathrm{nc}})\in V(\mathcal{T})\times V(\mathcal{T})\subset\boldsymbol{V_h}$ and  
					$(v_{\mathrm{pw}},0)\in \boldsymbol{V_h}$  in \eqref{eq:dis_EVP_alt} show
						\begin{align}
								a_{\mathrm{pw}}(\phi_{\mathrm{nc}}({j}),v_{\mathrm{nc}})=\lambda_h(j) b(\phi_{\mathrm{pw}}({j}),v_{\mathrm{nc}})\quad 
									\text{and}\quad 
								\phi_{\mathrm{pw}}({j})-\phi_{\mathrm{nc}}({j})=\lambda_h({j})\kappa_{m}^2h_{\mathcal{T}}^{2m}\phi_{\mathrm{pw}}({j}). 
								\label{eq:aquiv2}
					\end{align}
					For $\xi=(\xi_1,\dots,\xi_{j})\in\mathbb{R}^{j}$ with $\sum_{{\ell}=1}^{j}\xi_{\ell}^2=1$, set 
							\begin{align*}
								v_{\mathrm{nc}}:=\sum_{{\ell}=1}^{j} \xi_{\ell} \phi_{\mathrm{nc}}({\ell}),\quad  
								v_{\mathrm{pw}}:=\sum_{{\ell}=1}^{j} \xi_{\ell} \phi_{\mathrm{pw}}({\ell}),\quad   \text{ and } \quad  
								w_{\mathrm{pw}}:=\sum_{{\ell}=1}^{j} \xi_{\ell} \lambda_h(\ell)\phi_{\mathrm{pw}}({\ell}).
							\end{align*}
					Since $(\phi_{\mathrm{pw}}(\alpha),\phi_{\mathrm{pw}}(\beta))_{L^2(\Omega)}=\delta_{\alpha\beta}$ for $\alpha,\beta=1,\dots, k+1$, 
					$\Vert v_{\mathrm{pw}}\Vert_{L^2(\Omega)}=1$ and 
					$\Vert w_{\mathrm{pw}}\Vert_{L^2(\Omega)}=\sqrt{\sum_{{\ell}=1}^{j} \xi_{\ell}^2 \lambda_h({\ell})^2}\le \lambda_h({j})$. 
					The combination of this with \eqref{eq:aquiv2} and a Cauchy-Schwarz inequality leads to 
					$\vvvert v_{\mathrm{nc}}\vvvert_{\mathrm{pw}}^2=b(w_{\mathrm{pw}},v_{\mathrm{nc}})
					\le \lambda_h(j)\Vert v_{\mathrm{nc}}\Vert_{L^2(\Omega)}$ {and}
					$v_{\mathrm{pw}}-v_{\mathrm{nc}}=\kappa_{m}^2h_{\mathcal{T}}^{2m}w_{\mathrm{pw}}.$
					This and a reverse triangle inequality result in 
					\begin{align}
						0<1-\lambda_h(j)\kappa_{m}^2 h_{\max}^{2{m}}
						\le 1-\kappa_{m}^2 h_{\max}^{2m} \Vert w_{\mathrm{pw}}\Vert_{L^2(\Omega)} 
						\le \Vert v_{\mathrm{pw}}- \kappa_{m}^2h_{\mathcal{T}}^{2m}w_{\mathrm{pw}}\Vert_{L^2(\Omega)}
						=
						\Vert v_{\mathrm{nc}}\Vert_{L^2(\Omega)}.\label{eq:muisfine_vncL2}
					\end{align}
					This holds for all $v_{\mathrm{nc}}\in U_{{j}}:=\textup{span}\{\phi_{\mathrm{nc}}(1),\dots,\allowbreak\phi_{\mathrm{nc}}({j})\}
					\subset V({\mathcal{T}})$ with coefficients $(\xi_1,\dots,\xi_j)\in\mathbb{R}^j$ of Euclidean norm one. Hence  
					$\textup{dim}({U_{j}})={j}$ and the discrete min-max principle \cite{StrangFix2008, Boffi2010} for \eqref{eq:disEVP_intermediate} show 
					\begin{align}
						\mu_{j} \le \max_{v_{\mathrm{nc}}\in U_{j}\setminus\{0\}}
										{\vvvert v_{\mathrm{nc}}\vvvert^2_{\mathrm{pw}}}/{\Vert v_{\mathrm{nc}}\Vert_{1+\delta}^2}.
										\label{eq:lem_mu_is_fine_2}
					\end{align}
					Let  $v_{\mathrm{nc}}=\sum_{\ell=1}^{{j}}\alpha_\ell\phi_{\mathrm{nc}}(\ell)\in U_{{j}}$ denote a maximizer in 
					\eqref{eq:lem_mu_is_fine_2} with $\sum_{\ell=1}^{{j}}\alpha_\ell^2=1$. 
					The combination of $\vvvert v_{\mathrm{nc}}\vvvert^2_{\mathrm{pw}}
					\le \lambda_h(j)\Vert v_{\mathrm{nc}}\Vert_{L^2(\Omega)}$, 
					\eqref{eq:muisfine_vncL2}--\eqref{eq:lem_mu_is_fine_2}, and 
					 $\Vert v_{\mathrm{nc}}\Vert_{L^2(\Omega)}\le \Vert v_{\mathrm{nc}}\Vert_{1+\delta}$ from \cref{rem:Normequivalence}  provides 
					\begin{align*}
						\mu_{j} 
								\le \frac{\vvvert v_{\mathrm{nc}}\vvvert^2_{\mathrm{pw}}}{\Vert v_{\mathrm{nc}}\Vert_{1+\delta}^2}
								\le \frac{\vvvert v_{\mathrm{nc}}\vvvert^2_{\mathrm{pw}}}{\Vert v_{\mathrm{nc}}\Vert_{L^2(\Omega)}^2} 
								\le \frac{\lambda_h({j})}{1-\lambda_h(j)\kappa_{m}^2 h_{\max}^{2{m}}}.
					\end{align*} 
					Recall $\lambda_h(j)\le \lambda_h(k+1)\le \lambda_{k+1}$ from the lower bound property \eqref{eq:GLB_cond} to conclude the proof 
					of the associated lower bound for all $j=1,\dots, k$.  
				\end{proofof}
					 The subsequent corollaries adapt the notation $\mu_j,\, \lambda_h(j), \, \lambda_j$ from \cref{lem:mu_is_fine}. 
				\begin{corollary}\label{cor:convergenceMu}
					For any $j=1,\dots,k+1$, it holds $|\mu_{{j}}-\lambda_h({j})|+|\mu_{j} -\lambda_{j}|\to 0$ as  $h_{\max}\to 0$. 
				\end{corollary}
					 \begin{proof}
					 The a priori convergence analysis \cite[Thm.~1.2]{CP_Part1} implies 
					 $\lim_{h_{\max}\to 0}\lambda_h({j})\to\lambda_{j}$. 
					 \cref{lem:mu_is_fine} shows 
					 $|\lambda_h(j)-\mu_j|\le h_{\max}^{2m}\kappa_{m}^2\max\{ 2\lambda_h^2 ,  \lambda_h(j)\mu_{j}\}\to 0$ as $h_{\max}\to 0$.
					 \end{proof}
					 \begin{corollary}\label{rem:gapLambdaMu}
					 There exists $0<\varepsilon_2\le\min\{1/2,\varepsilon_1\}$ such that   \ref{item:H1}--\ref{item:H2} hold for 
					 $\mathcal{T}\in\mathbb{T}(\varepsilon_2)$. 
					  \end{corollary}
					 \begin{proof}
					 \cref{cor:convergenceMu} and  $\lambda_h=\lambda_h({k})\in \{\mu_1,\dots, \mu_N\}$ lead to $\varepsilon_a>0$ such 
					 that 
					  $\lambda_h=\lambda_h({k})=\mu_{k}$ has the correct index $k$ for all $\mathcal{T}\in\mathbb{T}(\varepsilon_a)$. 
					  It also leads to some $\varepsilon_b>0$ such that  $\mu_{{k}-1}<\mu_{k}<\mu_{{k}+1}$ 
					   for all $\mathcal{T}\in\mathbb{T}(\varepsilon_b)$. 
					Then $\varepsilon_2:= \min \{1/2,\varepsilon_1,\varepsilon_a,\varepsilon_b\}$  and  
					$\mathcal{T}\in\mathbb{T}(\varepsilon_2)$ imply  
					\ref{item:H1}--\ref{item:H2}.
					 \end{proof}
			\subsection{Proof of \eqref{eq:u_zCR_L2}--\eqref{eq:u_zCR_energy} for the $L^2$ error control}\label{sec:IntermediateSource}
					Recall $M_6$ from \ref{item:H2}, $\delta$ from \cref{rem:delta}, the norm equivalence from \cref{rem:Normequivalence}, 
					and the auxiliary source problem \eqref{eq:def_zCR}. 		
						\begin{proofof}{\textit{\eqref{eq:u_zCR_L2}.}}
							Recall the following straightforward result from \cite[Eq. (2.8)]{CGS15}: Any $u,v\in L^2(\Omega)$ with 
							$\Vert u\Vert_{L^2(\Omega)}=\Vert v\Vert_{L^2(\Omega)}=1$ 
							 satisfy 
				 			\begin{align*}
				 				\big(1+b(u,v)\big)\Vert u-v\Vert^2_{L^2(\Omega)}=2\min_{t\in\mathbb{R}}{\Vert u-tv\Vert_{L^2(\Omega)}^2}.
				 			\end{align*}
				 			This, a triangle inequality, 
							$t:=(z_{\mathrm{nc}},u_{\mathrm{nc}})_{1+\delta} \Vert \phi_{k}\Vert_{L^2(\Omega)}^2$, and 
							$v_{\mathrm{nc}}:=z_{\mathrm{nc}}-t u_{\mathrm{nc}}$ lead to 
							\begin{align}
								2^{-1/2}\Vert u-u_{\mathrm{nc}}\Vert_{L^2(\Omega)}\le \Vert u-t u_{\mathrm{nc}}\Vert_{L^2(\Omega)}
								\le\Vert u-z_{\mathrm{nc}}\Vert_{L^2(\Omega)}+\Vert v_{\mathrm{nc}}\Vert_{L^2(\Omega)}.\label{eq:lem_u_zCR_L2}
							\end{align}
							Since the eigenvectors $\phi_1,\dots,\phi_N$  of \eqref{eq:disEVP_intermediate} are  $(\bullet,\bullet)_{1+\delta}$-orthonormal 
							and form a basis of $V(\mathcal{T})$,  
							there exist Fourier coefficients $\alpha_1,\dots, \alpha_N\in\mathbb{R}$ with 
							$v_{\mathrm{nc}}=\sum_{j=1}^N \alpha_j\phi_j$ and $\Vert v_{\mathrm{nc}}\Vert_{1+\delta}^2=\sum_{{j}=1}^N\alpha_{j}^2$. 
							Since $(\lambda_h, u_{\mathrm{nc}})$ solves \eqref{eq:disEVP}, \ref{item:H1} implies $u_{\mathrm{nc}}\in\textup{span}\{\phi_k\}$ 
							with $\Vert u_{\mathrm{nc}}\Vert_{L^2(\Omega)}=1$. 
							Hence $u_{\mathrm{nc}}=\pm\phi_{k}/\Vert \phi_{k}\Vert_{L^2(\Omega)}$, 
							 $t=\pm (z_{\mathrm{nc}},\phi_{k})_{1+\delta}\Vert \phi_{k}\Vert_{L^2(\Omega)}$, and 
							$(u_{\mathrm{nc}},\phi_{k})_{1+\delta}=\pm\Vert\phi_{k}\Vert_{L^2(\Omega)}^{-1}$. Consequently, 
							\begin{align*}
								\alpha_{k}=(v_{\mathrm{nc}},\phi_{k})_{1+\delta}
												=(z_{\mathrm{nc}},\phi_{k})_{1+\delta}-t (u_{\mathrm{nc}},\phi_{k})_{1+\delta}
											=0.
							\end{align*}
							Since $(u_{\mathrm{nc}},\phi_j)_{1+\delta}=0$ for all $j=1,\dots, N$ with $j\not =k$, 
							$ \alpha_{j}	=(v_{\mathrm{nc}},\phi_{j})_{1+\delta}=(z_{\mathrm{nc}},\phi_{j})_{1+\delta}$.
							Since $\phi_j$ is an eigenvector in \eqref{eq:disEVP_intermediate} and $z_{\mathrm{nc}}$ solves
							\eqref{eq:def_zCR}, it follows 
							\begin{align*}
								\alpha_{j}	=(z_{\mathrm{nc}},\phi_{j})_{1+\delta}
								=\frac{1}{\mu_{j}}a_{\mathrm{pw}}(z_{\mathrm{nc}},\phi_{{j}})
								=\frac{\lambda}{\mu_{j}}(u,\phi_{{j}})_{1+\delta}.
							\end{align*}
							Hence $(u-z_{\mathrm{nc}},\phi_{j})_{1+\delta}=(\mu_{j}/\lambda-1)\alpha_{j}$.
							These values for the coefficients $\alpha_j$ 
							and the separation condition \ref{item:H2}  imply
							\begin{align*}
											\Vert v_{\mathrm{nc}}\Vert_{1+\delta}^2
											=\sum_{ {j} \ne {k}}\alpha_{j}^2 
											=\sum_{ {j} \ne {k}} \Big\vert \frac{\lambda}{\mu_{j}-\lambda}\Big\vert \vert\alpha_{{j}}\vert
															\vert (u-z_{\mathrm{nc}},\phi_{j})_{1+\delta} \vert
											\le M_6 \sum_{{j} \ne {k}}
															(u-z_{\mathrm{nc}},\alpha^\prime_{{j}}\phi_{j})_{1+\delta}
							\end{align*}
							for a sign in $\alpha_{j}^\prime\in\{\pm\alpha_{j}\}$ such that 
							$ \vert (u-z_{\mathrm{nc}},\alpha_{{j}}\phi_{j})_{1+\delta}\vert
							=(u-z_{\mathrm{nc}},\alpha_{j}^\prime\phi_{j})_{1+\delta}$ and with the abbreviation $\sum_{j\ne k}=\sum_{{{j}=1, {j} \ne {k}}}^N$.
							This and a Cauchy-Schwarz inequality show 
							\begin{align*}
											M_6^{-1}\Vert v_{\mathrm{nc}}\Vert_{1+\delta}^2
											&\le  \Big(u-z_{\mathrm{nc}},\sum_{ {j} \ne {k}}\alpha_{j}^\prime\phi_{j}\Big)_{1+\delta}
											\le \Vert u-z_{\mathrm{nc}}\Vert_{1+\delta} \Vert v_{\mathrm{nc}}\Vert_{1+\delta}.  
							\end{align*} 
							The norm equivalence  in \cref{rem:Normequivalence}  proves
							$\Vert v_{\mathrm{nc}}\Vert_{L^2(\Omega)}\le\Vert v_{\mathrm{nc}}\Vert_{1+\delta}
							\le \sqrt{2}M_6\Vert u-z_{\mathrm{nc}}\Vert_{L^2(\Omega)}$. 
							This and \eqref{eq:lem_u_zCR_L2} conclude the proof of \eqref{eq:u_zCR_L2} with $C_1:=\sqrt{2}(1+\sqrt{2} M_6)$. 
						\end{proofof}
						
							\begin{proofof}{\textit{\eqref{eq:u_zCR_energy}.}}
								Given the solution $z_{\mathrm{nc}}\in V(\mathcal{T})$ to \eqref{eq:def_zCR}, let ${w}\in V:=H^m_0(\Omega)$ solve 
								\begin{align}
									a({w},\varphi)=b(u-Jz_{\mathrm{nc}},\varphi)\quad \text{for all }\varphi \in V.  \label{eq:def_v}
								\end{align}
								Since $u-Jz_{\mathrm{nc}}\in V\subset L^2(\Omega)$, the elliptic regularity  \eqref{eq:def_sigma} guarantees 
								${w}\in H^{{m}+\sigma}(\Omega)$ and 
								\begin{align}
									\Vert {w}\Vert_{H^{{m}+\sigma}(\Omega)}\le C(\sigma) \Vert u-Jz_{\mathrm{nc}}\Vert_{L^2(\Omega)}.\label{eq:def_alpha}
								\end{align}
								The combination of \eqref{eq:def_alpha} with \cref{cor:InterpolationOperator}.\ref{item:cor_I_alpha} shows
								\begin{align}
									\vvvert {w}-I{w}\vvvert_{\mathrm{pw}}\le (h_{\max}/\pi)^{\sigma}\Vert {w}\Vert_{H^{{m}+\sigma}(\Omega)}
												\le C(\sigma)(h_{\max}/\pi)^{\sigma}\Vert u-Jz_{\mathrm{nc}}\Vert_{L^2(\Omega)}.\label{eq:v-Iv_alpha}
								\end{align}
								The test function 
								$\varphi=u-Jz_{\mathrm{nc}}$  in the auxiliary problem \eqref{eq:def_v} leads to 
								\begin{align}
									\Vert u-Jz_{\mathrm{nc}}\Vert_{L^2(\Omega)}^2
											&= a(u, {w}-JI{w})+a_{\mathrm{pw}}({w}, z_{\mathrm{nc}}-Jz_{\mathrm{nc}})
											+a(u, JI{w})-a_{\mathrm{pw}}({w},z_{\mathrm{nc}}). 
											\label{eq:lem_u_zCR_energy_sum}
								\end{align}
								Since \ref{item:J_rightInverse} asserts $I({w}-JI{w})=0$, 
								 \cref{cor:InterpolationOperator}.\ref{item:cor_I_a} and   a triangle inequality show 
								 \begin{align*}
									a(u,{w}-JI{w}) &= a_{\mathrm{pw}}(u, (1-I)({w}-JI{w}))
									\le \vvvert u-z_{\mathrm{nc}}\vvvert_{\mathrm{pw}}
										(\vvvert {w}-I{w}\vvvert_{\mathrm{pw}}+\vvvert I{w}-JI{w}\vvvert_{\mathrm{pw}}).
								\end{align*}
								Then \ref{item:J_approx} 
								implies that 
								$
									a(u,{w}-JI{w}) \le (1+{{M_2}}) \vvvert {w}-I{w}\vvvert_{\mathrm{pw}}\vvvert u-z_{\mathrm{nc}}\vvvert_{\mathrm{pw}}.
								$
								\cref{cor:ConformingCompanion}.\ref{item:cor_J_a_vh} 
								proves for the second term in the right-hand side of \eqref{eq:lem_u_zCR_energy_sum} that  
								\begin{align*}
									a_{\mathrm{pw}}({w}, z_{\mathrm{nc}}-Jz_{\mathrm{nc}}) 
										&\le M_2\vvvert {w}-I{w}\vvvert_{\mathrm{pw}}\vvvert u-z_{\mathrm{nc}}\vvvert_{\mathrm{pw}}.
								\end{align*}
								\cref{cor:InterpolationOperator}.\ref{item:cor_I_a} ensures 
								$a_{\mathrm{pw}}({w},z_{\mathrm{nc}})=a_{\mathrm{pw}}(I{w},z_{\mathrm{nc}})$.
								Since $(\lambda, u)$ is an eigenpair of \eqref{eq:contEVP} and $z_{\mathrm{nc}}$ satisfies \eqref{eq:def_zCR}, this implies 
								\begin{align*}
									a(u, JI{w})-a_{\mathrm{pw}}({w},z_{\mathrm{nc}})&=b(\lambda u, JI{w})-a_{\mathrm{pw}}(I{w},z_{\mathrm{nc}})
											=\lambda b(u, JI{w} -I{w}- \delta I{w}). 
								\end{align*}
								\cref{cor:ConformingCompanion}.\ref{item:cor_J_b_vh} 
								shows 
								$
									 b( u, JI{w} -I{w})
										\le {{M_2}}\kappa_{m}^{2} h_{\max}^{2{m}} \vvvert u-z_{\mathrm{nc}}\vvvert_{\mathrm{pw}} 
										\vvvert {w}-I{w}\vvvert_{\mathrm{pw}}.
								$	
								The discrete Friedrichs inequality 
								\begin{align}
								\Vert v_{\mathrm{nc}}\Vert_{L^2(\Omega)}\le C_{\mathrm{dF}}\vvvert v_{\mathrm{nc}}\vvvert_{\mathrm{pw}}
								\text{ for all } v_{\mathrm{nc}}\in V(\mathcal{T})\text{ with } C_{\mathrm{dF}}:=C_{F}(1+{M_2})+{M_2} h^m_{\max}
								\label{eq:disFriedrich}
								\end{align}
								is a direct consequence of the Friedrichs inequality 
								$\Vert v\Vert_{L^2(\Omega)}\le C_F\vvvert v\vvvert${ for any }$v\in V$ 
			 					and \ref{item:J_approx}; cf.\ \cite[Cor.~4.11]{CH17} for details in case $m=1$; the proof for $m=2$ is analogous. 								
								This, \ref{item:I_Pi0}, and the boundedness of $\Pi_0$ imply 
								$
									C_{\mathrm{dF}}^{-1}\Vert I {w}\Vert_{L^2(\Omega)}\le \vvvert Iw\vvvert_{\mathrm{pw}}=\Vert \Pi_0D^m w\Vert_{L^2(\Omega)}
																\le \Vert {w}\Vert_{H^{m}(\Omega)}.
								$	
							  	The Cauchy-Schwarz inequality leads to  
								\begin{align*}
									- b(\lambda u, \delta I{w})\le \Vert \delta \lambda u\Vert_{L^2(\Omega)} \Vert I{w}\Vert_{L^2(\Omega)}
									\le  C_{\mathrm{dF}}\Vert \delta\lambda u\Vert_{L^2(\Omega)} \Vert {w}\Vert_{H^{{m}+\sigma}(\Omega)}.
								\end{align*}
								This bounds the last term on the right-hand side of \eqref{eq:lem_u_zCR_energy_sum}. The substitution in 
								\eqref{eq:lem_u_zCR_energy_sum} and $\lambda\kappa_m^2 h_{\max}^{2m}\le 1/2$ result in 
								\begin{align*}
									\Vert u-Jz_{\mathrm{nc}}\Vert_{L^2(\Omega)}^2
									\le & (1+5M_2/2)
										\vvvert {w}-I{w}\vvvert_{\mathrm{pw}}\vvvert u-z_{\mathrm{nc}}\vvvert_{\mathrm{pw}}
										+  C_{\mathrm{dF}}\Vert \delta\lambda u\Vert_{L^2(\Omega)} 
										\Vert {w}\Vert_{H^{{m}+\sigma}(\Omega)}. 
								\end{align*}
								This and \eqref{eq:def_alpha}--\eqref{eq:v-Iv_alpha} imply 
								$$
									C(\sigma)^{-1}\Vert u-Jz_{\mathrm{nc}}\Vert_{L^2(\Omega)}
									\le (h_{\max}/\pi)^{\sigma}(1+5M_2/2)
									\vvvert u-z_{\mathrm{nc}}\vvvert_{\mathrm{pw}}
															+C_{\mathrm{dF}}\Vert \delta\lambda u\Vert_{L^2(\Omega)}. 
								$$
								\cref{cor:ConformingCompanion}.\ref{item:cor_J_L2_vh} implies 
								$
									\Vert z_{\mathrm{nc}}-Jz_{\mathrm{nc}}\Vert_{L^2(\Omega)}\le {{M_2}}\kappa_{m} h_{\max}^{m}
									\vvvert u-z_{\mathrm{nc}}\vvvert_{\mathrm{pw}}.
								$
								This, $0<\sigma\le 1\le m$, $h_{\max}<1$,  
								 and a triangle inequality show
								\begin{align*}
								\Vert u-z_{\mathrm{nc}}\Vert_{L^2(\Omega)}
									&\le \Vert Jz_{\mathrm{nc}}-z_{\mathrm{nc}}\Vert_{L^2(\Omega)}+\Vert u-Jz_{\mathrm{nc}}\Vert_{L^2(\Omega)}
									\le C_2 \big(h_{\max}^{\sigma }\vvvert u-z_{\mathrm{nc}}\vvvert_{\mathrm{pw}}
									+\Vert \delta \lambda u\Vert_{L^2(\Omega)}\big)
								\end{align*}
								with the constant 
								$
									C_2:=\max\big\{C(\sigma)(1+5M_2/2)/\pi^{\sigma}+{{M_2}} \kappa_{m},C(\sigma)C_{\mathrm{dF}}\big\}.							
								$ 
							\end{proofof} 
			\subsection{Proof of \eqref{eq:u_Iu_Energy} and \eqref{eq:ControlOfDeltaUh} for the energy error control }\label{sec:IntermediateContSource}
					Recall $\delta$ from \cref{rem:delta} and 
					that  $z_{\mathrm{nc}}\in V({\mathcal{T}}) $ solves  \eqref{eq:def_zCR}.	
					 	\begin{proofof}\textit{\eqref{eq:u_Iu_Energy}.}		
					 		Elementary algebra  with $a_{\mathrm{pw}}(z_{\mathrm{nc}},u)=a_{\mathrm{pw}}(z_{\mathrm{nc}},Iu)$ 
					 		from \cref{cor:InterpolationOperator}.\ref{item:cor_I_a}  shows 		
					 		\begin{align}
					 			\vvvert u-z_{\mathrm{nc}}\vvvert^2_{\mathrm{pw}}
					 	 			=& a (u, u-JIu)+a_{\mathrm{pw}}(u, Jz_{\mathrm{nc}}-z_{\mathrm{nc}})+ a(u,JIu-Jz_{\mathrm{nc}})
					 	 			+a_{\mathrm{pw}}(z_{\mathrm{nc}},z_{\mathrm{nc}}- Iu).\label{eq:lem_u_Iu_Energy_sum} 
					 		\end{align}
					 		\cref{cor:InterpolationOperator}.\ref{item:cor_I_a} and 
					 		\cref{cor:ConformingCompanion}.\ref{item:cor_J_a_vh}  
					 		control the terms in the decomposition 
					 		 \begin{align*}
							 	 a (u, u-JIu)+ a_{\mathrm{pw}}(u, Jz_{\mathrm{nc}}-z_{\mathrm{nc}})
							 	 &=a_{\mathrm{pw}}(u,u-Iu)+a_{\mathrm{pw}}(u,Iu-JIu)+ a_{\mathrm{pw}}(u, Jz_{\mathrm{nc}}-z_{\mathrm{nc}})\\
							 	 &\le (1+{{M_2}}) \vvvert u-Iu\vvvert_{\mathrm{pw}}^2+{{M_2}} \vvvert u-Iu\vvvert_{\mathrm{pw}}
							 	 \vvvert u-z_{\mathrm{nc}}\vvvert_{\mathrm{pw}}.
							 \end{align*}
					 		Recall that $(\lambda,u)$ is an eigenpair of \eqref{eq:contEVP} and   
					 	    $z_{\mathrm{nc}}$ satisfies \eqref{eq:def_zCR}. Consequently,
					 	 	\begin{align*}
					 	 		 a(u,JIu-Jz_{\mathrm{nc}})+a_{\mathrm{pw}}(z_{\mathrm{nc}}, z_{\mathrm{nc}}-Iu)
					 	 			&= b(\lambda u, JIu-Jz_{\mathrm{nc}}+ (1+\delta)(z_{\mathrm{nc}}-Iu))\\
					 	 			&=\lambda b( u, (J-1)(Iu-z_{\mathrm{nc}}))+\lambda b(\delta u, z_{\mathrm{nc}}-Iu).
					 		 \end{align*}
							 \cref{cor:ConformingCompanion}.\ref{item:cor_J_b_vh}, $\kappa_m^2\lambda h_{\max}^{2m}\le 1/2$, and a triangle inequality show
							 \begin{align*}
							 	\lambda b(u,(J-1)(Iu-z_{\mathrm{nc}}))
							 		\le M_2/2\,\vvvert u-Iu\vvvert_{\mathrm{pw}}(\vvvert u-Iu\vvvert_{\mathrm{pw}}
							 		+\vvvert u-z_{\mathrm{nc}}\vvvert_{\mathrm{pw}}) .
							 \end{align*}
							 Since Cauchy-Schwarz and triangle inequalities show 
							 $
							 	b(\delta \lambda u, z_{\mathrm{nc}}-Iu)
							 			\le \Vert\delta \lambda  u \Vert_{L^2(\Omega)}
							 			(\Vert u-z_{\mathrm{nc}} \Vert_{L^2(\Omega)}+\Vert u-Iu\Vert_{L^2(\Omega)}),$
							\ref{item:I_kappa} provides the first and \eqref{eq:u_zCR_energy} the second estimate in 
							 \begin{align*}
							 	b(\delta \lambda  u, z_{\mathrm{nc}}-Iu)
							 	\le& \Vert\delta \lambda  u \Vert_{L^2(\Omega)} (\Vert u-z_{\mathrm{nc}} \Vert_{L^2(\Omega)}
							 				+\kappa_m h_{\max}^{m}\vvvert u-Iu\vvvert_{\mathrm{pw}})
							 	\\
							 	\le&\Vert\delta \lambda  u \Vert_{L^2(\Omega)} 
							 		(C_2 h_{\max}^{\sigma}\vvvert u-z_{\mathrm{nc}}\vvvert_{\mathrm{pw}}+C_2\Vert\delta \lambda  u \Vert_{L^2(\Omega)}
							 				+\kappa_m h_{\max}^{m}\vvvert u-Iu\vvvert_{\mathrm{pw}}).
							 \end{align*}
							Since  $h_{\max}^{m}\vvvert u-Iu\vvvert_{\mathrm{pw}}\le h_{\max}^{\sigma }\vvvert u-z_{\mathrm{nc}}\vvvert_{\mathrm{pw}}$ 
							from  \cref{cor:InterpolationOperator}.\ref{item:cor_I_orthogonality}, a weighted Young inequality shows 
							 $
							 	b(\delta\lambda   u, z_{\mathrm{nc}}-Iu)
							 	\le ((C_2+\kappa_m)^2h_{\max}^{2\sigma }+C_2)\Vert \delta \lambda u\Vert_{L^2(\Omega)}^2
							 	+\vvvert u-z_{\mathrm{nc}}\vvvert_{\mathrm{pw}}^2/4.
							 $
							 The substitution of the displayed estimates in \eqref{eq:lem_u_Iu_Energy_sum} shows 
							 \begin{align*}
							 	\vvvert u-z_{\mathrm{nc}}\vvvert^2_{\mathrm{pw}}
							 		\le & (1+3M_2/2) 
							 		\vvvert u-Iu\vvvert_{\mathrm{pw}}^2+
							 		3 M_2/2\,
							 		\vvvert u-Iu\vvvert_{\mathrm{pw}}\vvvert u-z_{\mathrm{nc}}\vvvert_{\mathrm{pw}} 
							 	\\&+	((C_2+\kappa_m)^2h_{\max}^{2\sigma }+C_2)\Vert \delta \lambda u\Vert_{L^2(\Omega)}^2
							 	+\vvvert u-z_{\mathrm{nc}}\vvvert_{\mathrm{pw}}^2/4.
							 \end{align*}		
							This and 
							$3M_2/2\vvvert u-Iu\vvvert_{\mathrm{pw}}\vvvert u-z_{\mathrm{nc}}\vvvert_{\mathrm{pw}} 	
								\le 9M_2^2/4	\vvvert u-Iu\vvvert_{\mathrm{pw}}^2+	 \vvvert u-z_{\mathrm{nc}}\vvvert_{\mathrm{pw}} ^2/4$ 			
				 			conclude the proof of \eqref{eq:u_Iu_Energy}  with 
							 		$C_3^2:=2 \max\{1+3{{M_2}}/2+9{{M_2}}^2/4,
							 				(C_2+\kappa_m)^2h_{\max}^{2\sigma }+C_2\}$. 
					 	\end{proofof}

					 \begin{proofof}\textit{\eqref{eq:ControlOfDeltaUh}.} 
					 	The proof of the efficiency estimate of the volume residual 
					 	is based on Verf\"uhrt's bubble-function methodology \cite{Verf2013}, comparable to 
					 	\cite[Thm.~2]{BdVNS07}, \cite[Prop.~3.1]{Gal15_cluster}, and given here for completeness. 
					 	Let $\varphi_z\in S^1(\mathcal{T}):=P_1(\mathcal{T})\cap C(\Omega)$ denote 
					 	the nodal basis function associated with the vertex $z\in\mathcal{V}$. 
					 	For any $T\in\mathcal{T}$, 
					 	let $b_T:=4^{4m}\prod_{z\in\mathcal{V}(T)}\varphi_z^m\in P_{4m}(T)\cap W^{m,\infty}_0(T)\subset V$ 
					 	denote the volume-bubble-function with $\textup{supp}(b_T)=T$ and $\Vert b_T\Vert_{\infty}=1$.
					 	An inverse estimate  $\Vert p\Vert_{L^2(T)}\le c_b \Vert p\Vert_{b_T}$ for any polynomial $p\in P_m(T)$ leads to 
					 	 \begin{align}
					 	 	c_b^{-2}\Vert u_{\mathrm{nc}}\Vert_{L^2(T)}^2
					 	 	\le  \Vert  u_{\mathrm{nc}}\Vert_{b_T}^2 
					 	 	= (u_{\mathrm{nc}},u)_{b_T}-( u_{\mathrm{nc}},u-u_{\mathrm{nc}})_{b_T}.\label{eq:lem_ControlOfDeltaUh}
					 	 \end{align}
					 	The Cauchy-Schwarz inequality and $\Vert b_T\Vert_{\infty}=1$ show 
					 	$
					 		( u_{\mathrm{nc}},u-u_{\mathrm{nc}})_{b_T}
					 			\le \Vert u_{\mathrm{nc}}\Vert_{L^2(T)}\Vert u-u_{\mathrm{nc}}\Vert_{L^2(T)}.
					 	$
					 	An integration by parts proves 
					 	 $
					 	 	\int_T D^{m}(b_T u_{\mathrm{nc}})\,\textup{d}x=0
					 	 $ since $b_T u_{\mathrm{nc}}\in H^{m}_0(T)$, i.e.,
					 	 $D^{m}b_Tu_{\mathrm{nc}}$ is $L^2$-orthogonal to $P_0(T)$.
					 	 Recall that $(\lambda,u)$ is an eigenpair of \eqref{eq:contEVP} and the support of $b_Tu_{\mathrm{nc}}$ is $T$. 
					 	 This, \ref{item:I_Pi0}, and the Cauchy-Schwarz inequality result in  
					 	\begin{align*}
					 		\lambda b(u,b_{T}u_{\mathrm{nc}})
					 			&=a_{\mathrm{pw}}(u,b_{T}u_{\mathrm{nc}})=(D^{m} u, D^{m} (b_T u_{\mathrm{nc}}))_{L^2(T)}
					 			\le \vert u-Iu\vert_{H^m(T)}\vert b_T u_{\mathrm{nc}}\vert_{H^m(T)}.
					 	\end{align*}
						An inverse estimate for polynomials in $ P_{5m}(T)$ with the constant 
						$c_{\mathrm{inv}}$ and the boundedness of $b_T$
						show $\lambda b(u,b_{T}u_{\mathrm{nc}}) \le c_{\mathrm{inv}} h_T^{-m} \vert u-Iu\vert_{H^m(T)}\Vert u_{\mathrm{nc}}\Vert_{L^2(T)}$. 	
						This  provides $c_b^{-2}h_T^{m}\Vert u_{\mathrm{nc}}\Vert_{L^2(T)} 
						\le h_T^{m} \Vert u-u_{\mathrm{nc}}\Vert_{L^2(T)}+c_{\mathrm{inv}}\lambda^{-1}\vert u-Iu\vert_{H^m(T)}$ for all $T\in\mathcal{T}$ 
						in \eqref{eq:lem_ControlOfDeltaUh}. 
						The sum 
						 over all $T\in \mathcal{T}$ concludes the proof of \eqref{eq:ControlOfDeltaUh} with $C_4=c_b^2\max\{1,c_{\mathrm{inv}}\}$.
					 \end{proofof} 
					\subsection{Proof of \cref{thm:BoundInterpolationError}.\ref{item:BoundInterpolationError}}\label{sec:BestApproxFinish}
					\begin{proofof}\textit{\eqref{eq:L2errorEnergyError} for $\varepsilon_4>0$.} 
							Recall $c_1:=2\lambda^2\kappa_m^2 C_1C_2(1+C_3)$ and  \eqref{eq:proof_L2error_collection} as a result of 
							\eqref{eq:u_zCR_L2}--\eqref{eq:u_Iu_Energy}. 
							A triangle inequality, \cref{rem:delta}, and \eqref{eq:proof_L2error_collection} show 
								\begin{align*}
									\Vert \delta \lambda u\Vert_{L^2(\Omega)}
									&\le 2\lambda^2\kappa_m^2 h_{\max}^{2m}\Vert u-u_{\mathrm{nc}}\Vert_{L^2(\Omega)}+
									\Vert \delta \lambda u_{\mathrm{nc}}\Vert_{L^2(\Omega)}
									\notag\\&\le \frac{c_1C_3h_{\max}^{2m}}{1+C_3}h_{\max}^{\sigma} \vvvert u-Iu\vvvert_{\mathrm{pw}}
									+c_1h_{\max}^{2m}\Vert \delta \lambda u\Vert_{L^2(\Omega)}
									+\Vert \delta \lambda u_{\mathrm{nc}}\Vert_{L^2(\Omega)}.
								\end{align*}
							Since $0<\varepsilon_3:=\min\{\varepsilon_2,(2c_1)^{-1/2m}\}$ ensures $c_1h_{\max}^{2m}\le 1/2$ for all $\mathcal{T}\in\mathbb{T}(\varepsilon_3)$,  
							the previous displayed estimate reads 
							$\Vert \delta \lambda u\Vert_{L^2(\Omega)}
									\le \frac{c_1C_3h_{\max}^{2m}}{1+C_3}h_{\max}^{\sigma} \vvvert u-Iu\vvvert_{\mathrm{pw}}
									+\Vert \delta \lambda u\Vert_{L^2(\Omega)}/2
									+\Vert \delta \lambda u_{\mathrm{nc}}\Vert_{L^2(\Omega)}.$  
							This implies \eqref{eq:delta_uh_2}.  
							The bound \eqref{eq:delta_uh_2} for $ \Vert \delta \lambda u\Vert_{L^2(\Omega)}$ recasts \eqref{eq:proof_L2error_collection} as 
							\begin{align*}
								C_1^{-1}C_2^{-1}\Vert u-u_{\mathrm{nc}}\Vert_{L^2(\Omega)}
									\le 2C_3 h_{\max}^\sigma \vvvert u-Iu\vvvert_{\mathrm{pw}}
									+2(1+C_3)\lambda \Vert \delta u_{\mathrm{nc}}\Vert_{L^2(\Omega)}.	
							\end{align*}							 
							\cref{rem:delta} and \eqref{eq:ControlOfDeltaUh} control the last term in 
							\begin{align*}
								(2\kappa_m^2 C_4)^{-1}\Vert \delta u_{\mathrm{nc}}\Vert_{L^2(\Omega)}
								\le C_4^{-1}\lambda h_{\max}^m \Vert h_{\mathcal{T}}^{m} u_{\mathrm{nc}}\Vert_{L^2(\Omega)}
								\le \lambda h_{\max}^{2m}\Vert u-u_{\mathrm{nc}}\Vert_{L^2(\Omega)}
										+h_{\max}^m\vvvert u-Iu\vvvert_{\mathrm{pw}}.
							\end{align*}
							 Recall that $c_2:=4\lambda^2\kappa_m^2 C_1C_2(1+C_3)C_4$ and  
							 $\varepsilon_4:=\min\{\varepsilon_3, (2c_2)^{-1/2m}\}<1$ ensure $c_2 h_{\max}^{2m}\le 1/2$. 
							 Hence the last term in \eqref{eq:towards_L2errorEnergyError} is 
							 $\le\Vert u-u_{\mathrm{nc}}\Vert_{L^2(\Omega)}	/2$  and can be absorbed.  
							This concludes the proof 	
							of \eqref{eq:L2errorEnergyError}{ with }$C_5:=2{C_1C_2}
									\big(2C_3+4\kappa_m^2 \lambda (1+C_3)C_4\big).	
							$\phantom{x}
					\end{proofof}	
					Recall $0<\varepsilon_5\le\varepsilon_4$ such that $b(u,u_{\mathrm{nc}})>0$ for any $\mathcal{T}\in\mathbb{T}(\varepsilon_5)$. 
					\begin{proofof}\textit{\cref{thm:BoundInterpolationError}.\ref{item:BoundInterpolationError} for $\varepsilon_5$.} 
							Recall $\lambda_h\le \lambda$ and $\Vert u\Vert_{L^2(\Omega)}=\Vert u_{\mathrm{nc}}\Vert_{L^2(\Omega)}=1$. 
							The continuous eigenpair $(\lambda,u)$ in \eqref{eq:contEVP} satisfies $\lambda=\vvvert u\vvvert^2$.  
							The discrete eigenpair $(\lambda_h,u_{\mathrm{nc}})$ solves \eqref{eq:disEVP} and so 
							$\lambda_h=\vvvert u_{\mathrm{nc}}\vvvert_{\mathrm{pw}}^2/\Vert u_{\mathrm{nc}}\Vert_{1+\delta}^2$ 
							with $\Vert u_{\mathrm{nc}}\Vert_{L^2(\Omega)}=1$. 
							Then 
							\begin{align*}
								\vvvert u-u_{\mathrm{nc}}\vvvert_{\mathrm{pw}}^2
								=\lambda-2a_{\mathrm{pw}}(u,u_{\mathrm{nc}})+\lambda_h \Vert u_{\mathrm{nc}}\Vert_{1+\delta}^2
								\quad\text{and}\quad 
								\Vert u_{\mathrm{nc}}\Vert^2_{1+\delta}-1
								=b(\delta u_{\mathrm{nc}},u_{\mathrm{nc}})=\Vert u_{\mathrm{nc}}\Vert_{\delta}^2.
							\end{align*}				
							This and elementary algebra show for 
							the left-hand side of 
							\cref{thm:BoundInterpolationError}.\ref{item:BoundInterpolationError} that 
							\begin{align*}
								\textup{LHS}:=\lambda -\lambda_h+\vvvert u-u_{\mathrm{nc}}\vvvert_{\mathrm{pw}}^2
												+\Vert u_{\mathrm{nc}}\Vert_{\delta}^2
									&= 2\lambda -2 a_{\mathrm{pw}}(u,u_{\mathrm{nc}}) +(1+\lambda_h)
									\Vert u_{\mathrm{nc}}\Vert_{\delta}^2.
							\end{align*}
							 Since $u$ is the eigenfunction in \eqref{eq:contEVP} and $2b(u,u-u_{\mathrm{nc}})=\Vert u-u_{\mathrm{nc}}\Vert_{L^2(\Omega)}^2$ 
							 from $\Vert u_{\mathrm{nc}}\Vert_{L^2(\Omega)}=1=\Vert u\Vert_{L^2(\Omega)}$, 
							it follows  
							\begin{align*}
								\lambda 
								=\lambda b(u,u_{\mathrm{nc}})+\lambda b(u,u-u_{\mathrm{nc}})
								= \lambda b(u,u_{\mathrm{nc}}-Ju_{\mathrm{nc}})+a_{\mathrm{pw}}(u, Ju_{\mathrm{nc}})
								+\lambda/2\ \Vert u-u_{\mathrm{nc}}\Vert^2_{L^2(\Omega)}. 
							\end{align*}
							The combination of the last two displayed identities eventually leads to 
							\begin{align}
								\textup{LHS}
									=&(1+\lambda_h)\Vert u_{\mathrm{nc}}\Vert_{\delta}^2+  \lambda \Vert u-u_{\mathrm{nc}}\Vert^2_{L^2(\Omega)}
									+ 2\lambda b(u,u_{\mathrm{nc}}-Ju_{\mathrm{nc}})
									 +2a_{\mathrm{pw}}(u, Ju_{\mathrm{nc}}-u_{\mathrm{nc}}).
									 \label{eq:BIE_inbetween}
							\end{align}
							Recall $2\lambda \kappa_m^2h_{\max}^{2m}\le 1$. 
							The combination of \cref{rem:delta} and  \eqref{eq:ControlOfDeltaUh} implies that  
							\begin{align*}
								\Vert u_{\mathrm{nc}}\Vert_{\delta}
								&\le 	\sqrt{2}\kappa_m \lambda^{1/2}\Vert h_{\mathcal{T}}^m u\Vert_{L^2(\Omega)}
								\le  C_4\Vert u -u_{\mathrm{nc}}\Vert_{L^2(\Omega)}+\sqrt{2/\lambda}\kappa_m C_4\vvvert u-Iu\vvvert_{\mathrm{pw}}
							\end{align*}
						 	and \eqref{eq:L2errorEnergyError} controls $ \Vert u-u_{\mathrm{nc}}\Vert_{L^2(\Omega)}
								\le C_5h_{\max}^{\sigma}\vvvert u-Iu\vvvert_{\mathrm{pw}}.$ 
							\cref{cor:ConformingCompanion}.\ref{item:cor_J_b_vh} asserts 
							$
								2\lambda b(u,u_{\mathrm{nc}}-Ju_{\mathrm{nc}})
								\le M_2
								\vvvert u-Iu \vvvert_{\mathrm{pw}}\vvvert u-u_{\mathrm{nc}} \vvvert_{\mathrm{pw}}.
							$ 
							\cref{cor:ConformingCompanion}.\ref{item:cor_J_a_vh} shows 
							$
								a_{\mathrm{pw}}(u, Ju_{\mathrm{nc}}-u_{\mathrm{nc}})
								\le {{M_2}} \vvvert u-Iu\vvvert_{\mathrm{pw}} \vvvert u-u_{\mathrm{nc}}\vvvert_{\mathrm{pw}}.
							$
							Since $\lambda_h\le \lambda$, these estimates lead in \eqref{eq:BIE_inbetween} to 
							\begin{align*}
								\textup{LHS}\hspace{-0.2em}			
									\le &
										\big(\hspace{-0.1em}(1\hspace{-0.1em}	+\hspace{-0.1em}	\lambda) C_4^2 (C_5h_{\max}^{\sigma}+\sqrt{2/\lambda}\kappa_m)^2
										\hspace{-0.1em}	+\hspace{-0.1em}		 \lambda C_5^2h_{\max}^{2\sigma}											\big)
									\vvvert u-Iu\vvvert^2
									\hspace{-0.1em}	+\hspace{-0.1em}	3M_2
									\vvvert u-Iu \vvvert_{\mathrm{pw}}\vvvert u_{\mathrm{nc}}-u \vvvert_{\mathrm{pw}}.
							\end{align*}
							A weighted Young inequality and the absorption of  $\vvvert u_{\mathrm{nc}}-u \vvvert_{\mathrm{pw}}^2/2$  
							conclude the proof of \cref{thm:BoundInterpolationError}.\ref{item:BoundInterpolationError} 
							with $C_0:=\max\{C_5^2, 2((1+\lambda) C_4^2 (C_5h_{\max}^{\sigma}+\sqrt{2/\lambda}\kappa_m)^2+ \lambda C_5^2h_{\max}^{2\sigma})
							+9M_2^2\}$.\phantom{x}
						\end{proofof}
\section{Optimal convergence rates}\label{sec:aposteriori}
		This section verifies some general axioms of adaptivity \cite{CFPP14,CR16} sufficient for optimal rates for 
		\namecref{alg:AFEM4EVP} and prepares the conclusion of the proof of  \cref{thm:Optimality4GLB} in \cref{sec:optimality4evp}.
		  			
			\subsection{Stability and reduction}\label{sec:A1-A2}
				The $2$-level notation of  \cref{tab:2level} concerns one coarse triangulation $\mathcal{T}\in \mathbb{T}$ and 
				one fine triangulation $\widehat{\mathcal{T}}\in\mathbb{T}(\mathcal{T})$. 
				Let $(\lambda,u)\in \mathbb{R}^+\times V$  denote  the ${k}$-th continuous eigenpair 
					of \eqref{eq:contEVP} with a \emph{simple} eigenvalue $\lambda\equiv\lambda_k$ and the normalization 
					$\Vert u\Vert_{L^2(\Omega)}=1$. 
					Choose $\varepsilon_5>0$  as in  \cref{thm:BoundInterpolationError}, 
					suppose $\mathcal{T}\in\mathbb{T}(\varepsilon_5)$, 
					and let $\widehat{\mathcal{T}}\in\mathbb{T}(\mathcal{T})$ be any admissible refinement of $\mathcal{T}$. 
				\begin{definition}[$2$-level notation]\label{def:2-level} 
					Let $(\lambda_h,\boldsymbol{u_h})\in \mathbb{R}^+\times \boldsymbol{V_h}$ 
					(resp.\ $(\widehat{\lambda}_h,\boldsymbol{\widehat{u}_h})\in \mathbb{R}^+\times \boldsymbol{\widehat{V}_h}$) with 
					$\boldsymbol{u_h}=(u_{\mathrm{pw}},u_{\mathrm{nc}})\in \boldsymbol{V_h}:=P_m(\mathcal{T})\times V({\mathcal{T}})$
					(resp.\ $\boldsymbol{\widehat{u}_h}=(\widehat{u}_{\mathrm{pw}},\widehat{u}_{\mathrm{nc}})
						\in \boldsymbol{\widehat{V}_h}:=P_{m}(\widehat{\mathcal{T}})\times V(\widehat{\mathcal{T}})$)				
					 denote the ${k}$-th discrete eigenpair  
					of \eqref{eq:dis_EVP_alt} with the simple algebraic eigenvalue $\lambda_h\equiv\lambda_h({k})$ 
					(resp.\ $\widehat{\lambda}_h\equiv \widehat{\lambda}_h(k)$), 
					the normalization $\Vert u_{\mathrm{nc}}\Vert_{L^2(\Omega)}=1$ (resp.\ $\Vert \widehat{u}_{\mathrm{nc}}\Vert_{L^2(\Omega)}=1$),  
					and the sign convention $b(u,u_{\mathrm{nc}})> 0$ (resp.\ $b(u,\widehat{u}_{\mathrm{nc}})> 0$). 
					Recall $\widehat{h}_{\max}:=\max_{T\in\widehat{\mathcal{T}}}h_T\le h_{\max}:=\max_{T\in\mathcal{T}}h_T\le\varepsilon_5$, 
					$\lambda_h,\,\widehat{\lambda}_h\le \lambda$ from \cref{thm:BoundInterpolationError}.\ref{item:lambdahGLB}, and  
					${\delta}$ from \cref{rem:delta} with its analogue
					$\widehat{\delta}:=(1-\widehat{\lambda}_h\kappa_{m}^2h_{\widehat{\mathcal{T}}}^{2{m}})^{-1}-1
					\in P_0(\widehat{\mathcal{T}})$ on the fine level. The constant 
					$C_{\delta}:=2\lambda\kappa_m^2$ satisfies $ \delta\le C_\delta h_{\mathcal{T}}^{2{m}}$ 
					and $ \widehat{\delta}\le C_\delta h_{\widehat{\mathcal{T}}}^{2{m}}$.
					Recall the estimator $\eta^2(T)$ for any $T\in\mathcal{T}$ from \eqref{eq:def_eta} and define $\widehat{\eta}^2(T)$, 
					for any $T\in\widehat{\mathcal{T}}$ with volume $|T|$ and the set of faces $\widehat{\mathcal{F}}(T)$, by 
					\begin{align}
							\widehat{\eta}^2(T) := |T|^{2{m}/3}\Vert \widehat{\lambda}_h \widehat{u}_{\mathrm{nc}} \Vert^2_{L^2(T)}
												+|T|^{1/3}\sum_{F\in\widehat{\mathcal{F}}(T)}\Vert 
												[{D}^{m}_{\mathrm{pw}} \widehat{u}_{\mathrm{nc}}]_F\times \nu_F\Vert^2_{L^2(F)}. 	
												\label{eq:def_etahat}
					\end{align}
				\end{definition}
				\begin{table}
					\begin{center}
					\begin{tabular}{|l|l|}
							\hline 
							$(\lambda_h,\boldsymbol{u_h})\in \mathbb{R}^+\times \boldsymbol{V_h}$ $k$-th eigenpair in \eqref{eq:dis_EVP_alt}					
								\phantom{\huge{I}}
							& 
							$(\widehat{\lambda}_h,\boldsymbol{\widehat{u}_h})\in \mathbb{R}^+\times \boldsymbol{\widehat{V}_h}$  
							$k$-th eigenpair in \eqref{eq:dis_EVP_alt}
								\phantom{\huge{I}}
							\\
							with $\boldsymbol{u_h}=(u_{\mathrm{pw}},u_{\mathrm{nc}})\in P_m(\mathcal{T})\times V({\mathcal{T}})$
							& 
							with 
							$\boldsymbol{\widehat{u}_h}=(\widehat{u}_{\mathrm{pw}},\widehat{u}_{\mathrm{nc}})
							\in P_m(\widehat{\mathcal{T}})\times V(\widehat{\mathcal{T}})$
							\\
							$\Vert u_{\mathrm{nc}}\Vert_{L^2(\Omega)}=1$, $b(u,u_{\mathrm{nc}})> 0$, $\lambda_h\le\lambda$
							& 
							$\Vert \widehat{u}_{\mathrm{nc}}\Vert_{L^2(\Omega)}=1$, $b(u,\widehat{u}_{\mathrm{nc}})> 0$, $\widehat{\lambda}_h\le\lambda$
							\\
							$h_{\max}:=\max_{T\in\mathcal{T}}h_T$
							& 
							$\widehat{h}_{\max}:=\max_{T\in\widehat{\mathcal{T}}}h_T$
							\\
							${\delta}:=(1-{\lambda}_h\kappa_{m}^2h_{{\mathcal{T}}}^{2{m}})^{-1}-1 \le C_\delta h_{{\mathcal{T}}}^{2{m}}\le 1$
							& 
							$\widehat{\delta}:=(1-\widehat{\lambda}_h\kappa_{m}^2h_{\widehat{\mathcal{T}}}^{2{m}})^{-1}-1 
							\le C_\delta h_{\widehat{\mathcal{T}}}^{2{m}}\le 1$
							\\
							$\eta^2(T)$ from \eqref{eq:def_eta} for $T\in\mathcal{T}$ 
							& 
							$\widehat{\eta}^2(T)$ from \eqref{eq:def_etahat}  
							for $T\in\widehat{\mathcal{T}}$ 
							\\
							$\eta^2(\mathcal{M}):= \sum_{T\in\mathcal{M}}\eta^2(T)$ for $\mathcal{M}\subseteq\mathcal{T}$  
							&
							$\widehat{\eta}^2(\widehat{\mathcal{M}}):= \sum_{T\in\widehat{\mathcal{M}}}\widehat{\eta}^2(T)$ 
							for $\widehat{\mathcal{M}}\subseteq\widehat{\mathcal{T}}$  \\[1ex]
					\hline 
					\end{tabular}
					\caption{$2$-level notation with respect to $\mathcal{T}\in\mathbb{T}(\varepsilon)$ (left) and an admissible refinement 
					$\widehat{\mathcal{T}}\in\mathbb{T}(\mathcal{T})$ (right)}\label{tab:2level}				
					\end{center}
				\end{table}
				The sum conventions $\eta^2(\mathcal{M}):= \sum_{T\in\mathcal{M}}\eta^2(T)$ for $\mathcal{M}\subset\mathcal{T}$ 
				and $\widehat{\eta}^2(\widehat{\mathcal{M}}):= \sum_{T\in\widehat{\mathcal{M}}}\widehat{\eta}^2(T)$  for 
				$\widehat{\mathcal{M}}\subset\widehat{\mathcal{T}}$ from \cref{tab:2level} apply throughout this section. 
				Abbreviate the distance function 
				\begin{align}
					\delta^2(\mathcal{T},\widehat{\mathcal{T}})
							:=\Vert \lambda_h u_{\mathrm{nc}}-\widehat{\lambda}_h\widehat{u}_{\mathrm{nc}}\Vert_{L^2(\Omega)}^2
							+ \vvvert u_{\mathrm{nc}}-\widehat{u}_{\mathrm{nc}}\vvvert_{\mathrm{pw}}^2. \label{eq:def_dist_delta}
				\end{align}
				\begin{theorem}[stability and reduction]\label{thm:A1_A2}
					There exist $\Lambda_1,\allowbreak \Lambda_2~>~0$, 
						such that, for any $\mathcal{T}$ and $\widehat{\mathcal{T}}$ from \cref{def:2-level}, the following holds  
						\begin{enumerate}[label={\textup{({A${\arabic*}$})}}, leftmargin=3.5em]
							\item\label{A1Stability}Stability. 
									$\displaystyle
						 						\big\lvert{\eta(\mathcal{T}\cap \widehat{\mathcal T})
						 						-{\widehat{\eta}}( \mathcal{T}\cap\widehat{\mathcal T})}\big\rvert 
						 						\leq \Lambda_1 \delta(\mathcal{T},\widehat{\mathcal T}),
										$	
							\item\label{A2Reduction}Reduction. 
									$\displaystyle
									\widehat{\eta}(\widehat{\mathcal T}\setminus\mathcal{T}) 
											\leq 2^{-1/12} \eta(\mathcal{T}\setminus\widehat{\mathcal T}) 
											+ \Lambda_2 \delta(\mathcal{T},\widehat{\mathcal T}).
									$
						\end{enumerate}
				\end{theorem}
				\begin{proof}
					A reverse triangle inequality in $\mathbb{R}^L$ for the number $L:=|\mathcal{T}\cap\widehat{\mathcal{T}}|$ of tetrahedra in 
					$\mathcal{T}\cap\widehat{\mathcal{T}}$ and one for each common tetrahedra $T\in\mathcal{T}\cap\widehat{\mathcal{T}}$ and each 
					of its faces $F\in\mathcal{F}(T)$ lead to 
					\begin{align*}
						\big\vert\eta(\mathcal{T}\cap \widehat{\mathcal T})
						-\widehat{\eta}(\mathcal{T}\cap\widehat{\mathcal T})\big\vert^2
						\le & \sum_{T\in \mathcal{T}\cap\widehat{\mathcal{T}}}\Big(|T|^{2{m}/3}\Vert \lambda_hu_{\mathrm{nc}}
						-\widehat{\lambda}_h\widehat{u}_{\mathrm{nc}}\Vert_{L^2(T)}^2
						 	 				\\&+|T|^{1/3}\sum_{F\in\mathcal{F}(T)}
						 	 				\Vert [D^{m}_{\mathrm{pw}}(u_{\mathrm{nc}}-\widehat{u}_{\mathrm{nc}})]_F\times \nu_F\Vert_{L^2(F)}^2\Big). 
					\end{align*}
					The discrete jump control from \cite[Lem.~5.2]{CR16} with constant $C_{\mathrm{jc}}(\ell)$ 
					(that only depends on the shape-regularity of $\mathbb{T}$  and the polynomial degree $\ell\in\mathbb{N}_0$) reads 	
					\begin{align*}
						\sum_{T\in\mathcal{T}}|T|^{1/3}\sum_{F\in\mathcal{F}(T)}\Vert [g]_F\Vert^2_{L^2(F)}
						\le C_{\mathrm{jc}}(\ell)^2\Vert g\Vert_{L^2(\Omega)}^2 \quad \text{ for any }g\in P_\ell(\mathcal{T}).
					\end{align*}
					The combination of the two displayed estimates concludes the proof of \ref{A1Stability} with 
					$\Lambda_1^2=\max\big\{\max_{T\in\mathcal{T}_0}|T|^{2m/3}, C_{\mathrm{jc}}(0)^2\big\}$.
					 For any tetrahedron $K\in\mathcal{T}\setminus\widehat{\mathcal{T}}$, let 
					 $\widehat{\mathcal{T}}(K):=\{T\in\widehat{\mathcal{T}}:\, T\subset K\}$ denote its fine triangulation. 
					 The newest-vertex bisection guarantees $|{T}|\le |K|/2$ for the volume $|T|$ of any ${T}\in \widehat{\mathcal{T}}(K)$.
					 This, a triangle inequality, and 
					 $(a+b)^2\le (1+\beta)a^2+(1+1/\beta)b^2$ for $a,\,b\ge 0,\,\beta=2^{1/6}-1>0$  show
					\begin{align*}
						  	\SPnew{\widehat{\eta}^2(\widehat{\mathcal{T}}(K))}
						  		 \le 2^{-1/6}& 
						  		 \SPnew{\eta^2(K)}
						  		 +(1+1/\beta)\sum_{T\in\widehat{\mathcal{T}}(K)}\Big(|T|^{2{m}/3}
						  		 \Vert \lambda_h u_{\mathrm{nc}}-\widehat{\lambda}_h\widehat{u}_{\mathrm{nc}}\Vert_{L^2(K)}^2\\
						  		 &+|T|^{1/3}\sum_{F\in\widehat{\mathcal{F}}(T)}
						  		 \Vert [D^{m}_{\mathrm{pw}}(\widehat{u}_{\mathrm{nc}}-u_{\mathrm{nc}})]_F\times\nu_F\Vert_{L^2(F)}^2\Big).
					\end{align*}
					The  summation over all $K\in\mathcal{T}\setminus\widehat{\mathcal{T}}$ and the above jump control 
					conclude the proof of \ref{A2Reduction} with 
					$\Lambda_2^2=2^{1/6}/(2^{1/6}-1)\,\max\big\{\max_{T\in\mathcal{T}_0}|T|^{2m/3}, C_{\mathrm{jc}}(0)^2\big\}$.	
					The arguments for \ref{A1Stability}--\ref{A2Reduction}  are similiar for other problems; 
					cf., e.g., \cite{CKN08,CFPP14,CR16,CHel17} for more details. 
				\end{proof}
		\subsection{Towards discrete reliability}\label{sec:proofofA3}
					Given  the $2$-level notation of \cref{def:2-level} with respect to $\mathcal{T}$ and $\widehat{\mathcal{T}}$, let
					$\mathcal{R}_1:=\{K\in\mathcal{T}:\,\exists\, T\in \mathcal{T}\setminus\widehat{\mathcal{T}} \text{ with } \textup{dist}(K,T)=0\}
					\subset\mathcal{T}$ 
					 denote the set of coarse but not fine tetrahedra plus one layer of coarse tetrahedra around. 
				\cref{lem:TowardsDisRel}--\ref{lem:delta_control} prepare the proof of the discrete reliability in \cref{thm:A34EVP} below.  
				Let $\widehat{I}:V+V(\widehat{\mathcal{T}})\to V(\widehat{\mathcal{T}})$ denote the 
				interpolation operator on the fine level of $\widehat{\mathcal{T}}$ so that \ref{item:I_identity} and a  Cauchy-Schwarz inequality 
				show, for any $v\in V+V(\widehat{\mathcal{T}})$ and any 
				$w\in  V+V({\mathcal{T}})+V(\widehat{\mathcal{T}})$, that 
				\begin{align}
					\begin{split}
					\vert b( (I-\widehat{I})v,w)\vert\le &\Vert  (I-\widehat{I})v\Vert_{L^2(\mathcal{T}\setminus\widehat{\mathcal{T}})} 
															\Vert w\Vert_{L^2(\mathcal{T}\setminus\widehat{\mathcal{T}})},\\
					\vert a_{\mathrm{pw}}((I-\widehat{I})v,w)\vert
						\le& \Vert D^{m}_{\mathrm{pw}} (I-\widehat{I})v\Vert_{L^2(\mathcal{T}\setminus\widehat{\mathcal{T}})} 
															\Vert D^{m}_{\mathrm{pw}}w\Vert_{L^2(\mathcal{T}\setminus\widehat{\mathcal{T}})}.\end{split}
															\label{eq:I-hatI_CS}
				\end{align} 
				\begin{lemma}[distance control I]\label{lem:TowardsDisRel}
				There exists $C_6>0$ such that any $\mathcal{T}\in\mathbb{T}(\varepsilon_5)$ and the difference $e:= \widehat{u}_{\mathrm{nc}}-u_{\mathrm{nc}}$ satisfy
					$$
						C_6^{-1}\vvvert e\vvvert_{\mathrm{pw}}^2
							\le
								\Vert D^{m}_{\mathrm{pw}} (u_{\mathrm{nc}}-Ju_{\mathrm{nc}})\Vert_{L^2(\mathcal{T}\setminus\widehat{\mathcal{T}})}^2
										+\Vert h_{\mathcal{T}}^{m}\lambda_hu_{\mathrm{nc}}\Vert_{L^2(\mathcal{T}\setminus \widehat{\mathcal{T}})}^2
										+ \Vert e\Vert_{L^2(\Omega)}^2+\Vert \delta u_{\mathrm{nc}}\Vert_{L^2(\Omega)}^2
										+\Vert \widehat{\delta}\widehat{ u}_{\mathrm{nc}}\Vert_{L^2(\Omega)}^2.
					$$
				\end{lemma}
				\begin{proof}
					\cref{cor:InterpolationOperator}.\ref{item:cor_I_a} shows  
					$
						 a_{\mathrm{pw}}(e, \widehat{u}_{\mathrm{nc}}- Ju_{\mathrm{nc}})
							= a_{\mathrm{pw}}\big(\widehat{u}_{\mathrm{nc}}, \widehat{u}_{\mathrm{nc}}- \widehat{I}Ju_{\mathrm{nc}}\big)
							-a_{\mathrm{pw}}\big(u_{\mathrm{nc}}, I(\widehat{u}_{\mathrm{nc}}- Ju_{\mathrm{nc}})\big).
					$ 
					Since $(\lambda_h,u_{\mathrm{nc}})$ and $(\widehat{\lambda}_h,\widehat{u}_{\mathrm{nc}})$ solve \eqref{eq:disEVP}, this and 
					 \ref{item:J_rightInverse} lead to 
					\begin{align}
							a_{\mathrm{pw}}(e, \widehat{u}_{\mathrm{nc}}- Ju_{\mathrm{nc}})=& b\big(\widehat{\lambda}_h\widehat{u}_{\mathrm{nc}}, 
							(1+\widehat{\delta})(\widehat{u}_{\mathrm{nc}}- \widehat{I}Ju_{\mathrm{nc}})\big)
									-b\big(\lambda_h u_{\mathrm{nc}},(1+\delta)(I\widehat{u}_{\mathrm{nc}}- u_{\mathrm{nc}})\big)\notag\\
							=&b(\widehat{\lambda}_h\widehat{u}_{\mathrm{nc}}-\lambda_h u_{\mathrm{nc}},e)
							+b(\widehat{\lambda}_h\widehat{u}_{\mathrm{nc}},\widehat{\delta}e)-b(\lambda_h u_{\mathrm{nc}},\delta e)\notag\\
							 &+b\big(\widehat{\lambda}_h\widehat{u}_{\mathrm{nc}}, 
							 (1\hspace{-0.1em} +\hspace{-0.1em} \widehat{\delta})(u_{\mathrm{nc}}-\widehat{I}Ju_{\mathrm{nc}})\big)
							 +
							b\big(\lambda_h u_{\mathrm{nc}},(1\hspace{-0.1em} +\hspace{-0.1em} \delta)(\widehat{u}_{\mathrm{nc}}-I\widehat{u}_{\mathrm{nc}})\big).  
							 \label{eq:sum_TowardsDisRel}
					\end{align}  
					Elementary algebra with $\Vert u_{\mathrm{nc}}\Vert_{L^2(\Omega)}=\Vert\widehat{u}_{\mathrm{nc}}\Vert_{L^2(\Omega)}=1$ shows 
					(as, e.g., in \cite[Lem.~3.1]{CG11}) 
					\begin{align*}
						b(\widehat{\lambda}_h\widehat{u}_{\mathrm{nc}}-\lambda_h u_{\mathrm{nc}},e)
							&=\frac{\widehat{\lambda}_h+\lambda_h}{2}\Vert e\Vert_{L^2(\Omega)}^2
							+\frac{\widehat{\lambda}_h-\lambda_h}{2}b(\widehat{u}_{\mathrm{nc}}
							+u_{\mathrm{nc}},\widehat{u}_{\mathrm{nc}}-u_{\mathrm{nc}})
							=\frac{\widehat{\lambda}_h+\lambda_h}{2}\Vert e\Vert_{L^2(\Omega)}^2.
					\end{align*}
					Cauchy-Schwarz inequalities verify  
					\begin{align*}
						b(\widehat{\lambda}_h\widehat{u}_{\mathrm{nc}},\widehat{\delta}e)-b(\lambda_h u_{\mathrm{nc}},\delta e)
							\le \Vert e\Vert_{L^2(\Omega)}\big(\widehat{\lambda}_h 
							\Vert \widehat{\delta}\widehat{ u}_{\mathrm{nc}}\Vert_{L^2(\Omega)}
							+\lambda_h\Vert \delta u_{\mathrm{nc}}\Vert_{L^2(\Omega)}\big).
					\end{align*}
					Since $1+\widehat{\delta}\le 2$ and $\widehat{\lambda}_h\le \lambda$ from \cref{tab:2level}, 
					the right inverse property \ref{item:J_rightInverse} and \eqref{eq:I-hatI_CS} result in 
					\begin{align*}
						b\big((1+\widehat{\delta})\widehat{\lambda}_h\widehat{u}_{\mathrm{nc}}, u_{\mathrm{nc}}-\widehat{I}Ju_{\mathrm{nc}}\big)
							=&b\big((1+\widehat{\delta})\widehat{\lambda}_h\widehat{u}_{\mathrm{nc}}, (I-\widehat{I})Ju_{\mathrm{nc}}\big)
							\\\le& 2
							\Vert h_{\mathcal{T}}^{m}{\lambda}\widehat{u}_{\mathrm{nc}}\Vert_{L^2({\mathcal{T}}\setminus \widehat{\mathcal{T}})}
							\Vert h_{\mathcal{T}}^{-{m}}	(I-\widehat{I})Ju_{\mathrm{nc}}\Vert_{L^2({\mathcal{T}}\setminus \widehat{\mathcal{T}})}.
					\end{align*}
					The triangle inequality  
					$
						\Vert h_{\mathcal{T}}^{m}{\lambda}\widehat{u}_{\mathrm{nc}}\Vert_{L^2({\mathcal{T}}\setminus \widehat{\mathcal{T}})}
						\le h_{\max}^{m}\lambda\Vert e\Vert_{L^2(\Omega)}
						+\Vert  h_{\mathcal{T}}^{m}\lambda u_{\mathrm{nc}}\Vert_{L^2(\mathcal{T}\setminus \widehat{\mathcal{T}})}
					$ and $\lambda/\lambda_h\le 2$ from \cref{thm:BoundInterpolationError}.\ref{item:lambdahSimple} imply 
					$\Vert  h_{\mathcal{T}}^{m}\lambda u_{\mathrm{nc}}\Vert_{L^2(\mathcal{T}\setminus \widehat{\mathcal{T}})}
					\le 2\Vert  h_{\mathcal{T}}^{m}\lambda_h u_{\mathrm{nc}}\Vert_{L^2(\mathcal{T}\setminus \widehat{\mathcal{T}})}$. 
					Since the interpolation operators $I$ and $\widehat{I}$ satisfy \ref{item:I_identity}--\ref{item:I_kappa_disc}, it follows  that 
					\begin{align*} 
						\Vert h_{\mathcal{T}}^{-{m}}	(I-\widehat{I})Ju_{\mathrm{nc}}\Vert_{L^2({\mathcal{T}}\setminus \widehat{\mathcal{T}})}
							&=\Vert h_{\mathcal{T}}^{-{m}}	(1-I)\widehat{I}Ju_{\mathrm{nc}}\Vert_{L^2({\mathcal{T}}\setminus \widehat{\mathcal{T}})}
						\le \kappa_d \Vert D^{m}_{\mathrm{pw}}(1-I)\widehat{I}Ju_{\mathrm{nc}}\Vert_{L^2({\mathcal{T}}\setminus \widehat{\mathcal{T}})}.
					\end{align*}
					Recall $D^m_{\mathrm{pw}}u_{\mathrm{nc}}\in P_0(\mathcal{T};\mathbb{R}^{3^m})$. 
					The condition \ref{item:I_Pi0} and the $L^2$-orthogonal projections $\Pi_0$ (resp. $\widehat{\Pi}_0$) onto $P_0(\mathcal{T})$ (resp. 
					$P_0(\widehat{\mathcal{T}})$) lead to the estimate 
					\begin{align*}
						&\kappa_d^{-1} 
						\Vert h_{\mathcal{T}}^{-{m}}	(I-\widehat{I})Ju_{\mathrm{nc}}\Vert_{L^2({\mathcal{T}}\setminus \widehat{\mathcal{T}})}
							\le \Vert(\Pi_0-\widehat{\Pi}_0) D^{m}Ju_{\mathrm{nc}}\Vert_{L^2({\mathcal{T}}\setminus \widehat{\mathcal{T}})}\\
							=& \Vert(\Pi_0-\widehat{\Pi}_0) D^{m}_{\mathrm{pw}}
							(Ju_{\mathrm{nc}}-u_{\mathrm{nc}})\Vert_{L^2({\mathcal{T}}\setminus \widehat{\mathcal{T}})}
							\le  \Vert D^{m}_{\mathrm{pw}}(Ju_{\mathrm{nc}}-u_{\mathrm{nc}})\Vert_{L^2({\mathcal{T}}\setminus \widehat{\mathcal{T}})}.
					\end{align*}
				  	The estimate \eqref{eq:I-hatI_CS} and  $\delta\le 1$ from \cref{tab:2level} imply the first inequality and  
				  	\ref{item:I_kappa_disc} and \cref{cor:InterpolationOperator}.\ref{item:cor_I_orthogonality} the second  in  
				  	\begin{align*}
				  		&b\big( \lambda_h u_{\mathrm{nc}},(1+\delta)(\widehat{u}_{\mathrm{nc}}-I\widehat{u}_{\mathrm{nc}})\big)
				  				=b\big( \lambda_hu_{\mathrm{nc}},(1+\delta)(\widehat{I}-I)\widehat{u}_{\mathrm{nc}}\big)
				  				\\\le& 2 
				  				\Vert h_{\mathcal{T}}^{m}\lambda_hu_{\mathrm{nc}}\Vert _{L^2({\mathcal{T}}\setminus \widehat{\mathcal{T}})}
				  				\Vert h_{\mathcal{T}}^{-{m}}(\widehat{u}_{\mathrm{nc}}
				  				-I\widehat{u}_{\mathrm{nc}})\Vert _{L^2({\mathcal{T}}\setminus \widehat{\mathcal{T}})}
				  				\le 2 \kappa_d
				  				\Vert h_{\mathcal{T}}^{m}\lambda_hu_{\mathrm{nc}}\Vert _{L^2({\mathcal{T}}\setminus\widehat{ \mathcal{T}})}
				  							\vvvert e\vvvert _{\mathrm{pw}}. 
				  	\end{align*} 
				  	The combination of the   six previously displayed estimates and  
					$\lambda_h,\widehat{\lambda}_h\le \lambda$  lead in \eqref{eq:sum_TowardsDisRel} to 
				  	\begin{align*}
				  	a_{\mathrm{pw}}(e, \widehat{u}_{\mathrm{nc}}- Ju_{\mathrm{nc}})
				  	 \le  2\kappa_d \Vert D^{m}_{\mathrm{pw}} (u_{\mathrm{nc}}-Ju_{\mathrm{nc}})
				  					\Vert_{L^2(\mathcal{T}\setminus\widehat{\mathcal{T}})}
				  			\big(\lambda h_{\max}^{m}\Vert e\Vert_{L^2(\Omega)}
				  			+2\Vert  h_{\mathcal{T}}^{m}\lambda_hu_{\mathrm{nc}}\Vert_{L^2(\mathcal{T}\setminus \widehat{\mathcal{T}})}\big) &\\
				  			+\lambda\Vert e\Vert_{L^2(\Omega)}\big( \Vert e\Vert_{L^2(\Omega)}+\Vert \delta u_{\mathrm{nc}}\Vert_{L^2(\Omega)}+\Vert 
				  			\widehat{\delta}\widehat{ u}_{\mathrm{nc}}\Vert_{L^2(\Omega)}\big)
				  			+2\kappa_d\vvvert e\vvvert _{\mathrm{pw}}\Vert h_{\mathcal{T}}^{m} \lambda_h u_{\mathrm{nc}}
				  			\Vert _{L^2({\mathcal{T}}\setminus\widehat{ \mathcal{T}})}.&
					\end{align*}				  		 
				  	Additionally, \cref{cor:ConformingCompanion}.\ref{item:cor_J_a_vh} and \eqref{eq:I-hatI_CS} show 
					\begin{align*}
						 a_{\mathrm{pw}}(e, Ju_{\mathrm{nc}}-u_{\mathrm{nc}})
						 &=a_{\mathrm{pw}}((1-I)e, Ju_{\mathrm{nc}}-u_{\mathrm{nc}})
						 =a_{\mathrm{pw}}\big((\widehat{I}-I)\widehat{u}_{\mathrm{nc}}, Ju_{\mathrm{nc}}-u_{\mathrm{nc}}\big)\\
						 &\le \Vert D^{m}_{\mathrm{pw}}(1-I)e\Vert_{L^2(\mathcal{T}\setminus\widehat{\mathcal{T}})}
						 		\Vert D^{m}_{\mathrm{pw}} (u_{\mathrm{nc}}-Ju_{\mathrm{nc}})\Vert_{L^2(\mathcal{T}\setminus\widehat{\mathcal{T}})}.
					\end{align*}	
					Condition \ref{item:I_Pi0} and the boundedness of $\Pi_0$ show  	
					$\Vert D^{m}_{\mathrm{pw}}(1-I)e\Vert_{L^2(\mathcal{T}\setminus\widehat{\mathcal{T}})}
					\le \vvvert e\vvvert_{\mathrm{pw}}$.
				  	This and the combination of the two previously displayed estimates with a triangle inequality prove  
				  	\begin{align*}
				  		\vvvert e\vvvert_{\mathrm{pw}}^2 
				  			=&a_{\mathrm{pw}}(e, Ju_{\mathrm{nc}}-u_{\mathrm{nc}})+a_{\mathrm{pw}}(e, \widehat{u}_{\mathrm{nc}}- Ju_{\mathrm{nc}})\\
				  			 \le &  \Vert D^{m}_{\mathrm{pw}} (u_{\mathrm{nc}}-Ju_{\mathrm{nc}})
				  					\Vert_{L^2(\mathcal{T}\setminus\widehat{\mathcal{T}})}
				  			\big(\vvvert e\vvvert_{\mathrm{pw}}+2\kappa_d \lambda h_{\max}^{m}\Vert e\Vert_{L^2(\Omega)}
				  								+4\kappa_d \Vert  h_{\mathcal{T}}^{m}\lambda_hu_{\mathrm{nc}}
				  								\Vert_{L^2(\mathcal{T}\setminus \widehat{\mathcal{T}})}\big)\\
				  			&+\lambda\Vert e\Vert_{L^2(\Omega)}\big( \Vert e\Vert_{L^2(\Omega)}+\Vert \delta u_{\mathrm{nc}}\Vert_{L^2(\Omega)}+\Vert 
				  			\widehat{\delta}\widehat{ u}_{\mathrm{nc}}\Vert_{L^2(\Omega)}\big)
				  			+2\kappa_d\vvvert e\vvvert _{\mathrm{pw}}\Vert h_{\mathcal{T}}^{m}\lambda_h u_{\mathrm{nc}}
				  			\Vert _{L^2({\mathcal{T}}\setminus\widehat{ \mathcal{T}})}\\
						\le& (1+4\kappa_d^2+\kappa_d^2\lambda^2h_{\max}^{2m})
						\Vert D^{m}_{\mathrm{pw}} (u_{\mathrm{nc}}-Ju_{\mathrm{nc}})\Vert_{L^2(\mathcal{T}\setminus\widehat{\mathcal{T}})}^2
				  			+\Vert \delta u_{\mathrm{nc}}\Vert_{L^2(\Omega)}^2
				  			+\Vert\widehat{\delta}\widehat{ u}_{\mathrm{nc}}\Vert_{L^2(\Omega)}^2
				  			\\&+(1+\lambda+\lambda^2/2) \Vert e\Vert_{L^2(\Omega)}^2	
				  			+ (1+4\kappa_d^2)
				  			\Vert h_{\mathcal{T}}^{m}\lambda_h u_{\mathrm{nc}}\Vert _{L^2({\mathcal{T}}\setminus\widehat{ \mathcal{T}})}^2
				  			+{\vvvert e\vvvert _{\mathrm{pw}}^2}/{2}
					\end{align*}
				  	with weighted Young inequalities in the last step. This concludes the proof with 
				  	$C_6:=2\max\{1+4\kappa_d^2+\kappa_d^2\lambda^2h_{\max}^{2m}, 1+\lambda+\lambda^2/2\}$. 
					\end{proof}
					
\subsubsection{Reliability and efficiency}
A first consequence of \cref{lem:TowardsDisRel} is the reliability of the error estimator $\eta(\mathcal{T})$ from \eqref{eq:def_eta}. 
					
\begin{theorem}[reliability and efficiency]\label{thm:reliabilty+efficiency}
						There exist $C_{\mathrm{rel}},\,C_{\mathrm{eff}},$ and $\varepsilon_6>0$ such that \\
\mbox{}\hspace{35mm}$
		 C_{\mathrm{eff}}^{-1}\eta (\mathcal{T})
							 \le \vvvert u-u_{\mathrm{nc}}\vvvert_{\mathrm{pw}}\le C_{\mathrm{rel}} \,\eta(\mathcal{T})
						$ \hfill holds for $\mathcal{T}\in\mathbb{T}(\varepsilon_6)$. 
					\end{theorem}
			\begin{proofof}\textit{reliability.}	
			\cref{lem:TowardsDisRel} 
			holds for any refinement $\widehat{\mathcal{T}}\in\mathbb{T}(\mathcal{T})$ of $\mathcal{T}\in\mathbb{T}(\varepsilon_5)$  
			and we may consider a sequence $\widehat{\mathcal{T}}=\widehat{\mathcal{T}}_\ell$ of uniform mesh-refinements of $\mathcal{T}$. 
			The reliability follows  in the limit as 
			$\widehat{h}_{\max}\to 0$  for $\ell\to\infty$ and $\vvvert u-\widehat{u}_{\mathrm{nc}}\vvvert_{\mathrm{pw}}\to 0$ 
			from  \cref{thm:BoundInterpolationError}.\ref{item:BoundInterpolationError}.  
			The left-hand side of \cref{lem:TowardsDisRel} converges to $C_6^{-1}\vvvert u-u_{\mathrm{nc}}\vvvert_{\mathrm{pw}}$.  
			On the right-hand side, $\Vert \widehat{\delta}\widehat{ u}_{\mathrm{nc}}\Vert_{L^2(\Omega)}\le C_{\delta} \widehat{h}_{\max}^{2m}$ 
			converges to zero and $\Vert e\Vert_{L^2(\Omega)}\to \Vert u- u_{\mathrm{nc}}\Vert_{L^2(\Omega)}$ as 
			$\widehat{h}_{\max}\to 0$. 
			Moreover the shape-regularity $h_T\le C_{\mathrm{sr}}|T|^{1/3}$  for $T\in\mathcal{T}\in\mathbb{T}$, \ref{item:J_approx}, and 
			$\Vert {\delta}{ u}_{\mathrm{nc}}\Vert_{L^2(\Omega)}
			\le  2\kappa_m^2 {h}_{\max}^{m}\Vert h_{\mathcal{T}}^m\lambda_h{ u}_{\mathrm{nc}}\Vert_{L^2(\Omega)}$ 
			show  
			$$\Vert D^{m}_{\mathrm{pw}} (u_{\mathrm{nc}}-Ju_{\mathrm{nc}})\Vert_{L^2(\Omega)}^2
										+\Vert h_{\mathcal{T}}^{m}\lambda_h u_{\mathrm{nc}}\Vert_{L^2(\Omega)}^2
										+\Vert \delta u_{\mathrm{nc}}\Vert_{L^2(\Omega)}^2
										\le \max\{M_1,C_{\mathrm{sr}}^{2m}(1+4\kappa_m^4 {h}_{\max}^{2m})\} \eta^2(\mathcal{T}).$$						
			For the remaining term on the right-hand side, \eqref{eq:L2errorEnergyError} and 
			\cref{cor:InterpolationOperator}.\ref{item:cor_I_orthogonality} show 
			$$
					C_5^{-1}\Vert u-u_{\mathrm{nc}}\Vert_{L^2(\Omega)}\le h_{\max}^{\sigma } \vvvert u-Iu\vvvert_{\mathrm{pw}}
					\le  h_{\max}^{\sigma } \vvvert u-u_{\mathrm{nc}}\vvvert_{\mathrm{pw}} .			
			$$
			A reduction to $\varepsilon_6:=\min\{\varepsilon_5,(2C_5^2\SPnew{C_6})^{-1/(2\sigma)}\}$ such that $C_5^{2}C_6h_{\max}^{2\sigma}\le 1/2$ 
			allows for the absorption of $C_5^2\SPnew{C_6}h_{\max}^{2\sigma}\vvvert u-u_{\mathrm{nc}}\vvvert_{\mathrm{pw}}^2\le \vvvert u-u_{\mathrm{nc}}\vvvert_{\mathrm{pw}}^2/2$ 
			and concludes the proof with $C_{\mathrm{rel}}^2:=2C_6\max\{M_1,C_{\mathrm{sr}}^{2m}(1+4\kappa_m^4 {h}_{\max}^{2m})\}$. 
			\end{proofof}	
					\begin{proofof}\textit{efficiency.}
					The condition \ref{item:J_approx} guarantees 
					\begin{align*}
						M_1/M_2^2\sum_{T\in\mathcal{T}}|T|^{1/3}\sum_{F\in\mathcal{F}(T)}\Vert [D^{m}_{\mathrm{pw}}u_{\mathrm{nc}}]_F\times \nu_F\Vert_{L^2(F)}^2
						\le \min_{v\in V} \vvvert v-u_{\mathrm{nc}}\vvvert_{\mathrm{pw}}^2
						\le \vvvert u-u_{\mathrm{nc}}\vvvert_{\mathrm{pw}}^2. 
					\end{align*}
					The combination of $|T|^{1/3}\le h_T$, $\lambda_h\le\lambda$, 
					and the efficiency \eqref{eq:ControlOfDeltaUh} with  
					$\vvvert u-Iu\vvvert_{\mathrm{pw}}\le \vvvert u-u_{\mathrm{nc}}\vvvert_{\mathrm{pw}}$ from 
					\cref{cor:InterpolationOperator}.\ref{item:cor_I_orthogonality} implies  that 
					\begin{align*}
						\sum_{T\in\mathcal{T}}|T|^{2m/3}\Vert\lambda_h u_{\mathrm{nc}}\Vert_{L^2(T)}^2
							\le  \Vert h_{\mathcal{T}}^m \lambda_h u_{\mathrm{nc}}\Vert_{L^2(\Omega)}^2
							\le 2C_4^2 \big(\lambda^2h_{\max}^{2m}\Vert u-u_{\mathrm{nc}}\Vert_{L^2(\Omega)}^2
							+ \vvvert u-u_{\mathrm{nc}}\vvvert_{\mathrm{pw}}^2\big).
					\end{align*}
					\cref{thm:BoundInterpolationError}.\ref{item:BoundInterpolationError} concludes the 
					proof   with $C_{\mathrm{eff}}^2:=M_2^2/M_1+2C_4^2+2C_4^2 C_0\lambda^2h_{\max}^{2m+2\sigma}$. 
				\end{proofof}					
											
			\subsubsection{Discrete reliability}
				\begin{lemma}[distance control II]\label{lem:delta_control}
					There exists a constant $C_7>0$ such that 
					$\displaystyle
						\Vert \widehat{\lambda}_h\widehat{u}_{\mathrm{nc}}-\lambda_hu_{\mathrm{nc}}\Vert_{L^2(\Omega)}
							\hspace{-0.1em}+\hspace{-0.1em}\Vert \widehat{u}_{\mathrm{nc}}-u_{\mathrm{nc}}\Vert_{L^2(\Omega)}
								\hspace{-0.1em}+\hspace{-0.1em}\Vert \widehat{\delta} \widehat{u}_{\mathrm{nc}}\Vert_{L^2(\Omega)}
								\hspace{-0.1em}+\hspace{-0.1em}\Vert \delta u_{\mathrm{nc}}\Vert_{L^2(\Omega)}
							\hspace{-0.1em}\le\hspace{-0.1em} C_7  h_{\max}^{\sigma }\vvvert u-u_{\mathrm{nc}}\vvvert_{\mathrm{pw}}
							\hspace{-0.1em}\le\hspace{-0.1em} C_7C_{\mathrm{rel}}
							h_{\max}^{\sigma }\eta(\mathcal{T})
							$ holds for any $\mathcal{T}\in\mathbb{T}(\varepsilon_6)$.
				\end{lemma}
					\begin{proof}
						Triangle inequalities and the normalization $\Vert u\Vert_{L^2(\Omega)}=1$ show
						\begin{align*}
							\Vert \widehat{\lambda}_h\widehat{u}_{\mathrm{nc}}-\lambda_h u_{\mathrm{nc}}\Vert_{L^2(\Omega)}
							\le& \lambda_h \Vert u-u_{\mathrm{nc}}\Vert_{L^2(\Omega)} 
							+\widehat{\lambda}_h  \Vert u-\widehat{u}_{\mathrm{nc}}\Vert_{L^2(\Omega)}
										+\vert \widehat{\lambda}_h-\lambda_h\vert.
						\end{align*} 
						\cref{thm:BoundInterpolationError}.\ref{item:BoundInterpolationError}  and \cref{cor:InterpolationOperator}.\ref{item:cor_I_alpha} 
						imply $\vert \lambda-\lambda_h\vert 
								 \le C_0\vvvert u-Iu \vvvert_{\mathrm{pw}}^2\le C_0(h_{\max}/\pi)^{2\sigma}\Vert u\Vert_{H^{m+\sigma}(\Omega)}^2$. 
						Since the eigenfunction $u\in V$ in \eqref{eq:contEVP} 
						solves the source problem with right-hand side $\lambda u\in L^2(\Omega)$, \eqref{eq:def_sigma} implies 
						 $\Vert u\Vert_{H^{m+\sigma}(\Omega)}\le C(\sigma) \Vert \lambda u\Vert_{L^2(\Omega)}=C(\sigma)\lambda$. 
						The same arguments apply to $\vert \lambda-\widehat{\lambda}_h\vert$. 
						This and 
						$\widehat{h}_{\max}^{\sigma} \vvvert u-\widehat{I}u \vvvert_{\mathrm{pw}}\le h_{\max}^{\sigma }\vvvert u-Iu \vvvert_{\mathrm{pw}}$ 
						result in 
						\begin{align*}
								\vert \widehat{\lambda}_h-\lambda_h\vert \le \vert \lambda-\lambda_h\vert +\vert \lambda-\widehat{\lambda}_h\vert
									 \le 
									 2C_0C(\sigma)\lambda/\pi^{\sigma} h_{\max}^{\sigma }\vvvert u-Iu \vvvert_{\mathrm{pw}}. 
						\end{align*}
						Recall $\lambda_h,\widehat{\lambda}_h\le\lambda$, 
						$\Vert \delta u_{\mathrm{nc}}\Vert_{L^2(\Omega)}\le C_{\delta} h_{\max}^m\Vert u_{\mathrm{nc}}\Vert_{\delta}$, and 
						$ \Vert \widehat{\delta} \widehat{u}_{\mathrm{nc}}\Vert_{L^2(\Omega)}
						\le C_{\delta} \widehat{h}_{\max}^m\Vert \widehat{u}_{\mathrm{nc}}\Vert_{\widehat{\delta}}$ from \cref{tab:2level}. 
						The last two displayed estimates, a triangle inequality, and 
						 \cref{thm:BoundInterpolationError}.\ref{item:BoundInterpolationError}
						show 
						\begin{align*}
							\Vert \widehat{\lambda}_h\widehat{u}_{\mathrm{nc}}-\lambda_hu_{\mathrm{nc}}\Vert_{L^2(\Omega)}
											&+ \Vert \widehat{u}_{\mathrm{nc}}-u_{\mathrm{nc}}\Vert_{L^2(\Omega)}
											+\Vert \widehat{\delta} \widehat{u}_{\mathrm{nc}}\Vert_{L^2(\Omega)}
											+\Vert \delta u_{\mathrm{nc}}\Vert_{L^2(\Omega)}
											\\
								 &\le 
								2\big( (C_0C(\sigma)\lambda/\pi^{\sigma}+C_0^{1/2}(1+\lambda)) h_{\max}^{\sigma}+C_{\delta}C_0^{1/2} h_{\max}^{m}\big)
									\vvvert u-Iu\vvvert_{\mathrm{pw}}
						\end{align*}
						with $\vvvert u-\widehat{I}u \vvvert_{\mathrm{pw}}\le \vvvert u-Iu \vvvert_{\mathrm{pw}}$ 
						and $\widehat{h}_{\max}\le h_{\max}$. 
						Since $h_{\max}\le\varepsilon_6<1$ 
						and  $1/2<\sigma\le 1\le m$, \cref{cor:InterpolationOperator}.\ref{item:cor_I_orthogonality}
						concludes the proof of the first bound in \cref{lem:delta_control} with 
						$C_7:=2C_0C(\sigma)\lambda/\pi^{\sigma}+2C_0^{1/2}(1+\lambda+C_{\delta})$.  The second claim follows from 
						\cref{thm:reliabilty+efficiency}. 
					\end{proof}
					
\begin{theorem}[discrete reliability]\label{thm:A34EVP}
					There exist  constants $\Lambda_3, \, M_3>0$ 
					such that $\mathcal{T}\in\mathbb{T}(\varepsilon_6)$ with maximal mesh-size $h_{\max}\le\varepsilon_6$ ($\varepsilon_6$ 
from \cref{thm:reliabilty+efficiency})   and $\epsilon_3:=M_3{h}_{\max}^{2\sigma}$
imply 
					\begin{enumerate}[label=\textup{({A${\arabic*}$})}, leftmargin=3.5em]
						\item[\mylabel{A3DiscreteReliabilty_new}{\textup{({A${3_{\varepsilon}}$})}}]Discrete reliability. 
					 		$\displaystyle
					  		\delta^2(\mathcal{T},\widehat{\mathcal{T}}) \leq \Lambda_3 \eta^2(\mathcal{R}_1)
					  		+\epsilon_3\eta^2(\mathcal{T}).
					 		$
					 \end{enumerate}
				\end{theorem}
				\begin{proof}
						Recall that \cref{lem:TowardsDisRel} shows 
						\begin{align*}
							C_6^{-1}\vvvert \widehat{u}_{\mathrm{nc}}-u_{\mathrm{nc}}\vvvert_{\mathrm{pw}}^2
							\le&\Vert D^{m}_{\mathrm{pw}} (u_{\mathrm{nc}}-Ju_{\mathrm{nc}})\Vert_{L^2(\mathcal{T}\setminus\widehat{\mathcal{T}})}^2
										+\Vert h_{\mathcal{T}}^{m}\lambda_hu_{\mathrm{nc}}\Vert_{L^2(\mathcal{T}\setminus \widehat{\mathcal{T}})}^2
										+ \Vert \widehat{u}_{\mathrm{nc}}-u_{\mathrm{nc}}\Vert_{L^2(\Omega)}^2\\
										&+\Vert \widehat{\delta}\widehat{ u}_h\Vert_{L^2(\Omega)}^2
											+\Vert \delta u_{\mathrm{nc}}\Vert_{L^2(\Omega)}^2.
						\end{align*}
						 This and \cref{lem:delta_control} lead with $M_3:=C_7^2 C_{\mathrm{rel}}^2\max\{1,C_6\}$ to 
						\begin{align*}
							\delta^2(\mathcal{T},\widehat{\mathcal{T}})
								=&\Vert \widehat{\lambda}_h\widehat{u}_{\mathrm{nc}}-\lambda_h u_{\mathrm{nc}}\Vert_{L^2(\Omega)}^2
								+\vvvert \widehat{u}_{\mathrm{nc}}-u_{\mathrm{nc}}\vvvert_{\mathrm{pw}}^2
								\\
								\le&
								C_6\Vert D^{m}_{\mathrm{pw}} (u_{\mathrm{nc}}-Ju_{\mathrm{nc}})\Vert_{L^2(\mathcal{T}\setminus\widehat{\mathcal{T}})}^2
								+C_6\Vert h_{\mathcal{T}}^{m}\lambda_hu_{\mathrm{nc}}\Vert_{L^2(\mathcal{T}\setminus \widehat{\mathcal{T}})}^2
								+ 
								M_3h_{\max}^{2\sigma}\eta^2(\mathcal{T}).
						\end{align*}	
						The shape regularity $h_T\le C_{\mathrm{sr}}|T|^{1/3}$  for any $T\in\mathcal{T}\in\mathbb{T}$  
						guarantees 
\[
						\Vert h_{\mathcal{T}}^{m}\lambda_hu_{\mathrm{nc}}\Vert_{L^2(\mathcal{T}\setminus \widehat{\mathcal{T}})}
						\le  C_{\mathrm{sr}} ^{m} |T|^{m/3}\Vert \lambda_hu_{\mathrm{nc}}\Vert_{L^2(\mathcal{T}\setminus \widehat{\mathcal{T}})}
						\le C_{\mathrm{sr}} ^{m} \eta(\mathcal{T}\setminus\widehat{\mathcal{T}})\le C_{\mathrm{sr}} ^{m} \eta(\mathcal{R}_1). 
\]
with 		$\mathcal{T}\setminus\widehat{\mathcal{T}}\subset \mathcal{R}_1$	in the last step.		\cref{rem:J2_wanted}   asserts 
\begin{align*}
							M_5^{-1}\Vert D^{m}_{\mathrm{pw}} (u_{\mathrm{nc}}-Ju_{\mathrm{nc}})\Vert_{L^2(\mathcal{T}\setminus\widehat{\mathcal{T}})}^2
								\le  \sum_{T\in\mathcal{R}_1}|T|^{1/3}\sum_{F\in\mathcal{F}(T)}
								\Vert [D^{m}_{\mathrm{pw}} u_{\mathrm{nc}}]_F\times \nu_F\Vert_{L^2(F)}^2\le \eta^2(\mathcal{R}_1).
\end{align*}
The combination of the last three displayed inequalities 
 concludes the proof of \ref{A3DiscreteReliabilty_new} with 
						$\Lambda_3:=C_6(C_{\mathrm{sr}}^{2m}+ M_5)$. 
				\end{proof}
				
\subsection{Quasiorthogonality}\label{sec:proofofA4}
The quasiorthogonality in \cref{thm:A44EVP} below concerns the outcome  $(\mathcal{T}_j)_{j\in\mathbb{N}_0}$ of 
				\namecref{alg:AFEM4EVP}. Let $u_j\in V(\mathcal{T}_j)$ 
				abbreviate the nonconforming component of the discrete solution  
				$\boldsymbol{u}_j=(u_{\mathrm{pw}}, u_\mathrm{nc})=:(u_{\mathrm{pw}}, u_j)\in P_m(\mathcal{T}_j)\times V(\mathcal{T}_j)$ 
				and $\lambda_{j}(k)\le\lambda$ the associated eigenvalue from \namecref{alg:AFEM4EVP} on the level $j\in\mathbb{N}_0$. 
				 Recall the distance 
				 $$\delta^2(\mathcal{T}_j,\mathcal{T}_{j+1})=\Vert \lambda_j(k)u_j-\lambda_{j+1}(k)u_{j+1}\Vert_{L^2(\Omega)}^2+\vvvert u_j-u_{j+1}\vvvert_{\mathrm{pw}}^2$$ 
				for the triangulations $\mathcal{T}_j$ and $\mathcal{T}_{j+1}$. 
				Set $h_0:=\max_{T\in\mathcal{T}_0}h_T$ and recall $\varepsilon_6>0$ from \cref{thm:reliabilty+efficiency}. 
				
\begin{theorem}[quasiorthogonality]\label{thm:A44EVP}
				For  any $0<\beta\le C_{\mathrm{eff}}^2/C_{\mathrm{rel}}^2$,  
				there exist $\Lambda_4$, $\widetilde{\Lambda}_4$, and 
				$\epsilon_4:=\widetilde{\Lambda}_4(\beta+h_0^{2\sigma}(1+\beta^{-1}))>0$, 
				such that $\mathcal{T}_0\in\mathbb{T}(\varepsilon_6)$ implies that 
				 the output $(\eta_j)_{j\in\mathbb{N}_0}$ and $(\mathcal{T}_j)_{j\in\mathbb{N}_0}$ of \namecref{alg:AFEM4EVP} satisfies  
				\begin{enumerate}[label={({A${\arabic*}$})}, leftmargin=3.5em]
						\item[\mylabel{A4epsilon_Quasiothogonality}{\textup{({A${4_{\varepsilon}}$})}}]Quasiorthogonality. 
							$\displaystyle
										\sum_{j=\ell}^{\ell+L}\delta^2(\mathcal{T}_j,\mathcal{T}_{j+1}) \leq \Lambda_{4}(1+\beta^{-1}) \eta^2_\ell
											+\epsilon_4 \sum_{j=\ell}^{\ell+L}\eta_j^2\quad\text{ for any }\ell,L\in\mathbb{N}_0
										.
							$
				\end{enumerate}
\end{theorem}
				
				The following 
				\cref{lem:towardsQuasiOptimality} in the $2$-level notation of \cref{def:2-level} prepares the proof of \cref{thm:A44EVP} below. 
				
\begin{lemma}[$2$-level quasiorthogonality]\label{lem:towardsQuasiOptimality}
					There exists  $C_{\mathrm{qo}}>0$ such that,  for $\mathcal{T}\in\mathbb{T}(\varepsilon_6)$, \\ 
					\hspace*{4mm} $\displaystyle
						{ a_{\mathrm{pw}}(u-\widehat{u}_{\mathrm{nc}}, u_{\mathrm{nc}}-\widehat{u}_{\mathrm{nc}})}
						\le C_{\mathrm{qo}}\big(h_{\max}^{\sigma }
									\vvvert u-{u}_{\mathrm{nc}}\vvvert_{\mathrm{pw}}
									+\Vert h_{\mathcal{T}}^{m}\lambda u\Vert_{L^2(\mathcal{T}\setminus\widehat{\mathcal{T}})}\big) 
									{\vvvert u-\widehat{u}_{\mathrm{nc}}\vvvert_{\mathrm{pw}}}
					$ \hfill holds.  
\end{lemma}
				
					\begin{proof}
						Since $(\lambda_h,u_{\mathrm{nc}})$ (resp.\ $(\widehat{\lambda}_h,\widehat{u}_{\mathrm{nc}})$) solves  
						\eqref{eq:disEVP} with respect to $\mathcal{T}$ and 
						${\delta}\in P_0({\mathcal{T}})$ (resp.\ $\widehat{\mathcal{T}}$ and 
						$\widehat{\delta}\in P_0(\widehat{\mathcal{T}})$ from \cref{tab:2level}),  \cref{cor:InterpolationOperator}.\ref{item:cor_I_a} 
						and elementary algebra show that 
						\begin{align}
							&a_{\mathrm{pw}}(u_{\mathrm{nc}}-\widehat{u}_{\mathrm{nc}},u-\widehat{u}_{\mathrm{nc}})
								=a_{\mathrm{pw}}\big(u_{\mathrm{nc}}, I(u-\widehat{u}_{\mathrm{nc}})\big)
								-a_{\mathrm{pw}}\big(\widehat{u}_{\mathrm{nc}},\widehat{I}u-\widehat{u}_{\mathrm{nc}}\big)\notag\\
								=& b\big(\lambda_hu_{\mathrm{nc}}(1+\delta), I(u-\widehat{u}_{\mathrm{nc}})\big)
								- b\big( \widehat{\lambda}_h\widehat{u}_{\mathrm{nc}}(1+\widehat{\delta}), 
										\widehat{I}u-\widehat{u}_{\mathrm{nc}}\big)\notag\\
								=& b\big(\lambda_hu_{\mathrm{nc}}(1+\delta)
								-\widehat{\lambda}_h\widehat{u}_{\mathrm{nc}}(1+\widehat{\delta}),\widehat{I}u-\widehat{u}_{\mathrm{nc}}\big)
										+\big(\lambda_hu_{\mathrm{nc}}, (I-\widehat{I})(u-\widehat{u}_{\mathrm{nc}})\big)_{1+\delta}.
										\label{eq:lem_towardsQuasiOptimality_sum}
						\end{align}
						The Cauchy-Schwarz inequality, $\lambda_h,\widehat{\lambda}_h\le\lambda$, and  \cref{lem:delta_control} in the last step prove 
						\begin{align*}
							t_1&:=b\big(\lambda_hu_{\mathrm{nc}}(1+\delta)
							-\widehat{\lambda}_h\widehat{u}_{\mathrm{nc}}(1+\widehat{\delta}),\widehat{I}u-\widehat{u}_{\mathrm{nc}}\big)\\
								&\le \Big(\Vert \lambda_hu_{\mathrm{nc}}-\widehat{\lambda}_h\widehat{u}_{\mathrm{nc}}\Vert_{L^2(\Omega)}
								+\lambda_h \Vert \delta u_{\mathrm{nc}}\Vert_{L^2(\Omega)}
								+\widehat{\lambda}_h\Vert \widehat{\delta} \widehat{u}_{\mathrm{nc}}\Vert_{L^2(\Omega)}\Big)
								\Vert\widehat{I}u-\widehat{u}_{\mathrm{nc}} \Vert_{L^2(\Omega)}\\
								&\le\max\{1,\lambda\} C_7 h_{\max}^{\sigma }
									\vvvert u-{u}_{\mathrm{nc}}\vvvert_{\mathrm{pw}}
									\Vert\widehat{I}u-\widehat{u}_{\mathrm{nc}} \Vert_{L^2(\Omega)}.
						\end{align*} 
			 			The discrete  Friedrichs inequality \eqref{eq:disFriedrich}  with respect to $V(\widehat{\mathcal{T}})$,   
						 \ref{item:I_Pi0}, and the $L^2$-projection $\widehat{\Pi}_0$ onto $P_0(\widehat{\mathcal{T}})$ lead to 
						$$C_{\mathrm{dF}}^{-1}\Vert \widehat{I}u-\widehat{u}_{\mathrm{nc}}\Vert_{L^2(\Omega)}\le
							\vvvert \widehat{I}u-\widehat{u}_{\mathrm{nc}}\vvvert_{\mathrm{pw}}
							=\Vert \widehat{\Pi}_0 D^{m}_{\mathrm{pw}} (u-\widehat{u}_{\mathrm{nc}})\Vert_{L^2(\Omega)}
							\le \vvvert u-\widehat{u}_{\mathrm{nc}}\vvvert_{\mathrm{pw}}. 
						$$
						Consequently, 
						$
							t_1\le \max\{1,\lambda\} C_7C_{\mathrm{dF}} h_{\max}^{\sigma }
									\vvvert u-{u}_{\mathrm{nc}}\vvvert_{\mathrm{pw}}	\vvvert u-\widehat{u}_{\mathrm{nc}}\vvvert_{\mathrm{pw}}.
						$ 
						Since $1+\delta\le 2$ from \cref{tab:2level}, the arguments behind \eqref{eq:I-hatI_CS} also show 
						\begin{align*}
							t_2:=&\big(\lambda_hu_{\mathrm{nc}}, (I-\widehat{I})(u-\widehat{u}_{\mathrm{nc}})\big)_{1+\delta}
							\le2
							\Vert h_{\mathcal{T}}^{m}\lambda_h u_{\mathrm{nc}}\Vert_{L^2(\mathcal{T}\setminus\widehat{\mathcal{T}})}
							\Vert h_{\mathcal{T}} ^{-{m}}(I-\widehat{I})(u-\widehat{u}_{\mathrm{nc}}) \Vert _{L^2(\Omega)}.
						\end{align*}
						Since \ref{item:I_identity} implies $I(\widehat{I}u)=Iu$, 
						 \ref{item:I_Pi0} and \ref{item:I_kappa_disc} for $I$ and \ref{item:I_Pi0} for $\widehat{I}$ show 
						$
							\Vert h_{\mathcal{T}} ^{-{m}}(I-\widehat{I})(u-\widehat{u}_{\mathrm{nc}}) \Vert _{L^2(\Omega)}
							=\Vert h_{\mathcal{T}} ^{-{m}}(1-I)(\widehat{I}u-\widehat{u}_{\mathrm{nc}}) \Vert _{L^2(\Omega)}
							\le \kappa_d \vvvert (1-I)\widehat{I}(u-\widehat{u}_{\mathrm{nc}})\vvvert_{\mathrm{pw}}							
							\le \kappa_d \vvvert u-\widehat{u}_{\mathrm{nc}}\vvvert_{\mathrm{pw}}.					
						$
						On the other hand, $\lambda_h\le\lambda$, a triangle inequality,  
						\eqref{eq:L2errorEnergyError}, and \cref{cor:InterpolationOperator}.\ref{item:cor_I_orthogonality} imply 
						\begin{align*}
							\Vert h_{\mathcal{T}}^{m}\lambda_h u_{\mathrm{nc}}\Vert_{L^2(\mathcal{T}\setminus\widehat{\mathcal{T}})}
							&\le \Vert h_{\mathcal{T}}^{m}\lambda u \Vert_{L^2(\mathcal{T}\setminus\widehat{\mathcal{T}})}
								+ \lambda h_{\max}^m \Vert u-u_{\mathrm{nc}}\Vert_{L^2(\Omega)}
								\\&\le \Vert h_{\mathcal{T}}^{m}\lambda u \Vert_{L^2(\mathcal{T}\setminus\widehat{\mathcal{T}})}
								+ C_5\lambda h_{\max}^{m+\sigma}  \vvvert u-u_{\mathrm{nc}}\vvvert_{\mathrm{pw}}. 
						\end{align*}
						Hence the upper bound $t_1+t_2$ in \eqref{eq:lem_towardsQuasiOptimality_sum} is controlled and the above estimates 
						lead to the assertion with 
						$C_{\mathrm{qo}}:=\max\{2\kappa_d, \max\{1,\lambda\}C_7C_{\mathrm{dF}}+2C_5\lambda h_{\max}^m\kappa_d\}$. 
					\end{proof}
					
					\begin{proofof}\textit{\cref{thm:A44EVP}.}
					Recall that $u_j\in V(\mathcal{T}_j)$  is the nonconforming component of the discrete solution 
					$\boldsymbol{u}_j=(u_{\mathrm{pw}}, u_\mathrm{nc})=:(u_{\mathrm{pw}}, u_j)\in P_m(\mathcal{T}_j)\times V(\mathcal{T}_j)$
				 	and that $\lambda_{j}(k)\le\lambda$ is the associated eigenvalue 	from \namecref{alg:AFEM4EVP} on the $j$-th level for $\ell\le j\le\ell+L$.
						 Since $\mathcal{T}_j,\, \mathcal{T}_{j+1}\in \mathbb{T}(\mathcal{T}_0)$ for $\ell\le j\le\ell+L$, \cref{lem:delta_control}  shows
						\begin{align*}
							\delta^2(\mathcal{T}_j, \mathcal{T}_{j+1})
							&\le  \vvvert u_j-u_{j+1}\vvvert_{\mathrm{pw}}^2+C_7^2C_{\mathrm{rel}}^2h_0^{2\sigma}\eta^2_j.
						\end{align*}
						Elementary algebra, \cref{lem:towardsQuasiOptimality}, and two weighted Young inequalities show 
						\begin{align*}
							\vvvert u_j-u_{j+1}\vvvert_{\mathrm{pw}}^2
							-&\vvvert u-u_j\vvvert_{\mathrm{pw}}^2+\vvvert u-u_{j+1}\vvvert_{\mathrm{pw}}^2
							=2 a_{\mathrm{pw}}(u-u_{j+1}, u_j-u_{j+1})\\
							\le& 2 C_{\mathrm{qo}} \Big(h_0^{\sigma}
											\vvvert u-{u}_j\vvvert_{\mathrm{pw}}
											+\Vert h_{\mathcal{T}_j}^{m}\lambda u\Vert_{L^2(\mathcal{T}_j\setminus{\mathcal{T}_{j+1}})}\Big)
											\vvvert u-{u}_{j+1}\vvvert_{\mathrm{pw}}
							\\ \le&  \frac{2 C_{\mathrm{qo}}^2}{\beta} h_0^{2\sigma} C_{\mathrm{rel}}^2\eta_j^2
							+\beta C_{\mathrm{rel}}^2 \eta_{j+1}^2
							 +\frac{2 C_{\mathrm{qo}}^2}{\beta}
							 \Vert h_{\mathcal{T}_j}^{m} \lambda u\Vert^2_{L^2(\mathcal{T}_j\setminus\mathcal{T}_{j+1})}
						\end{align*}					
						with 	\cref{thm:reliabilty+efficiency} in the last step. 
						\cref{thm:reliabilty+efficiency} controls the  telescoping sum 
						$$\sum_{j=\ell}^{\ell+L}\big(\vvvert u-u_{j}\vvvert_{\mathrm{pw}}^2-\vvvert u-u_{j+1}\vvvert_{\mathrm{pw}}^2\big)
						= \vvvert u-u_{\ell}\vvvert_{\mathrm{pw}}^2-\vvvert u-u_{\ell+L+1}\vvvert_{\mathrm{pw}}^2
						\le C_{\mathrm{rel}}^2\eta_{\ell}^2-C_{\mathrm{eff}}^2\eta_{\ell+L+1}^2.$$ 
						Since $\beta\le C_{\mathrm{eff}}^2/C_{\mathrm{rel}}^2$ implies 
						$(\beta C_{\mathrm{rel}}^2-C_{\mathrm{eff}}^2)\,\eta_{\ell+L+1}^2\le 0$, 
						the last three displayed estimates show 
						\begin{align}					
							\sum_{j=\ell}^{\ell+L} \delta^2(\mathcal{T}_j, \mathcal{T}_{j+1})
								\le&	
										\sum_{j=\ell}^{\ell+L}\big(\vvvert u-u_{j}\vvvert_{\mathrm{pw}}^2-
										\vvvert u-u_{j+1}\vvvert_{\mathrm{pw}}^2\big)
												+\Big(\Big(\frac{2C_{\mathrm{qo}}^2}{\beta}+C_7^2\Big)h_0^{2\sigma}+\beta\Big)
												C_{\mathrm{rel}}^2\sum_{k=\ell}^{\ell+L}\eta^2_j\notag
												\\&
												+\beta C_{\mathrm{rel}}^2\eta_{\ell+L+1}^2												
												+\frac{2C_{\mathrm{qo}}^2}{\beta}\sum_{j=\ell}^{\ell+L}
												\Vert h_{\mathcal{T}_j}^{m}\lambda u\Vert^2_{L^2(\mathcal{T}_j\setminus\mathcal{T}_{j+1})}
												\label{eq:A4_sum}\\
								\le& C_{\mathrm{rel}}^2\eta_{\ell}^2 
											+\Big(\Big(\frac{2C_{\mathrm{qo}}^2}{\beta}+C_7^2\Big)h_0^{2\sigma}+\beta\Big)C_{\mathrm{rel}}^2\sum_{k=\ell}^{\ell+L}\eta^2_j
												+\frac{2C_{\mathrm{qo}}^2}{\beta} \sum_{j=\ell}^{\ell+L}
												\Vert h_{\mathcal{T}_j}^{m}\lambda u\Vert^2_{L^2(\mathcal{T}_j\setminus\mathcal{T}_{j+1})}.\notag
						\end{align} 
						Recall that $h_{\mathcal{T}_j}|_T:=\textup{diam}(T)$ for any $T\in\mathcal{T}_j$  and 
						compare it with the piecewise constant function $\tilde{h}_j\in P_0(\mathcal{T}_j)$ defined by 
						$\tilde{h}_j|_T:=|T|^{1/3}\le h_T\le C_{\mathrm{sr}}|T|^{1/3}$ (from shape-regularity) for any $T\in\mathcal{T}_j$ and $j\in\mathbb{N}_0$. 
						Then  
						$\tilde{h}_j\approx h_{\mathcal{T}_j}\in P_0(\mathcal{T}_j)$ and $\tilde{h}_j\in P_0(\mathcal{T}_j)$ 
						satisfies the reduction $\tilde{h}_{j+1}\le \tilde{h}_{j}/2^{1/3}$ a.e. in 
						the set of refined tetrahedra $\bigcup\big(\mathcal{T}_j\setminus\mathcal{T}_{j+1}\big)$. 
						Hence  $\tilde{h}_j^{m}\le \frac{2^{m/3}}{\sqrt{4^{m/3}-1}}\,\sqrt{\tilde{h} _j^{2{m}}-\tilde{h}_{j+1}^{2{m}}}$ a.e. in 
						$\bigcup\big(\mathcal{T}_j\setminus\mathcal{T}_{j+1}\big)$ and 
\begin{align*}
							C_{\mathrm{sr}}^{-2m}&\frac{{4^{m/3}-1}}{4^{m/3}}
							\sum_{j=\ell}^{\ell+L}\Vert h_{\mathcal{T}_j}^{m}\lambda u\Vert^2_{L^2(\mathcal{T}_j\setminus\mathcal{T}_{j+1})}
							\le\frac{{4^{m/3}-1}}{4^{m/3}}\sum_{j=\ell}^{\ell+L}\Vert \tilde{h}_j^{m} \lambda u\Vert^2_{L^2(\mathcal{T}_j\setminus\mathcal{T}_{j+1})}\\
							&\le \sum_{j=\ell}^{\ell+L}\Big\Vert \sqrt{\tilde{h}_j^{2{m}}-\tilde{h}_{j+1}^{2{m}}}\lambda u\Big\Vert^2_{L^2(\Omega)}
							= \int_{\Omega}(\tilde{h}_\ell^{2{m}}-\tilde{h}_{\ell+L+1}^{2{m}})(\lambda u)^2\textup{d}x
							\le  \Vert \tilde{h}_\ell^{m} \lambda u\Vert_{L^2(\Omega)}^2.
\end{align*}
						Since $\tilde{h}_\ell\le h_{\mathcal{T}_\ell}\le h_0:=\max_{T\in\mathcal{T}_0}h_T\le\varepsilon_6$, a triangle inequality implies 
						\[
							\Vert\tilde{h} _\ell^{m}\lambda u\Vert_{L^2(\Omega)}^2
							\le  2(\lambda/\lambda_\ell(k))^{2}\Vert \tilde{h} _\ell^{m}\lambda_\ell(k) u_\ell\Vert_{L^2(\Omega)}^2
							+ 2\lambda^2 h_0^{2m}\Vert u-u_\ell\Vert_{L^2(\Omega)}^2. \]
 \cref{thm:BoundInterpolationError}.\ref{item:lambdahSimple}	and   \eqref{eq:def_eta} show  
						$(\lambda/\lambda_\ell(k))^{2}\Vert \tilde{h} _\ell^{m}\lambda_\ell(k) u_\ell\Vert_{L^2(\Omega)}^2\le  4\eta_\ell^2$. 
		 \cref{cor:InterpolationOperator}.\ref{item:cor_I_orthogonality}, \cref{thm:reliabilty+efficiency}, and \eqref{eq:L2errorEnergyError}  imply 
							$\Vert u-u_\ell\Vert_{L^2(\Omega)}^2
							\le h_{0}^{2\sigma} C_5^2C_\mathrm{rel}^2\eta_\ell^2.
						$
						The substitution in \eqref{eq:A4_sum} 
						concludes the proof  
						with 
$\Lambda_4:=\max\{C_{\mathrm{rel}}^2,C_{\mathrm{qo}}^2C_{\mathrm{sr}}^{2m}\frac{4^{m/3+1}}{{4^{m/3}-1}}(4+h_{0}^{2m+2\sigma}C_5^2C_{\mathrm{rel}}^2\lambda^2)\}$ and 
						$\widetilde{\Lambda}_4:=C_{\mathrm{rel}}^2\max\{1,2C_{\mathrm{qo}}^2,\allowbreak C_7^2\}$.
					\end{proofof}
					
\section{Conclusion and comments}\label{sec:optimality4evp}

\subsection{Proof of \cref{thm:Optimality4GLB}}
				The proven properties  \ref{A1Stability}--\ref{A4epsilon_Quasiothogonality} are the axioms of adaptivity in 
				\cite{CFPP14,CR16} and known to imply  \eqref{eq:Optimality}. 
				Compared to \cite{CFPP14,CR16}  the  discrete reliability in  \cref{thm:A34EVP} is extended in that 
				\ref{A3DiscreteReliabilty_new} includes the additional term $M_3 h_{\max}^{2\sigma}\eta^2(\mathcal{T})$. 
				Minor modifications of the arguments in \cite{CFPP14,CR16}  prove that \ref{A1Stability}--\ref{A4epsilon_Quasiothogonality}  
				imply  \eqref{eq:Optimality}. 	
				This is stated and proven as \cref{thm:Optimality} in \cref{sec:Appendix} 
				for  some $\varepsilon:=\varepsilon_7\le \varepsilon_6$.  
					\hfill $\Box$

\subsection{Optimal convergence rates of the error}
				The reliability and efficiency in  \cref{thm:reliabilty+efficiency} provide the equivalence 
				$\vvvert u-u_{\ell}\vvvert_{\mathrm{pw}}\approx \eta_\ell(\mathcal{T}_\ell)$. This and \cref{thm:Optimality4GLB}	
				lead to optimal convergence rates for the error as well. 
				
\subsection{Global convergence}	\label{sec:modifiedAlgo}
					This paper on the asymptotic convergence rates justifies that a small initial mesh-size 
					guarantees the asymptotic convergence from the beginning.  
					Although the reasons are presented in several steps for $\varepsilon_0, \dots,\varepsilon_7$, 
					the computation of $\varepsilon_7$ may be cumbersome and a huge overestimation in practice. 
					To guarantee global convergence without a priori knowledge of $\varepsilon_7$, 
					we may modify the marking step in \namecref{alg:AFEM4EVP} as follows: 
					Enlarge  the set $\mathcal{M}_\ell$ in \namecref{alg:AFEM4EVP} by one tetrahedron of maximal mesh-size in  $\mathcal{T}_\ell$.  
					This guarantees that the maximal mesh-size  tends to zero as the level $\ell\to\infty$. 
					Consequently there exists some $L\in\mathbb{N}$ such that $\mathcal{T}_\ell \in \mathbb{T}(\varepsilon_7)$ 
					for all $\ell=L,L+1,L+2,\dots$  
					Relabel $\mathcal{T}_L$ by $\mathcal{T}_0$ so that  \cref{thm:Optimality4GLB} 
					leads to optimal convergence rates for $\eta_L,\eta_{L+1},\eta_{L+2},\dots $, 
					whence  for the entire  outcome of the adaptive algorithm. 	
However, the constant in the overhead control \cite[Thm.~6.1]{Stev08} depends on $\mathcal{T}_L$ 
and this possibly enlarges the equivalence constants in \eqref{eq:Optimality}.
					
\subsection{Numerical experiments}
Numerical experiments in \cite{CP_Part1,CEP20} show an asymptotic convergences of \namecref{alg:AFEM4EVP} with $\theta=0.5$ even for coarse initial triangulation and  confirm the optimal convergence rates of \cref{thm:Optimality4GLB} 
even for one example with a multiple eigenvalue. 
The extension to eigenvalue clusters requires an algorithm from 
\cite{Gal15_cluster,DaiHeZhou2015,BGG17}.  This paper  assumes {\em exact solve} of the algebraic eigenvalue problem \eqref{eq:dis_EVP_alt}, 
but perturbation results in numerical linear algebra \cite{Par98} can be 
included as in \cite{CG12}.
\paragraph*{Acknowledgements}
This work has been supported by the Deutsche Forschungsgemeinschaft (DFG) in the Priority Program 1748 
{\em Reliable simulation techniques in solid mechanics. Development of non-standard discretization methods, mechanical and mathematical analysis} under 
CA 151/22-2. The second author is supported by the Berlin Mathematical School.
\FloatBarrier
\newpage
\appendix
\counterwithout{equation}{section}
\section{Appendix --  A review and extension of the axioms of adaptivity} \label{sec:Appendix}
	The framework \ref{A1Stability_app}--\ref{A4epsilon_Quasiothogonality_app} 
	in \cref{sec:aposteriori} is a modification of  \cite{CFPP14,CR16}  with a more general discrete reliability \ref{A3DiscreteReliabilty_app}.
	\cref{thm:Optimality} below proves that the modified axioms are sufficient for optimal convergence rates of the \namecref{alg:AFEM} algorithm 
	with  D\"orfler marking and newest-vertex bisection \cite[Algorithm 2.2]{CFPP14}. 
	On level $\ell\in\mathbb{N}_0$ of the general purpose adaptive algorithm \namecref{alg:AFEM} there is given a regular triangulation 
	$\mathcal{T}_\ell$ of $\Omega\subset\mathbb{R}^n$ into closed simplices and an undisplayed discrete problem with a discrete solution $u_\ell$. 
	These allow for the computation of $\eta_\ell(T)$ for all $T\in\mathcal{T}_\ell$ in the step compute. The step mark uses the sum convention 
	$\eta_\ell^2(\mathcal{M}):=\sum_{T\in\mathcal{M}}\eta_\ell^2(T)$ for any $\mathcal{M}\subseteq\mathcal{T}_\ell$   and 
	$\eta_\ell^2:=\eta^2_\ell(\mathcal{T}_\ell)$. The selection of a set $\mathcal{M}_\ell$ 
			with \emph{almost minimal cardinality} in this step means that  there exists a constant $\Lambda_{\mathrm{opt}}\ge 1$ such that 
			the cardinality satisfies $|\mathcal{M}_\ell|\le \Lambda_{\mathrm{opt}}|\mathcal{M}_\ell^\star|$,
			where ${\mathcal{M}}^\star_\ell\subset\mathcal{T}_\ell$  denotes some set of minimal cardinality $|{\mathcal{M}}^\star_\ell|$ 
			with $\theta\eta_\ell^2 \leq \sum_{T\in\mathcal{M}_\ell^\star}\eta_\ell^2(T)$; cf. \cite{Stevenson2007,CFPP14,CR16} for details; 
			this is more general than in \namecref{alg:AFEM4EVP}, which utilizes a minimal set $M_\ell$ with $\Lambda_{\mathrm{opt}}=1$ constructed at linear cost in  \cite{PfeilerPraetorius2020}. 
	\begin{center}		
		\begin{minipage}{.96\linewidth}
		{\begin{algorithm}[H]
		\begin{algorithmic}
			\Require regular initial triangulation $\mathcal{T}_0$ of $\Omega\subset\mathbb{R}^n$ and bulk parameter $0<\theta\leq 1$
			\For{$\ell = 0,1,2,\dots$}
			\State \textbf{Solve} the discrete problem for the discrete solution $u_\ell$ based on $\mathcal{T}_\ell$
			\State \textbf{Compute} $\eta_\ell(T)$ for any $T\in\mathcal{T}_\ell$ with respect to the discrete solution
			\State \textbf{Mark} almost minimal subset $\mathcal{M}_\ell\subseteq\mathcal{T}_\ell$ with  
						$\theta\eta_\ell^2 \leq \eta_\ell^2(\mathcal{M}_\ell)$ 
			\State \textbf{Refine} $\mathcal{T}_\ell$ with newest vertex bisection to 
		compute $\mathcal{T}_{\ell+1}$ with $\mathcal{M}_\ell\subseteq \mathcal{T}_\ell \setminus \mathcal{T}_{\ell+1}$\textbf{ od}
		 \EndFor
		\Ensure sequence of triangulations $(\mathcal{T}_\ell)_{\ell\in\mathbb{N}_0}$ with  $(u_\ell)_{\ell\in\mathbb{N}_0}$ 
		and $(\eta_\ell)_{\ell\in\mathbb{N}_0}$
		\end{algorithmic}
		\caption{}
		\label[AFEM]{alg:AFEM}
		\end{algorithm}}
		\end{minipage}
		\par
		\end{center}
		\vspace{0.9em}
		This appendix is written in a self-contained way based on the set $\mathbb{T}:=\mathbb{T}(\mathcal{T}_0)$  
		of all admissible triangulation computed by successive newest-vertex bisection \cite{Stev08,GSS14} 
		of a regular initial triangulation $\mathcal{T}_0$ (plus some initialization of tagged $n$-simplices)  
		of the bounded polyhedral Lipschitz domain  $\Omega\subset\mathbb{R}^n$ into closed simplices 
	 	and the subset $\mathbb{T}(\mathcal{T})$ of admissible refinements of $\mathcal{T}\in\mathbb{T}$. 
	 	For $N\in\mathbb{N}_0$, set $\mathbb{T}(N):=\{\mathcal{T}\in\mathbb{T}:\, |\mathcal{T}|\le |\mathcal{T}_0|+ N\}$.
	 	To analyse the error estimates $\eta_\ell(\mathcal{T}_\ell)$ and their rates and in particular to compare with 
	 	error estimators $\eta(\mathcal{T},\bullet)$ for \emph{any} 
	 	admissible triangulation $\mathcal{T}\in\mathbb{T}$, we need to assume that the error estimators 
	 	are computable for any $\mathcal{T}\in\mathbb{T}$. 
	 	This leads to a family $\eta(\mathcal{T},\bullet)\in\mathbb{R}^{\mathcal{T}}$ 
	 	of error estimators parametrized by $\mathcal{T}\in\mathbb{T}$ with $\eta(\mathcal{T},K)\ge 0$ for all $K\in\mathcal{T}$. 
	 	For any subset $\mathcal{M}\subseteq\mathcal{T}\in\mathbb{T}$,  the sum convention reads 
			\begin{align}
				\eta^2(\mathcal{T},\mathcal{M}):= \big(\eta(\mathcal{T},\mathcal{M})\big)^2:=\sum_{T\in\mathcal{M}}\eta^2(\mathcal{T},T) 
				\quad\text{ and }\quad \eta^2(\mathcal{T}):=\eta(\mathcal{T},\mathcal{T}).			
				\label{eq:def_eta_global}
			\end{align}
			For any triangulation $\mathcal{T}_\ell$ in the \namecref{alg:AFEM} algorithm, 
			we abbreviate $\eta_\ell(\bullet):=\eta(\mathcal{T}_\ell,\bullet)$ and 
			$\eta_\ell:=\eta_\ell(\mathcal{T}_\ell)\equiv \eta(\mathcal{T}_\ell,\mathcal{T}_\ell)$.  
			Recall the Axioms \ref{A1Stability_app}--\ref{A4epsilon_Quasiothogonality_app} with constants 
			$\Lambda_1$, $\Lambda_2$,   $\Lambda_3$,  $\widehat{\Lambda}_3$, $\Lambda_4$, $\Lambda_{\mathrm{ref}}$, 
			$\epsilon_3$,  $\epsilon_4 >0$, and $0<\rho_2<1$ for convenient reading. 
			For any $\mathcal{T}\in\mathbb{T}$ and admissible refinement $\widehat{\mathcal{T}}\in\mathbb{T}(\mathcal{T})$,  
			there exists a set $\mathcal{R}(\mathcal{T},\widehat{\mathcal{T}})\subseteq \mathcal{T}$ with 
			$\mathcal{T}\setminus\widehat{\mathcal{T}}\subset \mathcal{R}(\mathcal{T},\widehat{\mathcal{T}})$ and 
			$|\mathcal{R}(\mathcal{T},\widehat{\mathcal{T}})|\le \Lambda_{\mathrm{ref}}|\mathcal{T}\setminus\widehat{\mathcal{T}}|$, 
			such that  $\mathcal{T}\in\mathbb{T}$,  $\widehat{\mathcal{T}}\in\mathbb{T}(\mathcal{T})$, $\mathcal{R}(\mathcal{T},\widehat{\mathcal{T}})$,  
			and 
			the output $(\mathcal{T}_k)_{k\in\mathbb{N}_0}$ and $(\eta_k)_{k\in\mathbb{N}_0}$ of \namecref{alg:AFEM} satisfy  
			\ref{A1Stability_app}--\ref{A4epsilon_Quasiothogonality_app}. 
		\begin{enumerate}[label={\textup{({A${\arabic*}$})\ }},ref={\textup{({A${\arabic*}$})}}, leftmargin=3.5em]
			\item\label{A1Stability_app}Stability. 
						$\displaystyle
		 						\big\lvert{\eta(\mathcal{T},\mathcal{T}\cap \widehat{\mathcal T})
		 						-{\eta}( \widehat{\mathcal T},\mathcal{T}\cap\widehat{\mathcal T})}\big\rvert 
		 						\leq \Lambda_1 \delta(\mathcal{T},\widehat{\mathcal T}).
						$
			\item\label{A2Reduction_app}Reduction. 
				$\displaystyle
					\eta( \widehat{\mathcal T},\widehat{\mathcal T}\setminus\mathcal{T}) 
							\leq \rho_2 \eta(\mathcal{T},\mathcal{T}\setminus\widehat{\mathcal T}) 
							+ \Lambda_2 \delta(\mathcal{T},\widehat{\mathcal T}).
					$
			\item[\mylabel{A3DiscreteReliabilty_app}{\textup{({A${3_{\varepsilon}}$})}}]Discrete reliability. 
	 		$\displaystyle
	  		\delta^2(\mathcal{T},\widehat{\mathcal{T}}) \leq \Lambda_3 \eta^2(\mathcal{T},\mathcal{R}(\mathcal{T},\widehat{\mathcal{T}}))
	  		+ \widehat\Lambda_3 \eta^2(\widehat{\mathcal{T}})+\epsilon_3 \eta^2(\mathcal{T}).
	 		$
	 		\item[\mylabel{A4epsilon_Quasiothogonality_app}{\textup{({A${4_{\varepsilon}}$})}}]Quasiorthogonality. 
			$\displaystyle
						\sum_{j=\ell}^{\ell+m}\delta^2(\mathcal{T}_j,\mathcal{T}_{j+1}) \leq \Lambda_{4} \eta^2_\ell
							+\epsilon_4\sum_{j=\ell}^{\ell+m}\eta_j^2
						\text{ for any }\ell,m\in\mathbb{N}_0.
			$		
		\end{enumerate}
		\cref{thm:Optimality} below contains smallness assumptions for the constants $\widehat{\Lambda}_3,\,\epsilon_3,$ and $\epsilon_4$. 
		In a typical application such as \cref{thm:Optimality4GLB} the quantities $\widehat{\Lambda}_3,\,\epsilon_3,\,\epsilon_4$ contain a power of the initial mesh-size 
		$h_0:=\max_{T\in\mathcal{T}_0}h_T$ such that the assumptions are satisfied for a sufficiently fine initial triangulation $\mathcal{T}_0$. 
		Given $\epsilon_3<\Lambda_1^{-2}$, set $\Theta:={(1-\Lambda_1^2\epsilon_3)}/{(1+\Lambda_1^2  \Lambda_3)}$. Any choice of $\mu$ and $\xi$ with 
		$0<\mu<\rho_2^{-2}-1$ and $0<\xi<(1-(1+\mu)\rho_2^2)\Theta /(1-\Theta)$ implies 	
		$$\rho_{12}:=\Theta\rho_2^2(1+\mu)+(1-\Theta)(1+\xi)<1\quad\text{and}\quad\Lambda_{12}:=(1+1/\xi)\Lambda_1^2+(1+1/\mu)\Lambda_2^2<\infty.$$ 
		\begin{theorem}[rate optimality of the adaptive algorithm]\label{thm:Optimality}
			Suppose \ref{A1Stability_app}--\ref{A4epsilon_Quasiothogonality_app} 
			with $$\Lambda_1^2\epsilon_3<1,\quad  \widehat{\Lambda}_3(\Lambda_1^2+\Lambda_2^2)<1,\quad
			 \epsilon_4<(1-\rho_{12})/\Lambda_{12},\quad\text{and}\quad 0<\theta < \Theta.$$  
			The output $(\mathcal{T}_{\ell})_{\ell\in\mathbb{N}_0}$ and $(\eta_\ell)_{\ell\in\mathbb{N}_0}$ of   
			 \namecref{alg:AFEM}  satisfy,  for any $s>0$, the equivalence 
			 \begin{equation*}
		 		\sup_{\ell\in\mathbb{N}_0} (1+\abs{\mathcal{T}_\ell} - \abs{\mathcal{T}_0})^s \eta_\ell 
		 		\approx  \sup_{N\in\mathbb{N}_0} (1+N)^s \min_{\mathcal{T}\in\mathbb{T}(N)}\eta(\mathcal{T}).
			\end{equation*}
		\end{theorem}
	The proof of \cref{thm:Optimality} reviews parts of the analysis in 
	\cite{CFPP14,CR16} and focusses on the relevant extensions in 
	\cref{thm:Quasimonotonicity} and \cref{thm:ComparisonLemma} below. The following results \eqref{eq:A12}, \eqref{eq:A4}, and \eqref{eq:Convergence} 
	follow verbatim as in \cite{CFPP14,CR16}: 
	\ref{A1Stability_app}--\ref{A2Reduction_app} and the D\"orfler marking  strategy with bulk parameter 
	$\theta<\Theta<1$
	provide the estimator reduction \cite[Thm.~4.1]{CR16}
			\begin{align}
				\eta^2(\widehat{\mathcal{T}})\le \varrho_{12}\eta^2(\mathcal{T})+\Lambda_{12}\delta^2(\mathcal{T},\widehat{\mathcal{T}})
				\tag{A12} \label{eq:A12}
			\end{align}
	for any $\mathcal{T}\in \mathbb{T}$ and any admissible refinement $\widehat{\mathcal{T}}\in\mathbb{T}(\mathcal{T})$. 
	The estimator reduction \eqref{eq:A12}, \ref{A4epsilon_Quasiothogonality_app}, and  $\Lambda_{\mathrm{qo}}:=\Lambda_4+\epsilon_4(1+\Lambda_{12}\Lambda_4)/(1-\rho_{12}-\epsilon_4\Lambda_{12})>0$ 
	guarantee the stricter quasi-orthogonality \cite[Thm.~3.1]{CR16}
		\begin{align}
			 \sum_{k=\ell}^{\ell+m}\delta^2(\mathcal{T}_k,\mathcal{T}_{k+1}) 
			 \leq \Lambda_{\mathrm{qo}} \eta^2_\ell\quad\text{ for any }\ell,m\in\mathbb{N}_0.\tag{A4}\label{eq:A4} 
		\end{align}
	This and \eqref{eq:A12}  imply plain and $R$-linear convergence on each level 
	for the output $(\eta_\ell)_{\ell\in\mathbb{N}_0}$  of \namecref{alg:AFEM} 
	in \cite[Thm.~4.2]{CR16}:  The constants $\Lambda_c:=(1+\Lambda_{12}\Lambda_{\mathrm{qo}})/(1-\rho_{12})>0$ 
	and $q_c:=\Lambda_c/(1+\Lambda_c)<1$ satisfy 
		\begin{align}
			\sum_{k=\ell}^{\ell+m}\eta_k^2&\le \Lambda_c\eta_\ell^2
			\quad\text{and}\quad\eta_{\ell+m}^2\le \frac{q_c^m}{1-q_c}\eta_{\ell}^2\quad\text{for any }\ell, m\in\mathbb{N}_0. 
			\label{eq:Convergence}
		\end{align}
	On the other hand, \ref{A1Stability_app}--(A3) are sufficient for 
	the  quasimonotonicity (QM) and the 
	comparison lemma. 
	But the discrete reliability is relaxed in \ref{A3DiscreteReliabilty_app} in this paper,  so the proofs of (QM) and the 
	comparison lemma are revisited below.
		\begin{theorem}[QM]\label{thm:Quasimonotonicity}
			The axioms \ref{A1Stability_app}, \ref{A2Reduction_app}, \ref{A3DiscreteReliabilty_app}, 
			and $\widehat{\Lambda}_3(\Lambda_1^2+\Lambda_2^2)<1$ 
			imply the existence of $\Lambda_{\mathrm{mon}}>0$ such that $\eta(\widehat{\mathcal{T}})\le \Lambda_{\mathrm{mon}}\eta(\mathcal{T})$ 
			holds for any  $\mathcal{T}\in\mathbb{T}$ and  $\widehat{\mathcal{T}}\in\mathbb{T}(\mathcal{T})$.
		\end{theorem}
		\begin{proof}
			This proof extends \cite[Lem.~3.5]{CFPP14} and \cite[Thm.~3.2]{CR16}. 
			The axioms \ref{A1Stability_app}--\ref{A2Reduction_app} apply to the decomposition 
			$
				\eta^2(\widehat{\mathcal{T}})=\eta^2(\widehat{\mathcal{T}},\mathcal{T}\cap\widehat{\mathcal{T}})
				+\eta^2(\widehat{\mathcal{T}},\widehat{\mathcal{T}}\setminus\mathcal{T})
			$
			of the estimator of the fine triangulation $\widehat{\mathcal{T}}\in\mathbb{T}(\mathcal{T})$ and show  
			\begin{align*}
				\eta^2(\widehat{\mathcal{T}})
				&\le\big(\eta(\mathcal{T},\mathcal{T}\cap \widehat{\mathcal T})+ \Lambda_1 \delta(\mathcal{T},\widehat{\mathcal T})\big)^2
				+\big(\rho_2 \eta(\mathcal{T},\mathcal{T}\setminus\widehat{\mathcal T}) + \Lambda_2 \delta(\mathcal{T},\widehat{\mathcal T})\big)^2\\
				&\le (1+1/\alpha)\eta^2(\mathcal{T})+(1+\alpha)(\Lambda_1^2+\Lambda_2^2)\delta^2(\mathcal{T},\widehat{\mathcal T})
			\end{align*}
			with $(a+b)^2\le (1+\alpha)a^2+(1+1/\alpha)b^2$ for any positive $a,\,b$ and 
			$0<\alpha<\big((\Lambda_1^2+\Lambda_2^2)\widehat{\Lambda}_3\big)^{-1}-1$ in the second step. 
			(For $\widehat{\Lambda}_3=0$, the upper bound for $0<\alpha<\infty$ is understood  as infinity.) 
			The Axiom \ref{A3DiscreteReliabilty_app} controls the distance $\delta^2(\mathcal{T},\widehat{\mathcal{T}})$ and leads to 
			\begin{align*}
				\eta^2(\widehat{\mathcal{T}})
				\le \big(1+1/\alpha+(1+\alpha)(\Lambda_1^2+\Lambda_2^2)(\Lambda_3+\epsilon_3)\big)\eta^2(\mathcal{T})
				+(1+\alpha)(\Lambda_1^2+\Lambda_2^2) \widehat\Lambda_3 \eta^2(\widehat{\mathcal{T}}).
			\end{align*}
			Since $(1+\alpha)(\Lambda_1^2+\Lambda_2^2) \widehat\Lambda_3<1$, this proves 
			$\eta^2(\widehat{\mathcal{T}})\le \Lambda_{\mathrm{mon}}^2\eta^2(\mathcal{T})$ for 
			\begin{align*}
				\Lambda_{\mathrm{mon}}^2 :=\frac{ 1+1/\alpha+(1+\alpha)(\Lambda_1^2+\Lambda_2^2)(\Lambda_3+\epsilon_3)}
				{1-(1+\alpha)(\Lambda_1^2+\Lambda_2^2) \widehat\Lambda_3}.
		    \end{align*}	 
		\end{proof}
		The convergence is guaranteed with \eqref{eq:Convergence} and the optimality requires the sufficient smallness of the bulk parameter 
		$\theta<\Theta$ in the adaptive algorithm. This enters with the help of the comparison lemma, 
		where some $\theta_0(\varkappa,\alpha)$ depends on parameter $\varkappa,\alpha$ that allow for $\theta\le \theta_0(\varkappa,\alpha)<\Theta$.
		The lemma dates back to the seminal contribution \cite{Stevenson2007}. 
		\begin{lemma}[comparison]\label{thm:ComparisonLemma}
			Let  $0<\varkappa<1$, $0<\alpha<\infty$, and let $s>0$ satisfy 
				\begin{align*}
					M:=\sup_{N\in\mathbb{N}_0}(N+1)^s\min_{\mathcal{T}\in\mathbb{T}(N)}\eta(\mathcal{T})<\infty. 
				\end{align*}
			Then for any level $\ell \in\mathbb{N}_0$, there exist $\widehat{\mathcal{T}}_\ell\in \mathbb{T}(\mathcal{T}_\ell)$ and  
			$$\theta_0(\alpha,\varkappa):=\Big(1- \varkappa^2\big( (1+\alpha)+(1+1/\alpha) \Lambda_1^2\widehat\Lambda_3 \big)
			- (1+1/\alpha) \Lambda_1^2\epsilon_3 \Big)/\big(1+(1+1/\alpha) \Lambda_1^2  \Lambda_3 \big)<1$$
			 such that 
			\begin{enumerate}[label=(\alph*), ref=\alph*]
				\item\label{item:ComparisonLemma_a} $\eta(\widehat{\mathcal{T}}_\ell)\le
				\varkappa \eta(\mathcal{T}_\ell)\le \Lambda_{\mathrm{mon}}M|\mathcal{T}_\ell\setminus \widehat{\mathcal{T}}_\ell|^{-s}$ and  
				\item\label{item:ComparisonLemma_c} $\theta_0(\alpha,\varkappa)\eta^2(\mathcal{T}_\ell)\le \eta^2(\mathcal{T}_\ell,\mathcal{R}_\ell)$ 
									with $\mathcal{T}_\ell\setminus \widehat{\mathcal{T}}_\ell\subset\mathcal{R}_\ell
									:=\mathcal{R}(\mathcal{T}_\ell, \widehat{\mathcal{T}}_\ell)$ and 
									$|\mathcal{R}_\ell|\le \Lambda_{\mathrm{ref}} |\mathcal{T}_\ell\setminus \widehat{\mathcal{T}}_\ell|$.
			\end{enumerate}
		\end{lemma}
		\begin{proof}
			The proof of (\ref{item:ComparisonLemma_a}) is verbatim that of \cite[Prop.~4.12]{CFPP14} or that of  
			\cite[Lem.~4.3]{CR16} based on the overlay control (i.e., \eqref{eq:overhead} below) and Theorem~\ref{thm:Quasimonotonicity}. 
			It remains to modify the proofs  in  \cite[Prop.~4.12]{CFPP14} or \cite[Lem.~4.3]{CR16} for the verification of
			(\ref{item:ComparisonLemma_c}).
			Axiom \ref{A1Stability_app} and (\ref{item:ComparisonLemma_a}) imply that
			\begin{align}
				\eta(\mathcal{T}_{\ell},\mathcal{T}_\ell\cap \widehat{\mathcal{T}}_\ell)
					&\le \eta(\widehat{\mathcal{T}}_{\ell},\mathcal{T}_\ell\cap \widehat{\mathcal{T}}_\ell)
					+\Lambda_1\delta(\mathcal{T}_{\ell},\widehat{\mathcal T}_{\ell})
					\le \varkappa \eta({\mathcal{T}}_{\ell})+\Lambda_1\delta(\mathcal{T}_{\ell},\widehat{\mathcal T}_{\ell}). 		
					\label{eq:proof_comparison_1}
			\end{align}
			Recall $\eta^2_\ell(\mathcal{M}_\ell):=\eta^2(\mathcal{T}_\ell,\mathcal{M}_\ell):= \sum_{T\in\mathcal{M}_\ell}\eta^2(\mathcal{T}_\ell,T)$ for any 
			$\mathcal{M}_\ell\subset\mathcal{T}_\ell$ 
			and   $\eta_\ell:=\eta(\mathcal{T}_\ell)\equiv\eta(\mathcal{T}_\ell,\mathcal{T}_\ell)$ and abbreviate 
			$\widehat{\eta}_\ell:=\eta(\widehat{\mathcal{T}}_\ell)\equiv\eta(\widehat{\mathcal{T}}_\ell,\widehat{\mathcal{T}}_\ell)$. 
			A weighted Young inequality with 
			$\alpha>0$, the Axiom \ref{A3DiscreteReliabilty_app} 
			with $\mathcal{R}(\mathcal{T}_\ell,\widehat{\mathcal{T}}_\ell)$ replaced by $\mathcal{R}_\ell$ 
			defined in (\ref{item:ComparisonLemma_c}), and (\ref{item:ComparisonLemma_a}) show that 
			\begin{align}
					\big(\varkappa \eta_{\ell}+\Lambda_1\delta(\mathcal{T}_{\ell},\widehat{\mathcal T}_{\ell})\big)^2
					\le& (1+\alpha)\varkappa^2 \eta^2_{\ell}
					+(1+1/\alpha) \Lambda_1^2 \big( \Lambda_3 \eta^2_\ell(\mathcal{R}_\ell)
					+ \widehat\Lambda_3 \widehat{\eta}_\ell^2+\epsilon_3 \eta^2_{\ell}\big)
					\notag\\
					\le& (1+\alpha)\varkappa^2 \eta^2_{\ell}
					+(1+1/\alpha) \Lambda_1^2 \big( \Lambda_3 \eta^2_\ell(\mathcal{R}_\ell)
					+ \widehat\Lambda_3 \varkappa^2\eta^2_{\ell}+\epsilon_3 \eta^2_{\ell}\big).\label{eq:proof_comparison_2}
			\end{align}	
			Recall $\varkappa<1$,  
			$\alpha>0$, and set  
			$$C_a:=(1+\alpha)\varkappa^2+(1+1/\alpha) \Lambda_1^2(\epsilon_3 + \widehat\Lambda_3 \varkappa^2)
			\quad \text{and}\quad 
			C_b:=(1+1/\alpha) \Lambda_1^2  \Lambda_3.$$ 
			Then the combination of \eqref{eq:proof_comparison_1}--\eqref{eq:proof_comparison_2} reads 
			\begin{align}
				\eta^2_\ell(\mathcal{T}_\ell\cap \widehat{\mathcal{T}}_\ell)
					&\le C_a \eta^2_{\ell}+ C_b\eta^2_\ell(\mathcal{R}_\ell). \label{eq:ComparisonLemma_Ca_Cb}
			\end{align}
			Since $\mathcal{T}_\ell\setminus\widehat{\mathcal{T}}_\ell\subseteq\mathcal{R}_\ell$, the estimate \eqref{eq:ComparisonLemma_Ca_Cb} 
			implies 
			\begin{align*}
				\eta^2_{\ell}
					&\le \eta^2_\ell(\mathcal{R}_\ell)+\eta^2_\ell(\mathcal{T}_\ell\cap \widehat{\mathcal{T}}_\ell)
				    \le C_a \eta^2_{\ell}+(1+C_b)\eta^2_\ell(\mathcal{R}_\ell). 
			\end{align*}
			This proves (\ref{item:ComparisonLemma_c}) with 
			\begin{align*}
				\frac{1-C_a}{1+C_b}
				=\frac{1- \big( (1+\alpha)\varkappa^2+(1+1/\alpha) \Lambda_1^2(\epsilon_3 + \widehat\Lambda_3 \varkappa^2)\big) }
						{1+(1+1/\alpha) \Lambda_1^2  \Lambda_3 }=\theta_0(\varkappa,\alpha)<1.
			\end{align*}
		\end{proof}
		
		The proof of Theorem~\ref{thm:Optimality} can be concluded as in \cite[Proof of Theorem 4.1 (ii)]{CFPP14} or 
		\cite[Section 4.3]{CR16}.
		The function $\theta_0(\alpha,\varkappa)$ in \cref{thm:ComparisonLemma}.\ref{item:ComparisonLemma_c} is bounded from above by 
		$\lim_{\alpha\to\infty} \theta_0(0,\alpha)=({1-\Lambda_1^2\epsilon_3})/({1+\Lambda_1^2  \Lambda_3})$ and there exist a choice of 
		$0<\varkappa<1$ and $0<\alpha<\infty$
		such that $0<\theta<\theta_0(\alpha,\varkappa)< \Theta$. 
		This is the first formula on page~2655 in \cite{CR16} and the remaining parts of the proof are summarized below 
		for convenient reading and almost verbatim to Case~A in \cite{CR16}. 
		The choice of $\theta$ and \cref{thm:ComparisonLemma}.\ref{item:ComparisonLemma_c} show  
		$$\theta \eta^2(\mathcal{T}_\ell)\le \theta_0(\alpha,\varkappa)\eta^2(\mathcal{T}_\ell)\le \eta^2(\mathcal{T}_\ell,\mathcal{R}_\ell),$$ i.e., 
		$\mathcal{R}_\ell$ satisfies the D\"orfler marking condition. 
		Recall that $\mathcal{M}_\ell$ denotes the set of marked elements on level $\ell$ in \namecref{alg:AFEM}, 
		while ${\mathcal{M}}^\star_\ell$ with $|{\mathcal{M}}^\star_\ell|=M_\ell$ is a minimal set of marked elements. 
		Then  there exists $\Lambda_{\mathrm{opt}}\ge 1$ with 
		$|\mathcal{M}_\ell|\le \Lambda_{\mathrm{opt}}M_\ell\le \Lambda_{\mathrm{opt}}|\mathcal{R}_\ell|$. 
		The control over $\mathcal{R}_\ell:=\mathcal{R}(\mathcal{T}_\ell, \widehat{\mathcal{T}}_\ell)$ and 
		\cref{thm:ComparisonLemma}.\ref{item:ComparisonLemma_a} ensure 
		$$
		|\mathcal{R}_\ell|\le \Lambda_{\mathrm{ref}}|\mathcal{T}_\ell\setminus\widehat{\mathcal{T}}_\ell|\le \Lambda_{\mathrm{ref}}
		\big(\Lambda_{\mathrm{mon}}M/(\varkappa \eta_\ell)\big)^{1/s}.$$
		Hence $|\mathcal{M}_\ell|\le C_c M^{1/s}\eta_\ell^{-1/s}$ with 
		$C_c:=\Lambda_{\mathrm{opt}}\Lambda_{\mathrm{ref}}\Lambda_{\mathrm{mon}}^{1/s}\varkappa^{-1/s}$. 
		One important ingredient of NVB is the 		overhead control \cite{BDdV04,Stev08}
		\begin{align}
			|\mathcal{T}_\ell|-|\mathcal{T}_0|\le \Lambda_{\mathrm{BDdV}}\sum_{k=0}^{\ell-1}|\mathcal{M}_k| \label{eq:overhead}
		\end{align}
		with a universal constant $\Lambda_{\mathrm{BDdV}}$ that exclusively depends on $\mathcal{T}_0$. 
		The combination  of the above with the overhead control leads to 
		 \begin{align}
		 	|\mathcal{T}_\ell|-|\mathcal{T}_0|\le \Lambda_{\mathrm{BDdV}} C_c M^{1/s}\sum_{k=0}^{\ell-1}\eta_k^{-1/s}. \label{eq:toOptimality}
		 \end{align} 
		 The R-linear convergence \eqref{eq:Convergence}  bounds the sum $\sum_{k=0}^{\ell-1}\eta_k^{-1/s}$ as in \cite[Thm.~4.2.c]{CR16}. 
		 For all $0\le k<\ell$, the second identity in  \eqref{eq:Convergence} implies 
		 $\eta_{k}^{-1/s}\le \eta_{\ell}^{-1/s}{q_c^{(\ell-k)/(2s)}}{(1-q_c)^{-1/(2s)}}$. Hence the formula for the partial sum of 
		 the geometric series shows 
		 \begin{align}
		  \sum_{k=0}^{\ell-1}\eta_k^{-1/s}\le C_d \eta_\ell^{-1/s}\quad\text{with}\quad 
		  C_d:=\frac{q_c^{1/(2s)}}{\big(1-q_c^{1/(2s)}\big)(1-q_c)^{1/(2s)}}.\label{eq:Rsum}
		 \end{align}
		The combination of \eqref{eq:toOptimality}--\eqref{eq:Rsum} 
		reads $|\mathcal{T}_\ell|-|\mathcal{T}_0|\le \Lambda_{\mathrm{BDdV}} C_cC_d M^{1/s} \eta_\ell^{-1/s}$. 
		 Hence $1\le |\mathcal{T}_\ell|-|\mathcal{T}_0|$ implies 
		$(1+ |\mathcal{T}_\ell|-|\mathcal{T}_0|)\le 2(|\mathcal{T}_\ell|-|\mathcal{T}_0|)\le 2\Lambda_{\mathrm{BDdV}} C_cC_d M^{1/s}  \eta_\ell^{-1/s}$,  
		while $|\mathcal{T}_\ell|=|\mathcal{T}_0|$ implies $1\le  \Lambda_{\mathrm{BDdV}} C_cC_dM^{1/s}  \eta_\ell^{-1/s}$. 
		This  concludes the proof of  
		$$\eta_\ell(1+  |\mathcal{T}_\ell|-|\mathcal{T}_0|)^s\le (2\Lambda_{\mathrm{BDdV}} C_c C_d)^s  M\text{ with } 
		M:=\sup_{N\in\mathbb{N}_0}(N+1)^s\min_{\mathcal{T}\in\mathbb{T}(N)}\eta(\mathcal{T})$$ and so of   
		\ldq{$\lesssim$}\rdq in \cref{thm:Optimality}.
		\vspace{\bigskipamount}
		
		For the proof of the converse implication, 
		assume, without loss of generality, that $ 0<\min_{\mathcal{T}\in\mathbb{T}(N)}\eta(\mathcal{T})$ and so $0<\eta_\ell$ for any 
		$\ell\in\mathbb{N}_0$ with $N_\ell:=|\mathcal{T}_\ell|-|\mathcal{T}_0|\le N$. \namecref{alg:AFEM} leads to $N_\ell<N_{\ell+1}$ 
		(since no refinement only occurs for $\eta_\ell=0$). 
		Hence there exists a level $\ell$ with $N_\ell<N\le N_{\ell+1}$ and 
		$(N+1)^s \min_{\mathcal{T}\in\mathbb{T}(N)}\eta(\mathcal{T})\le (N_{\ell+1}+1)^s\eta_\ell$.
		On each refinement level $\ell$ each simplex creates at most a finite number $K(n)$ 
		(depending only on the spatial dimension $n$) of children in the next level $\ell+1$ \cite{GSS14}. In other words 
		$|\mathcal{T}_{\ell+1}|\le K(n)|\mathcal{T}_\ell|$ and 
		$(N_{\ell+1}+1)/(N_\ell+1)\le K(n) +(K(n)-1)(|\mathcal{T}_0|-1)\lesssim 1$. 
		This concludes the proof of rate optimality for \namecref{alg:AFEM} in  \cref{thm:Optimality}.  \hfill$\Box$
		
		\paragraph{Proof of \cref{thm:Optimality4GLB}.}\hspace{-.6cm}
		The \namecref{alg:AFEM4EVP} in \cref{thm:Optimality4GLB} is a particular case with 
		$\mathcal{R}(\mathcal{T},\widehat{\mathcal{T}})\hspace{-1mm}:=\mathcal{R}_1:=\{K\in\mathcal{T}:\,\exists\, T\in \mathcal{T}\setminus\widehat{\mathcal{T}} 
		 			\text{ with }\allowbreak \textup{dist}(K,T)=~0\}$. 
		\cref{thm:A1_A2}, \ref{thm:A34EVP}, and \ref{thm:A44EVP} guarantee \ref{A1Stability}--\ref{A4epsilon_Quasiothogonality} with 
		 $\widehat{\Lambda}_3:=0$, $\epsilon_3:= {M}_3{h}_{\max}^{2\sigma}$, 
		 and $\epsilon_4:=\widetilde{\Lambda}_4(\beta+h_0^{2\sigma}(1+1/\beta))>0$. Once $\rho_{12}$ and $\Lambda_{12}$ have been selected, 
abbreviate $c_3:=(1-\rho_{12})/(2\Lambda_{12}\widetilde{\Lambda}_4)$,
$\beta:= \min\{C_{\mathrm{eff}}^2/C_{\mathrm{rel}}^2,c_3/2\}$, and 
\begin{equation}\label{defconsteps7}
	\varepsilon:=\varepsilon_7:=\min\big\{\varepsilon_6,
					 (2\Lambda_1^2 M_3)^{-1/(2\sigma)}, ((c_3-\beta)/(1+1/\beta))^{1/(2\sigma)}\big\}.
		\end{equation}				  
Then $\widehat{\Lambda}_3(\Lambda_1^2+\Lambda_2^2)=0$, $\epsilon_3 \Lambda_1^2\le 1/2$, 
		and $\epsilon_4\le(1-\rho_{12})/(2\Lambda_{12})$ in \cref{thm:Optimality}. 

		\begin{remark}[smallness assumptions on $\varepsilon_5,\varepsilon_6,\varepsilon_7$]
			The reduction to $\varepsilon_5$ guarantees   the best approximation result in \cref{thm:BoundInterpolationError}, while
			 $\varepsilon_6:=\min\{\varepsilon_5,(2C_5^2)^{-1/(2\sigma)}\}$ is sufficient 
			for reliability in \cref{thm:reliabilty+efficiency}. Optimal rates follow  with  $\varepsilon:=\varepsilon_7$ from \eqref{defconsteps7}. 	
			Since  $C_5$ from  \eqref{eq:L2errorEnergyError},  
			$c_3:=(1-\rho_{12})/(2\Lambda_{12}\widetilde{\Lambda}_4)$, and $M_3$ are bounded $\mathcal{O}(1)$, 
			independent of the 
			mesh-size,  
			 $\varepsilon_6=\min\{\varepsilon_5,\mathcal{O}(1)\}$ and $\varepsilon_7=\min\{\varepsilon_6,\mathcal{O}(1)\}$ 
			are {\em not } expected to be  dramatically smaller than $\varepsilon_5$. 
		\end{remark} 
		
		\begin{remark}[modification with global convergence]
			\hspace{-0.7pt}The modified algorithm of \cref{sec:modifiedAlgo}, with $\mathcal{T}_L,\mathcal{T}_{L+1},\dots$ has 	
			no influence on the constants $1/2\le \Theta(1+\Lambda_1^2+\Lambda_3)\le 1$, 
			$\Lambda_4\le\Lambda_{\mathrm{qo}}\le 2\Lambda_4+1/\Lambda_{12}$,  
			$1+(\Lambda_1^2+\Lambda_2^2)\Lambda_3\le \Lambda_\mathrm{mon}^2
			\le \big(1+\sqrt{(\Lambda_1^2+\Lambda_2^2)(\Lambda_3+\Lambda_1^{-2}/2)}\big)^2 $.			
			But  $\Lambda_{\mathrm{BDdV}}$ in the overhead control \eqref{eq:overhead} (e.g. \cite[Thm.~6.1]{Stev08})
			depends on $\mathcal{T}_L$ and 
			could become larger (when replacing $\mathcal{T}_0$ by $\mathcal{T}_L$)  and leads to  larger equivalence constants in \cref{thm:Optimality}.
			Fortunately, the asymptotic convergence rate remains optimal and  
			 the choice of $\theta$ is not affected. 
		\end{remark}
		
		\begin{remark}[parameter choice in praxis]
			In a practical computation, we suggest uniform mesh-refinement until the eigenvalue $\lambda_k$ 
			of interest is resolved in that $5h_{\max}$ is smaller or equal the estimated wavelength of $\lambda_k$.  
			This triangulation serves as initial triangulation in $\mathcal{T}_0$ in the modified algorithm of \cref{sec:modifiedAlgo} with 
			some bulk parameter $\theta$  smaller than $(1-\Lambda_1^2\Lambda_3)^{-1}$. 
			In this way, the pre-asymptotic range is (hopefully) kept small while the asymptotic convergence rate remains optimal. 
		\end{remark}
		
\FloatBarrier		


\newpage\footnotesize{
\bibliographystyle{alpha}
\bibliography{BibGLB}
}

\end{document}